\documentclass[11pt,a4paper]{article}
\usepackage{amsfonts}
\usepackage{amsthm}
\usepackage{amsmath}
\usepackage{amssymb}
\usepackage{amscd}
\usepackage{mathrsfs}
\usepackage[mathscr]{eucal}
\usepackage{graphicx}
\usepackage{epstopdf} 
\usepackage{pict2e}
\usepackage{epic}
\usepackage{xcolor}
\usepackage{mathtools}
\usepackage{float}
\usepackage{centernot}

\usepackage{cite}
\usepackage{hyperref}
\hypersetup{colorlinks=true, urlcolor= black, linkcolor=black, citecolor=black}

\usepackage[utf8]{inputenc}
\usepackage[T1]{fontenc}
\usepackage{lmodern}
\usepackage{dsfont}
\usepackage{indentfirst}

\numberwithin{equation}{section}

\theoremstyle{plain}
\newtheorem{Thm}{Theorem}[section]
\newtheorem{Lem}[Thm]{Lemma}
\newtheorem{Cor}[Thm]{Corollary}
\newtheorem{Prop}[Thm]{Proposition}

\theoremstyle{definition}
\newtheorem{Def}[Thm]{Definition}

\newtheorem{Rem}[Thm]{Remark}

\newtheorem{Claim}[Thm]{Claim}

\usepackage[margin=3.1cm]{geometry}

\newcommand\bfC{\mathbf{C}}
\newcommand{\connect}{\xleftrightarrow}

\title{An alternative approach for the mean-field behaviour of spread-out Bernoulli percolation in dimensions $d>6$}
\begin{document}

\author{Hugo Duminil-Copin\footnote{Institut des Hautes \'Etudes Scientifiques, \url{duminil@ihes.fr}}\ \footnotemark[2]\footnote{Université de Genève, \url{hugo.duminil@unige.ch}} , Romain Panis\footnotemark[3]\footnote{Institut Camille Jordan, \url{panis@math.univ-lyon1.fr}}}
\maketitle
{\em We dedicate this article to Geoffrey Grimmett on the occasion of his seventieth birthday. Geoffrey's contributions to percolation theory have been profoundly influential for our generation of probabilists. His renowned manuscript was the bedside book for the authors during their PhD theses. 
}

\begin{abstract}
This article proposes a new way of deriving mean-field exponents for sufficiently spread-out Bernoulli percolation in dimensions $d>6$. We obtain up-to-constant estimates for the full-space and half-space two-point functions in the critical and near-critical regimes. In a companion paper, we apply a similar analysis to the study of the weakly self-avoiding walk model in dimensions $d>4$ \cite{DumPan24WSAW}.
\end{abstract}

\section{Introduction}

\subsection{Motivation}

Grasping the (near-)critical behaviour of lattice models is one of the key challenges in statistical mechanics. A possible approach involves determining the models' \emph{critical exponents}. Performing this task is typically very difficult as it involves the unique characteristics of the models and the geometry of the graphs on which they are constructed.

A significant observation was made for models defined on the hypercubic lattice $\mathbb{Z}^d$: beyond the \emph{upper-critical dimension} $d_c$, the influence of geometry disappears, and the critical exponents simplify, matching those found on a Cayley tree (or \emph{Bethe lattice}) or on the complete graph. The regime $d>d_c$ forms the \emph{mean-field} regime of a model. 

Prominent techniques such as the \emph{lace expansion} \cite{BrydgesSpencerSAW} and the \emph{rigourous renormalisation group method} \cite{BauerschmidtBrydgesSlade2014Phi4fourdim,BauerschmidtBrydgesSlade2015WSAW4D,BauerschmidtBrydgesSlade2015WSAW4DLogCorrections,BauerschmidtBrydgesSladeBOOKRG2019} have been developed to analyse the mean-field regime. These approaches have a predominantly \emph{perturbative} nature, which necessitates to identify  a small parameter within the model. It has been established, using lace expansion, that in several contexts \cite{HaraSlade1990Perco,Sakai2007LaceExpIsing,HaraDecayOfCorrelationsInVariousModels2008,Sakai2015Phi4,FitznervdHofstad2017Perco-d>10,Sakai2022correctboundsIsing} this small parameter can be taken to be $\tfrac{1}{d-d_c}$, meaning that mean-field behaviour was recovered in these setups in dimensions $d\gg d_c$.

 In the example of the nearest-neighbour (meaning that bonds are pairs of vertices separated by unit Euclidean distance) Bernoulli percolation, mean-field behaviour was established in dimensions $d>10$ \cite{HaraSlade1990Perco,HaraDecayOfCorrelationsInVariousModels2008,FitznervdHofstad2017Perco-d>10}. This leaves a gap to fill in order to reach the expected upper critical dimension of the model $d_c=6$. 
It is however possible to provide rigorous arguments to identify $d_c$ by introducing an additional perturbative parameter in the model. In the \emph{spread-out} Bernoulli percolation model, the bonds are pairs of vertices separated by distance between $1$ and $L$, where $L$ is taken to be sufficiently large. According to the \emph{universality} conjecture, the critical exponents of these two models should match. This makes spread-out Bernoulli percolation a natural testing ground to develop the analysis of the mean-field regime of Bernoulli percolation.

 Lace expansion was successfully applied to study various spread-out models in statistical mechanics, including Bernoulli percolation \cite{HaraSlade1990Perco,HaraSladevdHofstad2003PercoSO}, lattice trees and animals \cite{HaraSlade1990LatticeTrees}, the Ising model \cite{Sakai2007LaceExpIsing,Sakai2022correctboundsIsing}, and even some long-range versions of the aforementioned examples \cite{ChenSakaiLongRange2015,ChenSakaiLongRange2019}. Much more information on the lace expansion approach can be found in \cite{SladeSaintFlourLaceExpansion2006}.

In this paper, we provide an alternative argument to obtain mean-field bounds on the two-point function of sufficiently spread-out Bernoulli percolation in dimensions $d>6$. This technique extends to a number of other (spread-out) models after relevant modifications. In  a companion paper \cite{DumPan24WSAW}, we provide a treatment of the weakly self-avoiding walk model. However, the strategy developed there does not apply \emph{mutatis mutandis} to the setup of spread-out Bernoulli percolation, and the presence of long finite range interactions requires additional care.
\paragraph{Notations.} Consider the hypercubic lattice $\mathbb Z^d$.
Set  ${\bf e}_j$ to be the unit vector with $j$-th coordinate equal to 1. Write $x_j$ for the $j$-th coordinate of $x$ and denote its $\ell^\infty$ norm by $|x|:=\max\{|x_j|:1\le j\le d\}$. Set $\Lambda_n:=\{x\in \mathbb Z^d:|x|\le n\}$ and for $x\in \mathbb Z^d$, $\Lambda_n(x):=\Lambda_n+x$. Also, set $\mathbb H_n:=-n{\bf e}_1+\mathbb H$, where $\mathbb H:=\mathbb Z_+\times\mathbb Z^{d-1}=\{0,1,\ldots\}\times \mathbb Z^{d-1}$. For a set $S$ of vertices, introduce $S^*:=S\setminus\{0\}$. Let $\partial S$ be the boundary of the set $S$ given by the vertices in $S$ with one neighbour outside $S$. Finally, define the set of \emph{generalised blocks} of $\mathbb Z^d$, 
\begin{equation}
\mathcal B:=\Big\{\Big(\prod_{i=1}^d[a_i, b_i]\Big)\cap \mathbb Z^d
\text{ such that }-\infty\le a_i\le 0\le b_i\le\infty\text{ for }1\le i\le d\Big\}.
\end{equation}
If $\Lambda\subset \mathbb Z^d$ and $x,y\in \Lambda$, we denote by $\{x\connect{\Lambda\:}y\}$ the event that $x$ and $y$ are connected by a path of open edges in $\Lambda$.
If $A$ and $B$ are two percolation events, we write $A\circ B$ for the event of \emph{disjoint} occurrence of $A$ and $B$, i.e.~the event that there exist two disjoint sets $I$ and $J$ of edges such that the configuration restricted to $I$ (resp.~$J$) is sufficient to decide that $A$ (resp.~$B$) occurs. 

\subsection{Definitions and statement of the results}
 Let $L\geq 1$. Since $L$ is fixed for the whole article, we omit it from the notations. We consider the Bernoulli percolation measure $\mathbb P_\beta$ such that for every $u,v\in \mathbb Z^d$,
 \begin{equation}
 	p_{uv,\beta}:=\mathbb P_\beta[uv\text{ is open}]=1-\exp(-\beta J_{uv})=1-\mathbb P_\beta[uv\text{ is closed}],
 \end{equation} 
 where $J_{uv}=c_L\mathds{1}_{1\leq |u-v|\leq L}$, and $c_L$ is a normalization constant which guarantees that $|J|:=\sum_{x\in \mathbb Z^d}J_{0,x}=1$ (i.e. $c_L=|\Lambda_L^*|^{-1}$). We also let $p_\beta:=1-e^{-\beta c_L}$. Much more general choices can be made for $J$ (see e.g. \cite{HaraSlade1990Perco,vdHofstafdSlade2002generalisedinductive,HaraSladevdHofstad2003PercoSO}) but we restrict our attention to the above for simplicity.
 
  We are interested in the model's two-point function which is, for $\Lambda \subset \mathbb Z^d$, the probability
 $\mathbb P_\beta[x\connect{\Lambda\:}y]$.
When $\Lambda=\mathbb Z^d$, we simply write $\mathbb P_\beta[x\connect{}y]=\mathbb P_\beta[x\connect{\mathbb Z^d\:}y]$.
 It is well known that Bernoulli percolation undergoes a phase transition for the existence of an infinite cluster at some parameter $\beta_c\in (0,\infty)$. Moreover, for $\beta<\beta_c$, $\mathbb P_\beta[x\connect{}y]$ decays exponentially fast in $|x-y|$, see \cite{AizenmanNewmanTreeGraphInequalities1984,Mensikov1986coincidence,AizenmanBarsky1987sharpnessPerco,DuminilTassionNewProofSharpness2016}. It is convenient to measure the rate of exponential decay in terms of the the \emph{sharp length} $L_\beta$ defined below (see also \cite{DuminilTassionNewProofSharpness2016,PanisTriviality2023,DuminilPanis2024newLB} for a study of this quantity in the context of the Ising model). For $\beta\geq 0$ and $S\subset \mathbb Z^d$, let
\begin{equation}
	\varphi_\beta(S):=\sum_{\substack{y\in S\\ z\notin S}}\mathbb P_\beta[0\connect{S\:}y]p_{yz,\beta}.
\end{equation}
The sharp length $L_\beta$ is defined as follows:
\begin{equation}\label{eq: def L beta}
    L_\beta:=\inf\{k\ge1:\varphi_\beta(\Lambda_k)\le1/e^2\}.
\end{equation}
Exponential decay of the two-point function yields that $L_\beta$ is finite for $\beta<\beta_c$. It is infinite for $\beta=\beta_c$; see \cite{SimonInequalityIsing1980}.

We now turn to our first main result, which provides uniform upper bounds on the critical full-space and half-space two-point functions.

\begin{Thm}\label{thm:mainperco} Let $d>6$.
There exist $C,L_0>0$ such that for every $L\geq L_0$ and every $\beta\le\beta_c$ for which $L_\beta\geq L$,
\begin{align}
\label{eq:bound full plane perco}
\mathbb P_\beta[0\connect{}x]&\le \frac{C}{L^d}\left(\frac{L}{L\vee |x|}\right)^{d-2}\exp(-|x|/L_\beta) &\forall x\in (\mathbb Z^d)^*,\\
\label{eq:bound half plane perco} \mathbb P_\beta[0\connect{\mathbb H\:}x]&\le \frac{C}{L^d}\left(\frac{L}{L\vee |x_1|}\right)^{d-1}\exp(-|x_1|/L_\beta) &\forall x\in \mathbb H^*.
\end{align} 
\end{Thm}
Using the lace expansion, it is possible to obtain exact asymptotics of the critical full-space two-point function \cite{HaraSladevdHofstad2003PercoSO,LiuSlade2023Spreadout}: if $d>6$, there exist $D,L_0>0$ such that if $L\geq L_0$,
\begin{equation}\label{eq:intro asymptotic}
	\mathbb P_{\beta_c}[0\connect{}x]=\frac{D}{L^2|x|_2^{d-2}}(1+o(1)),
\end{equation}
where $|\cdot|_2$ is the Euclidean norm on $\mathbb R^d$ and $o(1)$ tends to $0$ as $|x|_2$ tends to infinity. However, the best uniform (in $x$) upper bound that this method yield is the following: if $d>6$, for every $\epsilon>0$, there exists $C>0$ such that for all $L$ large enough and $\beta\leq \beta_c$,
\begin{equation}
	\mathbb P_\beta[0\connect{}x]\leq \frac{C}{L^{2-\epsilon}|x|^{d-2}} \qquad \forall x\in(\mathbb Z^d)^*.
\end{equation}
Hence, Theorem \ref{thm:mainperco} is an improvement of the bound obtained using the lace expansion. Half-space up-to-constants critical estimates have been obtained in \cite{ChatterjeeHanson2020Halfspace} using the upper and lower bounds on $\mathbb P_{\beta_c}[0\connect{}x]$ induced by \eqref{eq:intro asymptotic}. 
 Near-critical estimates have also been derived in \cite{ChatterjeeHansonSosoe2023subcritical,HutchcroftMichtaSladePercolationTorusPlateau2023}. In these two papers, $L_\beta$ is replaced by $(\beta_c-\beta)^{-1/2}$, but we will in fact prove that $L_\beta\asymp (\beta_c-\beta)^{-1/2}$ (where $\asymp$ means that the ratio of the quantities is bounded away from $0$ and $\infty$ by two constants that are independent of $\beta$), see Corollary \ref{coro: perco susc corr length}.

Let us also mention that in the long-range setup, sharp estimates on the full-space two-point function have recently been derived using non-perturbative arguments, see \cite{hutchcroft2022sharp,hutchcroft2024pointwise}.

A direct consequence of Theorem \ref{thm:mainperco} is the finiteness at criticality of the so-called \emph{triangle diagram}, which plays a central role in the study of the mean-field regime of Bernoulli percolation, see \cite{AizenmanNewmanTreeGraphInequalities1984,HaraSlade1990Perco,BarskyAizenmanCriticalExponentPercoUnderTriangle1991,SladeSaintFlourLaceExpansion2006}. In particular, its finiteness implies that various critical exponents exist and take their mean-field value.
\begin{Cor}[Finiteness of the triangle diagram]\label{cor: finiteness triangle} Let $d>6$. There exists $L_0=L_0(d)>0$ such that for every $L\geq L_0$,
\begin{equation}
	\nabla(\beta_c):=\sum_{x,y\in \mathbb Z^d}\mathbb P_{\beta_c}[0\connect{}x]\mathbb P_{\beta_c}[x\connect{}y]\mathbb P_{\beta_c}[y\connect{}0]<\infty.
\end{equation}
\end{Cor}

The second main result of this paper provides lower bounds matching the upper bound of Theorem \ref{thm:mainperco}. We introduce, 
\begin{equation}
	\beta_0:=\inf\{\beta \geq 0 : \varphi_\beta(\{0\})=1\}.
\end{equation}
\begin{Thm}\label{thm:main2perco} Let $d>6$. There exist $c,L_0$ such that for every $L\geq L_0$, every $\beta_0\leq\beta\leq \beta_c$, and every $x\in \mathbb Z^d\setminus \{0\}$ with $|x|\leq cL_\beta$,
	\begin{align}
\label{eq:lowerbound full plane perco}
\mathbb P_\beta[0\connect{}x]&\ge \frac{c}{L^d}\left(\frac{L}{L\vee |x|}\right)^{d-2}, &\\
\label{eq:lowerbound half plane perco} \mathbb P_\beta[0\connect{\mathbb H\:}x]&\ge \frac{c}{L^d}\left(\frac{L}{L\vee |x_1|}\right)^{d-1}, &\text{provided that }x_1=|x|.
\end{align} 
\end{Thm}

We now describe how to recover the mean-field behaviour of the \emph{susceptibility} defined by
\begin{equation}
	\chi(\beta):=\sum_{x\in \mathbb Z^d}\mathbb P_\beta[0\connect{}x],
\end{equation}
and the \emph{correlation length} $\xi_\beta$ defined by
\begin{equation}\label{eq:def corr length paper perco}
	\xi_\beta^{-1}:=\lim_{n\rightarrow \infty}-\frac{1}{n}\log \mathbb P_\beta[0\connect{}n\mathbf{e}_1].
\end{equation}

\begin{Cor}\label{coro: perco susc corr length} Let $d>6$. There exist $c,C,L_0$ such that for every $L\geq L_0$ and every $\beta_0\leq \beta<\beta_c$ for which $L_\beta\geq L$,
\begin{align}
	c(\beta_c-\beta)^{-1}&\leq\chi(\beta)\leq C(\beta_c-\beta)^{-1},\label{eq:suscperco}
	\\cL(\beta_c-\beta)^{-1/2}&\leq L_\beta\leq CL(\beta_c-\beta)^{-1/2}.\label{eq:lbetaperco}
\end{align}
Moreover,
\begin{equation}\label{eq:xibetaperco}
	cL\frac{(\beta_c-\beta)^{-1/2}}{\log(1/(\beta_c-\beta))}\le \xi_\beta\le CL(\beta_c-\beta)^{-1/2}.
\end{equation}
\end{Cor}
\begin{proof} Let $d>6$. Let $L_0$ be given by Corollary \ref{cor: finiteness triangle}. It is classical (see \cite{AizenmanNewmanTreeGraphInequalities1984}) that $\nabla(\beta_c)<\infty$ implies the existence of $c,C>0$ such that, for $\beta<\beta_c$,
\begin{equation}
	c(\beta_c-\beta)^{-1}\leq \chi(\beta)\leq C(\beta_c-\beta)^{-1}.
\end{equation}
Combining Theorems \ref{thm:mainperco} and \ref{thm:main2perco}, we additionally obtain that $\chi(\beta)\asymp (L_\beta/L)^2$ when $\beta_0\leq \beta<\beta_c$ and $L_\beta\geq L$. This gives \eqref{eq:suscperco}--\eqref{eq:lbetaperco}.

It remains to observe to prove \eqref{eq:xibetaperco}. We use the same proof as in \cite[Theorem~1.6]{DuminilPanis2024newLB}: combining \eqref{eq:bound full plane perco} and \eqref{eq:def corr length paper perco} readily implies $\xi_\beta\leq L_\beta$ for $ \beta<\beta_c$; using \cite[Proposition~6.47]{GrimmettPercolation1999} implies that $\varphi_\beta(\Lambda_n)\leq C_1 n^{d-1}e^{-n/\xi_\beta}$ (where $C_1=C_1(d)>0$) for every $\beta<\beta_c$, and every $n\geq 1$. This gives that $L_\beta\leq C_2 \xi_\beta \log \xi_\beta$ (where $C_2=C_2(d)>0$) and the result.
\end{proof}

\subsection{Strategy of the proof of Theorem \ref{thm:mainperco}}\label{section: strategy}
Recall the van den Berg--Kesten (BK) inequality (see \cite[Section~2.3]{GrimmettPercolation1999}) stating that for two increasing events (i.e.~events that are stable by the action of opening edges) $A$ and $B$,
\begin{equation}\label{eq:BK ineq}
\mathbb P_\beta[A\circ B]\le \mathbb P_\beta[A]\mathbb P_\beta[B].\tag{BK}
\end{equation}

The strategy of the proof is similar to that in \cite{DumPan24WSAW}, though fundamental new difficulties arise due to the presence of non-nearest neighbour edges. The starting point is the following inequality, which is a classical consequence of the BK inequality. 

\begin{Lem}\label{Lem: SL upper bound} Let $d\geq 2$, $\beta>0$, $o\in S\subset \Lambda$, and $x\in \Lambda$. Then,
\begin{equation}\label{eq:SLperco}
	\mathbb P_{\beta}[o\connect{\Lambda\:}x]\le \mathbb P_{\beta}[o\connect{S\:}x]+ \sum_{\substack{y\in S\\ z\in \Lambda\setminus S}}\mathbb P_{\beta}[o\connect{S\:}y]p_{yz,\beta}\mathbb P_{\beta}[z\connect{\Lambda\:}x].
\end{equation}
\end{Lem}
A similar statement in the context of the Ising model is often referred to as the Simon--Lieb inequality, see \cite{SimonInequalityIsing1980,LiebImprovementSimonInequality}. For completeness, we include the proof immediately.
\begin{proof} Considering an open self-avoiding path from $o$ to $x$ gives
\begin{equation}
	\{o\connect{\Lambda\:}x\}\setminus \{o\connect{S\:}x\}\subset \bigcup_{\substack{y\in S\\ z\in \Lambda\setminus S}}\{o\connect{S\:}y\}\circ\{yz\textup{ is open}\}\circ \{z\connect{\Lambda\:}x\},
\end{equation}
which implies the result by applying \eqref{eq:BK ineq}.
\end{proof}

For $\beta>0$ and $k\geq 0$, define
\begin{equation}
	\psi_\beta(\mathbb H_k):=\sum_{x\in (\partial \mathbb H_k)^*}\mathbb P_{\beta}[0\connect{\mathbb H_k\:}x].
\end{equation}
We will use a bootstrap argument (following the original idea from \cite{Slade1987Diffusion}) and prove that an \emph{a priori} estimate on the half-space two-point function can be improved for sufficiently large $L$.
The idea will be to observe that the Simon--Lieb type inequality \eqref{eq:SLperco} can be coupled with a reversed version of the inequality--- see Lemma~\ref{Lem: SL lower bound} below--- to provide a good control on $\psi_\beta(\mathbb H_k)$, which can be interpreted as an $\ell^1$ estimate on the half-space two-point function at distance $k$. The point-wise, or $\ell^\infty$, half-space bound will follow from a \emph{regularity} estimate which allows to compare two-point functions for points that are close. Finally, we will deduce the full-space estimate from the half-space one (see Lemma \ref{lem: full plane 2pt function out of halfplane one perco}). The improvement of the a priori estimates will be permitted by classical random walk computations--- gathered in Appendix \ref{appendix:rw}--- involving the spread-out random walk defined as follows.
\begin{Def}\label{def: the random walk}\label{def: random walk J}
Define the random walk $(X_k^x)_{k\geq 0}$ started at $x\in \mathbb Z^d$ and of law $\mathbb P^{\rm RW}_{x}$ given by the step distribution:
\begin{equation}\label{eq:def random walk J}
	\mathbb P^{\rm RW}_{x}[X_1^x=y]:=J_{xy}.
\end{equation}
\end{Def}

To implement the scheme described in the previous paragraph, we introduce the following parameter $\beta^*$.

\begin{Def} Fix $d\geq 2$, $\bfC>1$, and $L\geq 1$. We define $\beta^*=\beta^*(\bfC,L)$ to be the largest real number in  $[0,2\wedge\beta_c]$ such that for every $\beta<\beta^*$, 
\begin{align}
\psi_\beta(\mathbb H_n)&< \frac{\bfC}{L} &\forall n\ge0,\label{eq:H_beta perco}\tag{$\ell^1_\beta$}
\\
 \mathbb P_{\beta}[0\connect{\mathbb H\:}x]&<\frac{\bfC}{L^d}\left(\frac{L}{L\vee |x_1|}\right)^{d-1} &\forall x\in\mathbb H^*.\label{eq:H_beta-' perco}\tag{$\ell^\infty_\beta$}
\end{align}
\end{Def}
It might be surprising to require $\beta^*\leq 2$. However, we will see that for $d>6$, one has $\beta_c= 1+O(L^{-d})$ as $L$ goes to infinity, so this bound is a posteriori harmless. Let us mention that the lace expansion allows to obtain an estimate of the form $\beta_c=1+C_1L^{-d}+O(L^{-(d+1)})$ for an explicit $C_1>0$, see \cite{vdHofstadSakai2005CriticalPoints}.

It could also be, a priori, that $\beta^*(\bfC,L)=0$. However, we will see in the lemma below that if $\bfC$ exceeds a large enough constant one may compare $\beta^*(\bfC,L)$ to $\beta_0$ defined by 
\begin{equation}\label{eq: def beta_0}
	\beta_0:=\min\{\beta\geq 0: \varphi_{\beta}(\{0\})=1\},
\end{equation}
and obtain that $\beta^*(\bfC,L)\ge1$.

\begin{Lem}\label{lem: intro lem about betastar} Let $d>2$. The following properties hold: 
\begin{enumerate}
	\item[$(i)$] for every $L\geq 1$, one has $1\leq \beta_0\leq 2$, and if $\beta \geq \beta_0$,
	\begin{equation}\label{eq: full volume times p}
	|\Lambda_L^*|p_\beta\geq 1;
\end{equation}
	\item[$(ii)$] for every $\bfC>0$, every $L\geq 1$, and every $\beta\leq \beta^*(\bfC,L)\vee \beta_0$,
	\begin{equation}\label{eq: bound volume times p}
	|\Lambda_L|p_{\beta}\leq 4;
\end{equation}
	\item[$(iii)$] there exists $\bfC_{\rm RW}\geq 64e^2$ such that, for every $\bfC>\mathbf C_{\rm RW}$, and every $L\geq 1$, $\beta^*(\bfC,L)\geq \beta_0$.
\end{enumerate}
 \end{Lem}

\begin{proof}Let us first prove $(i)$. Using that $|J|=1$ and the inequality $1-e^{-x}\leq x$ for every $x\geq 0$, we get
\begin{equation}
1=\varphi_{\beta_0}(\{0\})= \sum_{x\in \mathbb Z^d}(1-e^{-\beta_0J_{0x}})\leq \beta_0\sum_{x\in \mathbb Z^d}J_{0x}=\beta_0|J|=\beta_0.
\end{equation}
Moreover, one has that for $x\geq 0$, $1-e^{-x}\geq x-\tfrac{x^2}{2}$. Hence, 
\begin{equation}\varphi_{2}(\{0\})\geq |\Lambda_L^*|(2c_L-2c_L^2)=2-2c_L>1,
\end{equation} 
which implies that $\beta_0\leq 2$. The second part of the statement follows from the observation that, if $\beta\geq \beta_0$, $|\Lambda_L^*|p_\beta=\varphi_\beta(\{0\})\geq \varphi_{\beta_0}(\{0\})= 1$.

We now turn to the proof of $(ii)$. Note that by definition, ${\beta^*(\bfC,L)\leq 2}$. Hence, using $(i)$ we obtain that $\beta^*(\bfC,L)\vee \beta_0\leq 2$. Thus, one has $|\Lambda_L^*|p_\beta=c_L^{-1}p_\beta\leq \beta\leq 2$ when $\beta\leq \beta^*(\bfC,L)\vee \beta_0$. As a result,
\begin{equation}
	|\Lambda_L|p_\beta=\frac{|\Lambda_L|}{|\Lambda_L^*|}|\Lambda_L^*|p_\beta\leq  4,
\end{equation}
which concludes the proof.

Finally, we prove $(iii)$. Iterating \eqref{eq:SLperco} with $S$ a singleton and $\Lambda=\mathbb H$ gives for $x\neq 0$,
\begin{equation}\label{eq: iterated sl for beta_0}
		\mathbb P_{\beta_0}[0\connect{\mathbb H\:}x]\leq \mathbb E_0^{\rm RW}\Big[\sum_{\ell<\tau}\mathds{1}_{X_\ell=x}\Big],
\end{equation}
where $\tau$ is the exit time of $\mathbb H$, that is $\tau:=\inf\{\ell\geq 1: X_\ell\notin \mathbb H\}$ (and more generally, $\tau_k$ is the exit time of $\mathbb H_k$). 

Using a classical random walk estimate proved in Proposition \ref{prop: rw estimates}, we obtain the existence of $C_{\rm RW}=C_{\rm RW}(d)>0$ such that: for every $L\geq 1$ and every $x\in \mathbb H^*$,
\begin{equation}\label{eq: bound green function half space}
	\mathbb E_0^{\rm RW}\Big[\sum_{\ell<\tau}\mathds{1}_{X_\ell=x}\Big]\leq \frac{C_{\rm RW}}{L^d}\left(\frac{L}{L\vee |x_1|}\right)^{d-1}.
\end{equation}

The two previously displayed equations imply that 
\begin{equation}
\mathbb P_{\beta_0}[0\connect{\mathbb H\:}x]\leq\frac{C_{\rm RW}}{L^d}\left(\frac{L}{L\vee |x_1|}\right)^{d-1}.
\end{equation}
Now, let $k\geq 0$. Markov's property gives that
\begin{equation}
	\varphi^{\rm RW}(\mathbb H_k):=\sum_{\substack{y\in \mathbb H_k\\z\notin \mathbb H_k}}\mathbb E_0^{\rm RW}\Big[\sum_{\ell<\tau_k}\mathds{1}_{X_\ell=y}\Big]J_{yz}=\mathbb P_0^{\rm RW}[\tau_k<\infty]\leq 1.
\end{equation}
Since $p_{\beta,xy}\leq \beta J_{xy}$, we have by \eqref{eq: iterated sl for beta_0} that $\beta_0\varphi^{\rm RW}(\mathbb H_k)\geq \varphi_{\beta_0}(\mathbb H_k)$. It is possible to compare $\varphi_{\beta_0}(\mathbb H_k)$ to $\psi_{\beta_0}(\mathbb H_k)$ to obtain that $\varphi_{\beta_0}(\mathbb H_k)\tfrac{32e^2}{L}\geq \psi_{\beta_0}(\mathbb H_k)$. The proof of this technical result is postponed to Lemma \ref{lem: small values 1 perco}. Collecting all the above and using $(i)$, we obtain
\begin{equation}\label{eq: bound psi rw}
\psi_{\beta_0}(\mathbb H_k)\leq \beta_0\frac{32e^2}{L}\leq\frac{64e^2}{L}\leq\frac{\bfC_{\rm RW}}{L}	
\end{equation}
if we set $\bfC_{\rm RW}:=C_{\rm RW}\vee 64e^2$. Hence, from \eqref{eq: iterated sl for beta_0}--\eqref{eq: bound psi rw} and the monotonicity of the two-point function in $\beta$, we have obtained that for every $\bfC>\bfC_{\rm RW}$ and every $L\geq 1$, one has $\beta^*(\bfC,L)\geq \beta_0$.

\end{proof}

\paragraph{Organisation.} In Lemma \ref{lem: full plane 2pt function out of halfplane one perco}, we prove that we can estimate the full-space two-point function at $\beta$ under the hypothesis that $\beta\leq \beta^*$. Thus, our main goal is to show that $\beta^*$ is in fact equal to $\beta_c$ provided that $\bfC$ and $L$ are large enough. The proof goes in three steps:
\begin{itemize}
\item First, we show that we can improve the $\ell^1$ bound. We do so by relating $\psi_\beta(\mathbb H_{n})$ to $\varphi_\beta(\mathbb H_n)$ (see Lemma \ref{lem: small values 1 perco}) and then obtaining a bound on $\varphi_\beta(\mathbb H_n)$ when $\beta<\beta^*$ that becomes better and better when $L$ is chosen larger and larger. 
\item Second, we control the gradient of the two-point function (see Proposition \ref{prop:regularity perco}) and use this estimate to turn the improved $\ell^1$ bound into an improved $\ell^\infty$ bound.
\item Third, we show that these improvements imply that $\beta^*$ cannot be strictly smaller than $\beta_c$ (we will also derive that $\beta_c<2$), since otherwise the improved estimates would remain true for $\beta$ slightly larger than $\beta^*$, which would contradict the definition of $\beta^*$. 
\end{itemize}
\paragraph{Organisation of the paper.} The three steps are implemented in Section~\ref{sec:3}.  Before this, we describe the reversed Simon--Lieb type inequality in 
Section~\ref{sec:2}, and explain how one may bound an ``error'' term appearing in this reversed inequality. The proof of the lower bounds of Theorem \ref{thm:main2perco} builds on the results of Section \ref{sec:3}, and is performed in Section \ref{sec:lowerboundsperco}. Again, it follows the strategy used in \cite{DumPan24WSAW}, and relies on a key Harnack-type estimate.

\paragraph{Acknowledgements.} Early discussions with Vincent Tassion have been fundamental to the success of this project. We are tremendously thankful to him for these interactions. We warmly thank Gordon Slade for stimulating discussions and for useful comments. We also thank Émile Avérous, Dmitrii Krachun, and an anonymous referee for many useful comments.
 This project has received funding from the Swiss National Science Foundation, the NCCR SwissMAP, and the European Research Council (ERC) under the European Union’s Horizon 2020 research and innovation programme (grant agreement No. 757296). HDC acknowledges the support from the Simons collaboration on localization of waves. 

\section{A reversed Simon--Lieb type inequality}\label{sec:2}

\subsection{The inequality}

A crucial role will be played by the following inequality, which can be viewed as a \emph{reversed} Simon--Lieb inequality. 

\begin{Lem}\label{Lem: SL lower bound} Let $d\geq 2$, $0<\beta<\beta_c$, $o\in S\subset \Lambda$, and $x\in \Lambda$. Then,
\begin{equation}\label{eq:reversed SL perco}
\mathbb P_{\beta}[o\connect{\Lambda\:}x]\ge \mathbb P_{\beta}[o\connect{S\:}x]+\sum_{\substack{y\in S\\ z\in \Lambda\setminus S}}\mathbb P_{\beta}[o\connect{S\:}y]p_{yz,\beta} \mathbb P_{\beta}[z\connect{\Lambda\:}x]-E_\beta(S,\Lambda,o,x),
\end{equation}
where
\begin{align}\label{eq:error first term}
&E_\beta (S,\Lambda,o,x)
:=\sum_{u,v\in S}\sum_{\substack{y \in S\\ z\in\Lambda\setminus S}}\mathbb P_{\beta}[o\connect{S\:}u]\mathbb P_{\beta}[u\connect{S\:}y]p_{yz,\beta}\mathbb P_{\beta}[z\connect{\Lambda\:}v]\mathbb P_{\beta}[u\connect{S\:}v]\mathbb P_\beta[v\connect{\Lambda\:}x]
\\\label{eq:error second term}
&+\sum_{\substack{u\in S\\ v\in \Lambda}}\sum_{\substack{y,s\in S\\ z,t\in \Lambda\setminus S\\ yz\neq st}}\mathbb P_\beta[o\connect{S\:}u]\mathbb P_\beta[u\connect{S\:}y]\mathbb P_\beta[u\connect{S\:}s]p_{yz,\beta}p_{st,\beta}\mathbb P_\beta[z\connect{\Lambda\:}v]\mathbb P_\beta[t\connect{\Lambda\:}v]\mathbb P_\beta[v\connect{\Lambda\:}x].
\end{align}
\end{Lem}
\begin{figure}[htb]
	\begin{center}
	\includegraphics{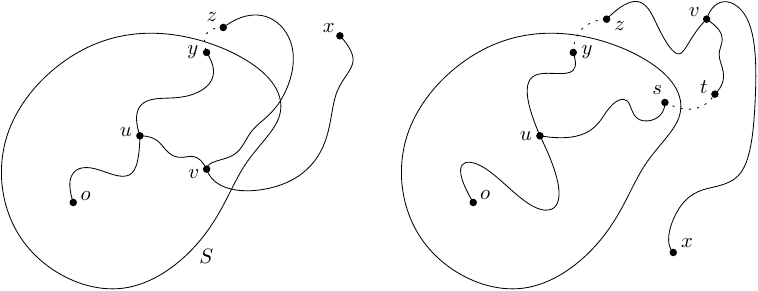}
	\caption{An illustration of the different terms contributing to $E_\beta(S,\Lambda,o,x)$. The dotted lines represent the open edges leaving $S$. From left to right, we illustrate \eqref{eq:error first term} and \eqref{eq:error second term}.}
	\label{fig: error}
	\end{center}
\end{figure}
Before diving into the proof of this lemma, let us recall an elementary yet important \emph{tree-bound inequality} from \cite[Proposition~4.1]{AizenmanNewmanTreeGraphInequalities1984}: for every set $S\subset \mathbb Z^d$, and every $o,a,b\in S$,
\begin{equation}\label{eq:tree bound}
\mathbb P[o\connect{S\:}a,b]\leq\sum_{u\in S}\mathbb P[o\connect{S\:}u]\mathbb P[u\connect{S\:}a]\mathbb P[u\connect{S\:}b].
\end{equation}
This bound directly follows from the BK inequality and the observation that for $o$ to be connected to both $a$ and $b$ in $S$, there must exist $u\in S$ such that the following event occurs
\begin{equation}
\{o\connect{S\:}u\}\circ\{u\connect{S\:}a\}\circ\{u\connect{S\:}b\}.
\end{equation}
\begin{proof} We fix $\beta$ and drop it from notation (e.g.~we write $p_{xy}$ instead of $p_{xy,\beta}$). Let
\begin{equation}
	\mathcal N:=\sum_{\substack{y\in S\\z\in \Lambda \setminus S}}\mathds{1}\{o\connect{S}y, \: yz \textup{ is open}, \: z \connect{\mathcal C^c\:}x\},
\end{equation}
where $\mathcal C$ is the (vertex) cluster of $o$ in $S$. Clearly,
$
	\{\mathcal N\geq 1\}= \{o\connect{\Lambda\:}x\}\setminus \{o\connect{S\:}x\}.
$

Since for $t\in \mathbb N$, $\mathds{1}_{t\geq 1}\ge 2t-t^2$, we find
\begin{equation}\label{eq:proofSL1}
\mathbb P[o\connect{\Lambda\:}x]-\mathbb P[o\connect{S\:}x]=\mathbb P[\mathcal N\geq 1]\geq 2\mathbb E[\mathcal N]-\mathbb E[\mathcal N^2]=\mathbb E[\mathcal N]-(\mathbb E[\mathcal N^2]-\mathbb E[\mathcal N]).
\end{equation}
\paragraph{Lower bound on $\mathbb E[\mathcal N]$.} Write
\begin{equation}\label{eq: sl proof first moment n}
	\mathbb E[\mathcal N]=\sum_{\substack{y\in S\\z\in \Lambda \setminus S}}\sum_{C \ni y}\mathbb P[\mathcal C= C]p_{yz}\mathbb P[z \connect{C^c\:}x].
\end{equation}
Using \cite[Proposition~5.2]{AizenmanNewmanTreeGraphInequalities1984}, we find that
\begin{equation}\label{eq:ah2}
	\mathbb P[z \connect{\Lambda\:}x]-\mathbb P[z\connect{C^c\:}x]\leq \sum_{v\in C}\mathbb P[\mathcal A(z,v)]\mathbb P[v\connect{\Lambda\:} x],
\end{equation}
where $\mathcal A(z,v)$ is the event that $z$ and $v$ are connected by a path whose only element in $C$ is the last vertex. Combining \eqref{eq: sl proof first moment n} and \eqref{eq:ah2} with the fact that $C\subset S$ yields
\begin{equation}
	\sum_{\substack{y\in S\\z\in \Lambda \setminus S}}\mathbb P[o\connect{S\:}y]p_{yz}\mathbb P[z \connect{\Lambda\:}x]-\mathbb E[\mathcal N]\leq \sum_{\substack{y\in S\\ z\in \Lambda\setminus S}}\sum_{v\in S}\mathbb P[\{o\connect{S\:}y, v\}\circ \{z \connect{\Lambda\:}v\}]p_{yz}\mathbb P[v\connect{\Lambda\:} x].
\end{equation}
Using \eqref{eq:BK ineq} and \eqref{eq:tree bound} yields,
\begin{align}
	\mathbb P[\{o\connect{S\:}y, v\}\circ \{z \connect{\Lambda\:}v\}]&\leq \mathbb P[o\connect{S\:}y,v]\mathbb P[z\connect{\Lambda\:}v]\notag\\&\leq \sum_{u\in S}\mathbb P[o\connect{S\:}u]\mathbb P[u\connect{S\:}v]\mathbb P[u\connect{S\:}y]\mathbb P[z\connect{\Lambda\:}v].
\end{align}
Altogether, we obtained,
\begin{align}
	\mathbb E[\mathcal N]&\geq \sum_{\substack{y\in S\\ z\in \Lambda\setminus S}}\mathbb P[o\connect{S\:}u]p_{yz} \mathbb P[z\connect{\Lambda\:}x]\notag
	\\&-\sum_{u,v\in S}\sum_{\substack{y\in S\\ z\in \Lambda\setminus S}}\mathbb P[o\connect{S\:}u]\mathbb P[u\connect{S\:}y]\mathbb P[u\connect{S\:}v] p_{yz}\mathbb P[z\connect{\Lambda\:}v]\mathbb P[v\connect{\Lambda\:}x] .\label{eq:proofSL2}
\end{align}
\paragraph{Upper bound on $\mathbb E[\mathcal N^2]-\mathbb E[\mathcal N]$.} Notice that,
\begin{equation}
	\mathbb E[\mathcal N^2]=\mathbb E[\mathcal N]+\sum_{\substack{y,s\in S\\ z,t\in \Lambda \setminus S\\ yz\neq st}}\mathbb P[o\connect{S}y, \: yz \textup{ is open}, \: z \connect{\mathcal C^c\:}x, o\connect{S\:}s, \: st \textup{ is open}, \: t \connect{\mathcal C^c\:}x].
\end{equation}
Using the BK inequality once again,
\begin{align}\notag
	\mathbb E[\mathcal N^2]-\mathbb E[\mathcal N]&\leq \sum_{\substack{y,s\in S\\ z,t\in \Lambda \setminus S\\ yz\neq st}}\mathbb P[o\connect{S\:}y,s]p_{yz}p_{st}\mathbb P[z,t\connect{\Lambda\:}x]\\
	&\leq \sum_{\substack{u\in S\\v\in \Lambda}}\sum_{\substack{y,s\in S\\ z,t\in \Lambda \setminus S\\ yz\neq st}}\mathbb P[o\connect{S\:}u]\mathbb P[u\connect{S\:}y]\mathbb P[u\connect{S\:}s]p_{yz}p_{st}\mathbb P[z\connect{\Lambda\:}v]\mathbb P[t\connect{\Lambda\:}v]\mathbb P[v\connect{\Lambda\:}x]\label{eq:proofSL6}
	\end{align}
	where in the second line we used \eqref{eq:tree bound} twice.
\paragraph{Conclusion.} The proof follows by combining \eqref{eq:proofSL1}, \eqref{eq:proofSL2}, and \eqref{eq:proofSL6}.	
\end{proof}
%


\subsection{The error amplitude: definition and bound}

For a set $S\ni 0$, introduce the {\em error amplitude} \begin{align}
	E_\beta(S)&:=\sum_{\substack{u,v,y\in S\\z\notin S}}\mathbb P_{\beta}[0\connect{S\:}u]\mathbb P_{\beta}[u\connect{S\:}y]p_{yz,\beta}\mathbb P_{\beta}[z\connect{\:}v]\mathbb P_{\beta}[u\connect{\:}v]\label{eq: definition E(o;u,v)}\\
&+\sum_{\substack{u,y,s\in S\\ v\in \mathbb Z^d\\ z,t\notin S\\ yz\neq st}}\mathbb P_\beta[0\connect{S\:}u]\mathbb P_\beta[u\connect{S\:}y]\mathbb P_\beta[u\connect{S\:}s]p_{yz,\beta}p_{st,\beta}\mathbb P_\beta[z\connect{\:}v]\mathbb P_\beta[t\connect{\:}v]. \notag
\end{align}
The main goal of this section is to prove the following bound, which provides a control of the error amplitude in terms of $\bfC$ and $L$ when $\beta<\beta^*(\bfC,L)$.
It is the place where the assumption that $d>6$ plays a fundamental role.
\begin{Prop}[Control of the error amplitude]\label{lem: error amplitude perco} Fix $d>6$, $\bfC>1$, and $L\geq 1$. There exists $K=K(\bfC,d)>1$ such that for every $\beta<\beta^*(\bfC,L)$ and every $B\in \mathcal B$,
\begin{equation}
	E_\beta(B)\leq \frac{K}{L^d}.
\end{equation}
\end{Prop}

As a warm-up, we derive some elementary computations obtained under the assumption that $\beta<\beta^*$. Our first estimate is a bound on the full-space two-point function.

\begin{Lem}[Bound on the full-space two-point function]\label{lem: full plane 2pt function out of halfplane one perco} Fix $d>2$, $\bfC>1$, and $L\geq 1$. For every $\beta<\beta^*(\bfC,L)$,
\begin{equation}\label{eq: full plane estimate from half plane perco}
\mathbb P_{\beta}[0\connect{}x]\le \frac{3\bfC^2}{L^d}\left(\frac{L}{L\vee|x|}\right)^{d-2} \qquad\forall x\in (\mathbb Z^d)^*.
\end{equation}
\end{Lem}

The proof consists in decomposing an open self-avoiding path from $0$ to $x$ with respect to the first vertex on it which minimises the first coordinate, and then use the half-space estimates provided by the assumption $\beta<\beta^*$.

\begin{proof} Without loss of generality, assume that $x_1=|x|\geq 1$.
If the connection to $x$ is not included in $\mathbb H$, decompose an open self-avoiding path connecting $0$ to $x$ according to the first left-most point $z$ on it (see Figure \ref{fig:perco1}) to get
\begin{align}
    \mathbb P_{\beta}[0\connect{}x]& \stackrel{\eqref{eq:BK ineq}}\le \mathbb P_{\beta}[0\connect{\mathbb H\:}x]+\sum_{n\ge 1}\sum_{\substack{z\in \partial\mathbb H_{n}}}\mathbb P_{\beta}[0\connect{\mathbb H_n\:}z] \mathbb P_{\beta}[z\connect{\mathbb H_n\:}x]\notag\\
    &\stackrel{\eqref{eq:H_beta-' perco}}\le \frac{\bfC}{L^d}\left(\frac{L}{|x|\vee L}\right)^{d-1}+ \sum_{n\ge 1}\psi_\beta(\mathbb H_n)\frac{\bfC}{L^d}\left(\frac{L}{L\vee (|x|+n)}\right)^{d-1}\notag
    \\&\stackrel{\phantom{\eqref{eq:H_beta-' perco}}}= \frac{\bfC}{L^d}\left(\frac{L}{|x|\vee L}\right)^{d-1}+\sum_{n=1}^{L-|x|}\psi_\beta(\mathbb H_n)\frac{\bfC}{L^d}+\sum_{n\geq (L-|x|+1)\vee 1}\psi_\beta(\mathbb H_n)\frac{\bfC}{L}\frac{1}{(|x|+n)^{d-1}}.\label{eq:belowbeta*implies full space 1}
\end{align}
If $|x|\leq L$, \eqref{eq:belowbeta*implies full space 1} and the assumption that $\beta< \beta^*(\bfC,L)$ imply that
\begin{equation}
	\mathbb P_\beta[0\connect{}x]\leq \frac{\bfC}{L^d}+\frac{\bfC^2}{L^d}+\frac{\bfC^2}{L^2}\frac{1}{(d-2)L^{d-2}}\leq \frac{3\bfC^2}{L^d},
\end{equation}
    where we used that $\sum_{n\geq \alpha}\tfrac{1}{(n+1)^{d-1}}\leq \tfrac{1}{(d-2)\alpha^{d-2}}$ (for $\alpha\geq 1$) with $\alpha=L$, and the hypothesis that $\bfC>1$. Now, when $|x|>L$, \eqref{eq:belowbeta*implies full space 1} and the preceding observation with $\alpha=|x|$ yield
    \begin{equation}
    	\mathbb P_\beta[0\connect{}x]\leq \frac{\bfC}{L|x|^{d-1}}+\frac{\bfC^2}{L^2}\sum_{k\geq |x|}\frac{1}{(k+1)^{d-1}}\leq \frac{\bfC}{L^2|x|^{d-2}}+\frac{\bfC^2}{(d-2)|x|^{d-2}}\leq \frac{2\bfC^2}{L^2|x|^{d-2}},
    \end{equation}
    where we used once more the assumption that $\beta< \beta^*(\bfC,L)$ and the fact that $\bfC>1$.
    The proof follows readily.	
\end{proof}

\begin{figure}[!htb]
	\begin{center}
		\includegraphics{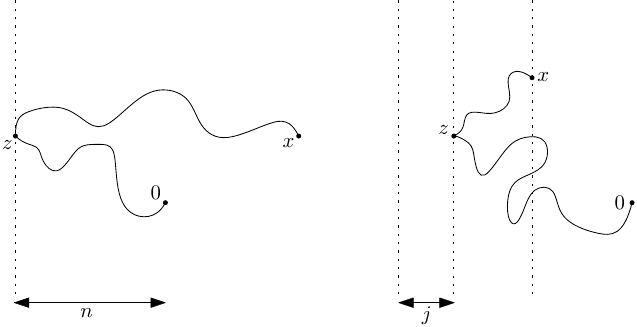}
		\put(-48,140){$\partial\mathbb H_{n-k}$}
        \put(-112,140){$\partial\mathbb H_{n}$}
        \put(-297,140){$\partial\mathbb H_{n}$}
		\caption{On the left, an illustration of the decomposition of a self-avoiding path connecting $0$ to $x$ used in the proof of \eqref{eq: full plane estimate from half plane perco}. On the right, a similar decomposition used in the proof of \eqref{eq: half plane at distance $k$ from half space perco} in the case $x_1<0$.}
		\label{fig:perco1}
	\end{center}
\end{figure}

For a generalised block $B\in \mathcal B$ and $k\geq 0$, we define the {\em $k$-boundary} of $B$ to be
\begin{equation}
	\partial^k B:=\{x\in B : \mathrm{d}(x,\partial B)=k\},
\end{equation}
where $\mathrm{d}(x,\partial B):=\inf\{|x-b|:b\in \partial B\}$. In particular, $\partial^0 B=\partial B$. 
We now turn to $\ell^\infty$ and $\ell^1$ estimates for points living on the $k$-boundary of a generalised block $B$.

\begin{Lem}[$\ell^\infty$ and $\ell^1$ bounds at distance $k$ from the boundary]\label{lem: distance k from half-space} Fix $d>2$, $\bfC>1$, and $L\geq 1$. Fix $B=(\prod_{i=1}^d[a_i,b_i])\cap \mathbb Z^d\in \mathcal B$, with $a_i\leq 0 \leq b_i$. For every $\beta<\beta^*(\bfC,L)$, 
\begin{align}    \label{eq: half plane at distance $k$ from half space perco}
    \sup_{\substack{x\in (\partial^k B)^*:\\ x_i\in\{a_i+k,b_i-k\}}} \mathbb P_{\beta}[0\connect{B\:}x]&\le \frac{\bfC^2(k+L)}{L^{d+1}}\left(\frac{L}{L\vee |x_i|}\right)^{d-1}  \quad& \forall 1\leq i\leq d,  \forall k\geq 0,    \\ \label{eq: sum over half plane at distance k from half plane perco}
    \sum_{x\in (\partial^k B)*}\mathbb P_{\beta}[0\connect{B\:}x]&\le \frac{2d\bfC^2(k+L)}{L^2}\quad&\forall k\geq 0.
\end{align}
\end{Lem}
The proof consists in a simple decomposition with respect to the first vertex on the open path connecting $0$ to $x$ which minimises the first coordinate.
\begin{proof} We prove each point separately.
\paragraph{Proof of \eqref{eq: half plane at distance $k$ from half space perco}.} Without loss of generality, assume that $i=1$ and that $x\in (\partial^k B)^*$ is such that $x_1=a_1+k$. Let $n=-a_1$ so that $x_1=k-n$. Notice that 
\begin{equation}
	\mathbb P_\beta[0\connect{B\:}x]\leq \mathbb P_\beta[0\connect{\mathbb H_n\:}x].
\end{equation}
We distinguish two cases: $x_1<0$ and $x_1\geq 0$. We begin with the former case, i.e. $k<n$. To bound $\mathbb P_{\beta}[0\connect{\mathbb H_n\:}x]$, decompose an open self-avoiding path connecting $0$ to $x$ according to its first left-most point $z$ (see Figure \ref{fig:perco1}) to get
  \begin{align}\label{eq: bound distance k 1}
  	\mathbb P_{\beta}[0\connect{\mathbb H_n\:}x]&\stackrel{\eqref{eq:BK ineq}}\leq \mathbb P_{\beta}[0\connect{\mathbb H_{n-k}\:}x]+\sum_{j=0}^{k-1}\sum_{\substack{z\in \partial \mathbb H_{n-j}}}\mathbb P_\beta[z\connect{\mathbb H_{n-j}\:}x]\mathbb P_\beta[z\connect{\mathbb H_{n-j}\:}0]
  	\\&\stackrel{\eqref{eq:H_beta-' perco}}\leq \frac{\bfC}{L^d}\left(\frac{L}{L\vee (n-k)}\right)^{d-1}+\sum_{j=0}^{k-1}\sum_{\substack{z\in \partial \mathbb H_{n-j}}}\frac{\bfC}{L^d}\left(\frac{L}{L\vee (n-j)}\right)^{d-1}\mathbb P_\beta[z\connect{\mathbb H_{n-j}\:}x].  	\label{eq:proofdistancekbounds 1}
  \end{align}
Using translation invariance and \eqref{eq:H_beta perco}, we obtain that 
  \begin{equation}
  	\sum_{z\in \partial \mathbb H_{n-j}}\mathbb P_\beta[z\connect{\mathbb H_{n-j}\:}x]=\psi_\beta(\mathbb H_{k-j})\le \frac{\bfC}{L},
  \end{equation}
  which implies that
  \begin{align}
    	\sum_{j=0}^{k-1}\sum_{\substack{z\in \partial \mathbb H_{n-j}}}\frac{\bfC}{L^d}\left(\frac{L}{L\vee (n-j)}\right)^{d-1}\mathbb P_\beta[z\connect{\mathbb H_{n-j}\:}x]  &\leq k\frac{\bfC^2}{L^{d+1}}\max_{0\leq j\leq k-1}\left(\frac{L}{L\vee (n-j)}\right)^{d-1}\notag\\&\leq
  	\frac{\bfC^2 k}{L^{d+1}}\left(\frac{L}{L\vee (n-k)}\right)^{d-1}.\label{eq:proofdistancekbounds 2}
  \end{align}
  Plugging the previously displayed equation in \eqref{eq:proofdistancekbounds 1} and using the assumption $\bfC>1$ yields
  \begin{equation}
  	\mathbb P_\beta[0\connect{\mathbb H_n\:}x]\leq \left(\frac{\bfC}{L^d}+\frac{\bfC^2 k}{L^{d+1}}\right)\left(\frac{L}{L\vee (n-k)}\right)^{d-1}\leq\frac{\bfC^2(k+L)}{L^{d+1}}\left(\frac{L}{L\vee |x_1|}\right)^{d-1}. 
  \end{equation}
  
  We turn to the case $x_1\geq 0$, i.e.\ $k\geq n$. Again, to bound $\mathbb P_\beta[0\connect{\mathbb H_n\:}x]$, we decompose an open self-avoiding path connecting $0$ to $x$ according to its first left-most point $z$. We obtain
\begin{align}\label{eq: bound distance k 1 bis}
  	\mathbb P_{\beta}[0\connect{\mathbb H_n\:}x]&\stackrel{\eqref{eq:BK ineq}}\leq \mathbb P_{\beta}[0\connect{\mathbb H\:}x]+\sum_{j=0}^{n-1}\sum_{\substack{z\in \partial \mathbb H_{n-j}}}\mathbb P_\beta[z\connect{\mathbb H_{n-j}\:}x]\mathbb P_\beta[z\connect{\mathbb H_{n-j}\:}0]
  	\\&\stackrel{\eqref{eq:H_beta-' perco}}\leq \frac{\bfC}{L^d}\left(\frac{L}{L\vee |x_1|}\right)^{d-1}+\sum_{j=0}^{n-1}\sum_{\substack{z\in \partial \mathbb H_{n-j}}}\frac{\bfC}{L^d}\left(\frac{L}{L\vee (|x_1|+n-j)}\right)^{d-1} 	\mathbb P_\beta[z\connect{\mathbb H_{n-j}\:}0]. \label{eq:proofdistancekbounds 3}
  \end{align}
 Reasoning as in \eqref{eq:proofdistancekbounds 2}, we get
  \begin{equation}\label{eq:proofdistancekbounds 4}
  	\sum_{j=0}^{n-1}\sum_{\substack{z\in \partial \mathbb H_{n-j}}}\frac{\bfC}{L^d}\left(\frac{L}{L\vee (|x_1|+n-j)}\right)^{d-1} 	\mathbb P_\beta[z\connect{\mathbb H_{n-j}\:}0]\leq \frac{\bfC^2 n}{L^{d+1}}\left(\frac{L}{L\vee |x_1|}\right)^{d-1}.
  \end{equation}
 Combining \eqref{eq:proofdistancekbounds 3} and \eqref{eq:proofdistancekbounds 4} gives
 \begin{equation}
 	\mathbb P_\beta[0\connect{\mathbb H_n\:}x]\leq \frac{\bfC^2(n+L)}{L^{d+1}}\left(\frac{L}{L\vee |x_1|}\right)^{d-1}\leq \frac{\bfC^2(k+L)}{L^{d+1}}\left(\frac{L}{L\vee |x_1|}\right)^{d-1},
 \end{equation}
 and concludes the proof of \eqref{eq: half plane at distance $k$ from half space perco}.

 \paragraph{Proof of  \eqref{eq: sum over half plane at distance k from half plane perco}.} This time, notice that
 \begin{align}\notag
 	\sum_{x\in (\partial^k B)^*}\mathbb P_\beta[0\connect{B\:}x]&\leq \sum_{i=1}^{d}\sum_{\substack{x\in (\partial^k B)^*\\x_i\in\{a_i+k,b_i-k\}}}\mathbb P_{\beta}[0\connect{B\:}x]\\&\leq (2d)\cdot \sup_{n\geq 0}\sum_{x\in (\partial \mathbb H_{n-k})^*}\mathbb P_\beta[0\connect{\mathbb H_n\:}x],\label{eq:bound sum distance k B proof}
 \end{align}
where we stress that $n-k$ can take negative values.
 We let $n\geq 0$ and bound the sum appearing in \eqref{eq:bound sum distance k B proof}. We consider the same decomposition as in \eqref{eq: bound distance k 1}, but use \eqref{eq:H_beta perco} twice instead of \eqref{eq:H_beta perco} and \eqref{eq:H_beta-' perco}. We first look at the case $0\leq k \leq n-1$. Summing \eqref{eq: bound distance k 1} (which holds for that range of $k$) over $x\in \partial \mathbb H_{n-k}$ gives 
  \begin{align}\notag
  	 \sum_{x\in \partial\mathbb H_{n-k}}&\mathbb P_\beta [0\connect{\mathbb H_n\:}x]\\&\stackrel{\phantom{\eqref{eq:H_beta perco}}}\leq \sum_{x\in \partial\mathbb H_{n-k}}\mathbb P_\beta[0\connect{\mathbb H_{n-k}\:}x]+\sum_{j=0}^{k-1}\sum_{z\in \partial \mathbb H_{n-j}}\mathbb P_\beta[0\connect{\mathbb H_{n-j}\:}z]\sum_{x\in \partial\mathbb H_{n-k}}\mathbb P_\beta[z\connect{\mathbb H_{n-j}\:}x]\notag
  	 \\&\stackrel{\phantom{\eqref{eq:H_beta perco}}}=\psi_\beta(\mathbb H_{n-k})+\sum_{j=0}^{k-1}\psi_\beta(\mathbb H_{n-j})\psi_\beta(\mathbb H_{k-j})\notag
  	 \\&\stackrel{\eqref{eq:H_beta perco}}\leq \frac{\bfC}{L}+\frac{k\bfC^2}{L^2}\leq \frac{\bfC^2(k+L)}{L^2},
 \end{align}
 where we used translation invariance to argue that for $z\in \mathbb H_{n-j}$, $\sum_{x\in \partial\mathbb H_{n-k}}\mathbb P_\beta[z\connect{\mathbb H_{n-j}\:}x]=\psi_\beta(\mathbb H_{k-j})$.
Now, if $k\geq n$, using \eqref{eq: bound distance k 1 bis} with $k\geq n$ and summing it over $x\in (\partial \mathbb H_{n-k})^*$ (here it is important to omit 0 in the case $k=n$) yields
 \begin{equation}
 	\sum_{x\in (\partial \mathbb H_{n-k})^*}\mathbb P_\beta[0\connect{\mathbb H_n\:}x]\leq \psi_\beta(\mathbb H)+\sum_{j=0}^{n-1}\psi_\beta(\mathbb H_{n-j})\psi_\beta(\mathbb H_{k-j})\stackrel{\eqref{eq:H_beta perco}}\leq  \frac{\bfC}{L}+\frac{n\bfC^2}{L^2}\leq \frac{\bfC^2(n+L)}{L^2}.
 \end{equation}
This concludes the proof.
\end{proof}

We are now equipped to prove the main result of this section on the error amplitude. This computation is not especially hard but it is quite lengthy. Since it simply consists in plugging the previous estimates in the expression for the error, it may be skipped in a first reading.
\begin{proof}[Proof of Proposition~\textup{\ref{lem: error amplitude perco}}] 
 We fix $B=(\prod_{i=1}^d[a_i,b_i])\cap \mathbb Z^d\in \mathcal B$ and assume that $ B\neq \mathbb Z^d$ (otherwise one has $E_\beta(\mathbb Z^d)=0$ by definition).
%
%
We fix $\beta<\beta^*$ and drop it from the notations. Introduce the quantities $\mathcal E(1)$ and $\mathcal E(2)$ corresponding to the sums in  \eqref{eq: definition E(o;u,v)} on the first and second line respectively. Since $	E_\beta(B)=\mathcal E(1)+\mathcal E(2)$, it suffices to bound these two terms separately. In this proof, the constants $C_i$ only depend on $\bfC$ and $d$.

\paragraph{Bound on $\mathcal E(1)$.} We take the position of $u$ into account and write
\begin{equation}\label{eq:boundamplitudeError0}
\mathcal E(1)=
 \sum_{\ell\geq 0}\sum_{u\in \partial^\ell B}\mathbb P[0\connect{B\:}u]
\sum_{\substack{y\in B\\ z\notin B}}\mathbb P[u\connect{B\:}y] p_{yz}\sum_{v\in B}\mathbb P[u\connect{\:}v]\mathbb P[z\connect{}v].
\end{equation}
Using \eqref{eq: full plane estimate from half plane perco} and Proposition \ref{prop: convolution estimate}, there exists $C_1>0$ such that for all $\ell\ge0$ and all $u$ and $z$ contributing to \eqref{eq:boundamplitudeError0},
\begin{equation}\label{eq:boundamplitudeError1}
	\sum_{v\in \mathbb Z^d}\mathbb P[u\connect{}v]\mathbb P[z\connect{}v]\leq \frac{1}{L^4}\frac{C_1}{(L\vee\ell)^{d-4}},
\end{equation}
where we used that $|z-u|\geq \ell+1$.

Let $u\in \partial^\ell B$. Recall from \eqref{eq: bound volume times p} that for $\beta<\beta^*$, one has $|\Lambda_L^*|p\leq 4$. Thus,
\begin{align}
	\sum_{\substack{y\in B\\ z\notin B}}\mathbb P[u\connect{B\:}y] p_{yz}&\stackrel{\phantom{\eqref{eq: sum over half plane at distance k from half plane perco}}}\leq |\Lambda_L^*|p\sum_{k=0}^{L-1}\sum_{y\in \partial^k B}\mathbb P[u\connect{B\:}y]\stackrel{\eqref{eq: sum over half plane at distance k from half plane perco}}\leq 4(4d\bfC^2+1),\label{eq:boundamplitudeError2}
\end{align}
where in the last inequality we took into account the possibility of having $y=u$.

Finally, using \eqref{eq: sum over half plane at distance k from half plane perco} one more time,
\begin{equation}\label{eq:boundamplitudeError3}
	\sum_{u\in \partial^\ell B}\mathbb P[0\connect{B\:}u]\leq \mathds{1}_{0\in \partial^\ell B}+\frac{2d\bfC^2(\ell+L)}{L^2}\leq \mathds{1}_{0\in \partial^\ell B}+\frac{4d\bfC^2(L\vee \ell)}{L^2}.
\end{equation}
Plugging \eqref{eq:boundamplitudeError1}--\eqref{eq:boundamplitudeError3} in \eqref{eq:boundamplitudeError0}, we found $C_2,C_3>0$ such that,
\begin{equation}\label{eq: bound on E(1)}
	\mathcal E(1)\leq \frac{C_2}{L^d}+\frac{C_2}{L^6}\sum_{\ell\geq 0}\frac{1}{(L\vee \ell)^{d-5}}\leq \frac{C_3}{L^d},
\end{equation}
where we used that $d>6$ to justify that the sum is finite. 
\paragraph{Bound on $\mathcal E(2)$.} Again, we take the position of $u$ into account,
\begin{equation}\label{eq:boundamplitudeError3.5}
	\mathcal E(2)=\sum_{\ell\geq 0} \sum_{\substack{u\in \partial^\ell B}}\mathbb P[0\connect{B\:}u]\sum_{\substack{y,s\in B\\ z,t  \notin B\\ yz\neq st}}\mathbb P[u\connect{B\:}y]\mathbb P[u\connect{B\:}s]p_{yz}p_{st}\sum_{v\in \mathbb Z^d}\mathbb P[z\connect{\:}v]\mathbb P[t\connect{\:}v].
\end{equation}
A diagrammatic representation of $\mathcal E(2)$ is given in Figure \ref{fig: error} (once we remove the path from $v$ to $x$ in the right-most picture). From this figure we can interpret $\mathcal E(2)$ as a bound on the probability that $0$ connects to some $v\in \mathbb Z^d$ through two paths which exit $B$ and which use different ``outgoing'' open edges $yz$ and $st$. Heuristically, the requirement of finding more than one such edges is responsible for an additional cost $O(L^{-d})$, which justifies that $\mathcal E(2)$ should be small. We now turn to a proper justification of this fact.

Using \eqref{eq: full plane estimate from half plane perco} and Proposition \ref{prop: convolution estimate}, there exists $C_4>0$ such that for all $z,t\in \mathbb Z^d$,
 \begin{equation}\label{eq:boundamplitudeError4}
 	\sum_{v\in \mathbb Z^d}\mathbb P[z\connect{}v]\mathbb P[t\connect{}v]\leq C_4\Big(\mathds{1}_{z=t}+\mathds{1}_{z\ne t}\frac{1}{L^4}\frac{1}{(L\vee|z-t|)^{d-4}}\Big).
 \end{equation}
 We plug \eqref{eq:boundamplitudeError4} in \eqref{eq:boundamplitudeError3.5} and consider the two contributions coming from $z=t$ and $z\neq t$.
 
\paragraph{Case $z\ne t$.} We begin with the contribution of $z\neq t$ and further split the analysis into two subcases: $|z-t|\geq \ell+1$ and $|z-t|\leq \ell$. 

We begin with the subcase $|z-t|\geq \ell+1$. Notice that
 \begin{align}
	\frac{1}{L^4}\sum_{\ell\geq 0}\sum_{u\in \partial^\ell B}\mathbb P[0 \connect{B\:}u]&\sum_{\substack{y,s\in B\\ z,t\notin B\\ |z-t|\geq \ell+1}}\mathbb P[u \connect{B\:}y]\mathbb P[u\connect{B\:}s]\frac{p_{yz}p_{st}}{(L\vee|z-t|)^{d-4}}\notag
	\\&\stackrel{\eqref{eq:boundamplitudeError2}}\leq \frac{1}{L^4}\sum_{\ell\geq 0}\frac{[2(4d\bfC^2+1)]^2}{(L\vee\ell)^{d-4}}\sum_{u\in \partial^\ell B}\mathbb P[0\connect{B\:}u]\notag
	\\&\stackrel{\eqref{eq:boundamplitudeError3}}\leq \frac{C_5}{L^d}+\frac{C_5}{L^6}\sum_{\ell\geq 0}\frac{1}{(L\vee \ell)^{d-5}}\leq \frac{C_6}{L^d},\label{eq:boundamplitudeError6}
\end{align}
where $C_5,C_6>0$, and where we used again that $d>6$ to justify the finiteness of the sum in  \eqref{eq:boundamplitudeError6}. 

We now look at the subcase $|z-t|\leq \ell$.  
Let $u\in \partial^\ell B$ and $s\in \cup_{k=0}^{L-1}\partial^k B$ with $u\neq s$. If $\ell \leq L$, then Lemma \ref{lem: full plane 2pt function out of halfplane one perco} implies that
\begin{equation}
	\mathbb P[u\connect{B\:}s]\leq \mathbb P[u\connect{}s]\leq \frac{3\bfC^2}{L^d}.
\end{equation}
If now $\ell>L$ and $s\in \partial^kB$ (with $k\leq L-1$), \eqref{eq: half plane at distance $k$ from half space perco} with $k$ and applied to the box $B-u$ gives
\begin{equation}
	\mathbb P[u\connect{B\:}s]=\mathbb P[0\connect{B-u\:}s-u]\leq \frac{\bfC^2(k+L)}{L^{d+1}}\left(\frac{L}{(\ell-L)\vee L}\right)^{d-1}\leq \frac{2\bfC^2}{L^d}\left(\frac{L}{(\ell-L)\vee L}\right)^{d-1}.
\end{equation}
As a result, for every $\ell \geq 0$, every $u \in \partial^\ell B$, and every $s\in \cup_{k=0}^{L-1}\partial^k B$,
\begin{equation}\label{eq:boundamplitudeError7}
	\mathbb P[u\connect{B\:}s]\leq \frac{3\bfC^2}{L^d}\left(\frac{L}{(\ell-L)\vee L}\right)^{d-1}.
\end{equation}

Moreover, there exists $C_7>0$ such that for fixed $y,z$ as above,
\begin{equation}\label{eq:boundamplitudeError8}
	\sum_{\substack{ t\in \Lambda_\ell(z)\setminus \{z\}\cap B^c\\s\in B}}p_{st}\frac{1}{(L\vee|z-t|)^{d-4}}\leq C_7 L\ell^3.
\end{equation}
Indeed, note that in the above sum, one has $s \in \cup_{k=0}^{L-1}\partial^k B$ and $|s-t|\leq L$ (otherwise $p_{st}=0$). As a result, for each such $s$,
\begin{equation}
	\{t: p_{(s-z)t}\neq 0\}\cap\Lambda_\ell^*\cap (B-z)^c\subset \{t'\in \Lambda_\ell: \exists 1\leq i \leq d, \: |t'_i|\leq L\}.
\end{equation}
Hence, using again that $|\Lambda_L^*|p\leq 4$,
\begin{align}	\sum_{\substack{ t\in \Lambda_\ell(z)\setminus \{z\}\cap B^c\\s\in B}}p_{st}\frac{1}{(L\vee|z-t|)^{d-4}}&\leq 4(2d)\sum_{\substack{t\in \Lambda_\ell\setminus \{0\}\cap [-(L\wedge \ell),L\wedge \ell]\times [-\ell,\ell]^{d-1}}}\frac{1}{(L\vee|t|)^{d-4}}\notag\\
	&\leq C_8L\ell^3,
\end{align}
for some $C_8>0$. From \eqref{eq:boundamplitudeError4} and \eqref{eq:boundamplitudeError6}, we know that the case $z\neq t$ will be concluded once we bound
\begin{equation}\label{eq:boundamplitudeError8.5}
	\frac{1}{L^4}\sum_{\ell\geq 0}\sum_{u\in \partial^\ell B}\mathbb P[0 \connect{B\:}u]\sum_{\substack{y,s\in B\\ z,t\notin B\\ 1\leq|z-t|\leq \ell}}\mathbb P[u \connect{B\:}y]\mathbb P[u\connect{B\:}s]\frac{p_{yz}p_{st}}{(L\vee|z-t|)^{d-4}}.
\end{equation}
We first consider the contribution $S_1$ in this sum \eqref{eq:boundamplitudeError8.5} coming from values of $u,y,s$ in which $u\neq y$ and/or $u\neq s$. Without loss of generality we may even assume that $u\neq s$. In this case, combining \eqref{eq:boundamplitudeError7}, \eqref{eq:boundamplitudeError8}, and \eqref{eq:boundamplitudeError2} gives 
\begin{align}\notag
	S_1&\leq \frac{1}{L^4}\frac{3\bfC^2C_7}{L^d}4(4d\bfC^2+1)\sum_{\ell\geq 0}\left(\frac{L}{(\ell-L)\vee L}\right)^{d-1}L\ell^3 \sum_{u\in \partial^\ell B}\mathbb P[0\connect{B\:}u]\\&\leq \frac{C_9}{L^{d+3}}\sum_{\ell \geq 0}\left(\frac{L}{(\ell-L)\vee L}\right)^{d-1}\ell^3 \Big(\mathds{1}_{0\in \partial^\ell B}+\frac{4d\bfC^2(L\vee \ell)}{L^2}\Big),
\end{align}
where we used \eqref{eq:boundamplitudeError3} and where $C_9>0$. Notice that 
\begin{equation}\label{eq: bound contribution 0 in partial l}
	\sum_{\ell \geq 0}\left(\frac{L}{(\ell-L)\vee L}\right)^{d-1}\ell^3 \mathds{1}_{0\in \partial^\ell B}\leq \left\{
    \begin{array}{ll}
        (2L)^3 & \mbox{if } \ell\leq 2L,\\
        (2L/\ell)^{d-4}(2L)^3 & \mbox{otherwise}.
    \end{array}
\right.
\end{equation}
Combining the two previously displayed equations, we get that
\begin{equation}\label{eq:boundS1}
	S_1\leq \frac{C_{10}}{L^{d+3}}\Big( (2L)^3+ \frac{1}{L}\sum_{\ell \leq 2L}\ell^3+\frac{L^{d-1}}{L^2}\sum_{\ell\geq 2L}\frac{1}{\ell^{d-5}}\Big)\leq\frac{C_{11}}{L^d},
\end{equation}
where $C_{10},C_{11}>0$. We now turn to the contribution $S_2$ in \eqref{eq:boundamplitudeError8.5} which corresponds to $u=y=s$. Note that in that case one must have $\ell \leq L$. Using \eqref{eq:boundamplitudeError3} and the bound $|\Lambda_L^*|p\leq 4$ again,
\begin{align}\notag
	S_2&=\frac{1}{L^4}\sum_{\ell=0}^L\sum_{u\in \partial^\ell B}\mathbb P[0\connect{B\:}u]\sum_{\substack{z,t\notin B\\1\leq |z-t|\leq \ell}}\frac{p_{uz}p_{ut}}{(L\vee|z-t|)^{d-4}}\\&\leq \frac{C_{12}}{L^d}\sum_{\ell=0}^L\Big(\mathds{1}_{0\in \partial^\ell B}+\frac{4d\bfC^2}{L}\Big)(|\Lambda_L^*|p)^2\leq \frac{C_{13}}{L^d},\label{eq:boundS2}
\end{align}
where $C_{12},C_{13}>0$.

%

\paragraph{Case $z=t$.} In the case $z=t$, we must have $y\neq s$ and therefore $y\neq u$ or $s\neq u$. The two cases are symmetric. Hence, we get,
 \begin{align}
 	\sum_{\ell\geq 0}&\sum_{u\in \partial^\ell B}\mathbb P[0\connect{B\:}u]\sum_{\substack{y,s\in B\\y\neq s\\ z\notin B}}\mathbb P[u\connect{B\:}y]\mathbb P[u\connect{B\:}s]p_{yz}p_{sz}\notag
 	\\&\stackrel{\eqref{eq:boundamplitudeError7}}\leq 2|\Lambda_L^*|p\frac{3\bfC^2}{L^d}\sum_{\ell\geq 0}\left(\frac{L}{(\ell-L)\vee L}\right)^{d-1}\sum_{u\in \partial^\ell B}\mathbb P[0\connect{B\:}u]\sum_{\substack{y\in B\\z\notin B}}\mathbb P[u\connect{B\:}y]p_{yz}\notag
 	\\&\stackrel{\eqref{eq:boundamplitudeError2}}\le \frac{96\bfC^2(4d\bfC^2+1)}{L^d}\sum_{\ell\geq 0}\left(\frac{L}{(\ell-L)\vee L}\right)^{d-1}\sum_{u\in \partial^\ell B}\mathbb P[0\connect{B\:}u]\notag
 	\\&\stackrel{\eqref{eq:boundamplitudeError3}}\leq \frac{C_{14}}{L^d}\sum_{\ell\geq 0}\left(\frac{L}{(\ell-L)\vee L}\right)^{d-1}\Big(\mathds{1}_{0\in \partial^\ell B}+\frac{4d\bfC^2(L\vee \ell)}{L^2}\Big) \leq \frac{C_{15}}{L^d},\label{eq:boundamplitudeError5}
 \end{align}
 where we used that $|\Lambda_L^*|p\leq 4$ in the second inequality, where $C_{14},C_{15}>0$, and where we used a similar reasoning as in \eqref{eq: bound contribution 0 in partial l} to get
 \begin{equation}
 	\sum_{\ell\geq 0}\left(\frac{L}{(\ell-L)\vee L}\right)^{d-1}\mathds{1}_{0\in \partial^\ell B}\leq C_{16},
 \end{equation}
 for some $C_{16}>0$.

Combining \eqref{eq:boundamplitudeError6}, \eqref{eq:boundS1}, \eqref{eq:boundS2}, and \eqref{eq:boundamplitudeError5} gives $C_{17}>0$ such that 
\begin{equation}
	\mathcal E(2)\leq \frac{C_{17}}{L^d}.
\end{equation}
This concludes the proof.
\end{proof}

To conclude this section, one last estimate (whose proof is postponed to Appendix \ref{appendix:boundE}) is derived similarly as Proposition~\ref{lem: error amplitude perco}. The interest of this result may be unclear for now but will become much more explicit in Section \ref{section: proof mainperco}.
\begin{Lem}\label{lem: estimate E(x)} Fix $d>6$, $\bfC>1$, and $L\geq 1$. 
There exists $D=D(\bfC,d)>0$ such that the following holds. For every $\beta<\beta^*(\bfC,L)$, every $n> 24L$, and every $x\in \partial\mathbb H_n$,
\begin{equation}
\max_{w\in \Lambda_{n/2}}\max_{\substack{B\in \mathcal B\\B+w\subset \Lambda_{n/2}}}E_\beta(B+w,\mathbb H_n,w,x)\leq \frac{D}{L^{d}}\cdot\frac{\bfC}{Ln^{d-1}}.
\end{equation}
	
\end{Lem}

\section{Proof of Theorem~\ref{thm:mainperco}}\label{section: proof mainperco}\label{sec:3}

In this section, we prove Theorem \ref{thm:mainperco}. As explained in Section \ref{section: strategy}, we will show that for $\bfC$ and $L$ large enough, one has $\beta^*(\bfC,L)=\beta_c$. We begin with an improvement of the $\ell^1$ bound.

\subsection{Improving the $\ell^1$ bound}

The main goal of this section is the following proposition. 

\begin{Prop}[Improvement of the $\ell^1$ bound]\label{prop: improved l^1 perco} Fix $d>6$. There exists $\bfC_1=\bfC_1(d)>0$ such that the following holds true: for every $\bfC>\bfC_1$, there exists $L_1=L_1(\bfC,d)>0$ such that for every $L\geq L_1$ and every $\beta<\beta^*(\bfC,L)$,
\begin{equation}
	\psi_\beta(\mathbb H_n)\leq \frac{64e^2}{L}\qquad \forall n\geq 0.
\end{equation}
\end{Prop}

Note that for $\bfC>64 e^2$, the previous bound is an improvement on the a priori bound \eqref{eq:H_beta perco} provided by the assumption $\beta<\beta^*$.
To prove the proposition, we start by relating $\psi_\beta(\mathbb H_n)$ to $\varphi_\beta(\mathbb H_n)$. 
%
%
\begin{Lem}[Bounding $\psi_\beta(\mathbb H_n)$ in terms of $\varphi_\beta(\mathbb H_n)$]\label{lem: small values 1 perco} Let $d>6$, $\bfC>1$, and $L\geq 1$. For all $\beta<\beta^*(\bfC,L)$,
\begin{align}
 \label{eq: improve psi small values}
\psi_\beta(\mathbb H_n)&\le\frac{32e^2}{L}\varphi_{\beta\vee \beta_0}(\mathbb H_n)\qquad\forall n\ge0,
\end{align}
where we recall that $\beta_0$ was defined in \eqref{eq: def beta_0}.
\end{Lem}

Before proving this lemma, let us explain how it implies the proposition.
Lemma \ref{lem: small values 1 perco} suggests to find a bound on $\varphi_\beta(\mathbb H_n)$ for $n\geq 0$. The following proposition is the crucial step of our strategy: from Proposition~\ref{lem: error amplitude perco}, we obtain a bound on $\varphi_\beta(B)$ for $B\in\mathcal B$ that involves the range $L$. For large $L$, this bound will provide an improvement on \eqref{eq:H_beta perco} thanks to Lemma \ref{lem: small values 1 perco}.
\begin{Prop}\label{prop:bound phi perco}
Fix $d>6$, $\bfC>1$, and $L\geq 1$.  There exists $K=K(\bfC,d)>0$ such that for every $\beta\leq\beta^*(\bfC,L)$ and every $B\in\mathcal B$,
\begin{equation}
\varphi_\beta(B)\leq 1+\frac{K}{L^{d}}.
\end{equation}
\end{Prop}
\begin{proof} Thanks to the monotone convergence theorem, it suffices to prove the result for $\beta<\beta^*(\bfC,L)$. Fix $\beta<\beta^*(\bfC,L)$ and $B\in\mathcal B$ . Observe that
\begin{equation}
	\sum_{x\in \mathbb Z^d}E_\beta(B,\mathbb Z^d,0,x)=\chi(\beta)E_\beta(B),
\end{equation}
where $\chi(\beta):=\sum_{x\in \mathbb Z^d}\mathbb P_\beta[0\connect{\:}x]<\infty$ because $\beta<\beta_c$ \cite{Mensikov1986coincidence,AizenmanBarsky1987sharpnessPerco} (see also \cite{DuminilTassionNewProofSharpness2016}).
Hence, summing \eqref{eq:reversed SL perco} (with $S=B$ and $\Lambda=\mathbb Z^d$) over every $x\in \mathbb Z^d$, and using that $\mathbb P_\beta[0\connect{B\:}x]\geq 0$, gives
\begin{equation}
\chi(\beta)\geq \varphi_\beta(B)\chi(\beta)-E_\beta(B)\chi(\beta).
\end{equation}
The result follows by dividing by $\chi(\beta)$ and using Proposition~\ref{lem: error amplitude perco}.
\end{proof}

\begin{Rem}\label{rem: generalise the bound on phi} It would be interesting to prove that such a bound holds for every finite set $B$ containing $0$ and not only for $B\in \mathcal B$. 
%
\end{Rem}

We are now in a position to prove the main result of this section.

\begin{proof}[Proof of Proposition~\textup{\ref{prop: improved l^1 perco}}] Let $\bfC>\bfC_{\rm RW}$ (where $\bfC_{\rm RW}$ is given by Lemma \ref{lem: intro lem about betastar}) so that for every $L\geq 1$, one has $\beta^*(\bfC,L)\geq \beta_0$. Fix $L_1$ such that $\tfrac{K}{L_1^d}\leq 1$, where $K$ is provided by Proposition \ref{prop:bound phi perco}. Fix $L\ge L_1$ and $n\geq 0$. Combining Lemma \ref{lem: small values 1 perco} and Proposition \ref{prop:bound phi perco} gives
\begin{equation}
	\psi_\beta(\mathbb H_n)\leq \frac{32e^2}{L}\varphi_{\beta\vee \beta_0}(\mathbb H_n)\leq \frac{32e^2}{L}\varphi_{\beta^*(\bfC,L)}(\mathbb H_n)\leq \frac{32e^2}{L}\Big(1+\frac{K}{L^d}\Big)\leq \frac{64e^2}{L}.
\end{equation}
\end{proof}

To conclude the proof of Proposition~\ref{prop: improved l^1 perco}, it remains to prove Lemma~\ref{lem: small values 1 perco}.
\begin{proof}[Proof of Lemma~\textup{\ref{lem: small values 1 perco}}] By monotonicity, it suffices to prove the result for $\beta_0\leq \beta<\beta^*$. Fix $n\geq 0$ and $\beta_0\leq\beta<\beta^*$. Drop $\beta$ from the notations. Let $S_k:=\mathbb H_n\setminus\mathbb H_{n-k-1}$. If $x\in (\partial \mathbb H_n)^*$, Lemma \ref{Lem: SL upper bound} with $S=\{x\}$ and $\Lambda=\mathbb H_n$ gives
\begin{equation}\label{eq:hh0}
	\mathbb P[0\connect{\mathbb H_n\:}x]\leq \sum_{y\in S_L}p_{xy}\mathbb P[0\connect{\mathbb H_n\:}y].
\end{equation}
Summing \eqref{eq:hh0} over $x\in (\partial \mathbb H_n)^*$ implies
\begin{align}\label{eq:hha}
\psi(\mathbb H_n)&\le \sum_{\substack{x\in(\partial\mathbb H_n)^*\\ y\in S_L}}p_{xy}\mathbb P[0\connect{\mathbb H_n\:}y]\le \frac{p|\Lambda_L|}{2L+1}\sum_{y\in S_L}\mathbb P[0\connect{\mathbb H_n\:}y].
\end{align}

Our aim is now to compare $p|\Lambda_L|\sum_{y\in S_L}\mathbb P[0\connect{\mathbb H_n\:}y]$ to $\varphi_\beta(\mathbb H_n)$. Such a comparison is not completely straightforward since $|\{z\notin \mathbb H_n: |z-y|\leq L\}|$ can be of much smaller order than $|\Lambda_L|$ for $y$ in $S_L\setminus S_{L/2}$. The idea will be to argue that for $y\in \mathcal S_L$, if the event $\{0\connect{\mathbb H_n\:}y\}$ occurs, then $y$ typically connects to a point $z$ in $S_{L/2}$. This will allow to bound $\sum_{y\in S_L}\mathbb P[0\connect{\mathbb H_n\:}y]$ in terms of $\sum_{z\in S_{L/2}}\mathbb P[0\connect{\mathbb H_n\:}z]$ and to conclude.

For $y\in S_L$, let
\begin{equation}
A_y:=S_{L/2}\cap \Lambda_L(y)\setminus\{z\in S_{L/2}:z\connect{\mathbb H_n\:}y\}
\end{equation}
be the set of points in the slab $S_{L/2}$ which are within a distance $L$ of $y$, but not connected to $y$ in $\mathbb H_n$, see Figure \ref{fig:psiphi}. Observe that
\begin{equation}\label{eq:proof comp psi phi1}
	\tfrac18|\Lambda_L|\mathbb P[0\connect{\mathbb H_n\:}y,|A_y|\ge\tfrac18|\Lambda_L|]\le \mathbb E[\mathds{1}\{0\connect{\mathbb H_n\:}y\}|A_y|]=\sum_{z\in S_{L/2}\cap \Lambda_L(y)}\mathbb P[0\connect{\mathbb H_n\:}y,z\in A_y].
\end{equation}
Moreover, by choosing $z\in A_y$ (note that $z$ is not in the cluster of $y$ in $\mathbb H_n$) and opening the edge $yz$ while closing all other edges emanating from $z$, we get that for $z\in S_{L/2}\cap \Lambda_L(y)$,
\begin{align}\label{eq:proof comp psi phi2}
p&\mathbb P[0\connect{\mathbb H_n\:}y,z\in A_y]
\le e^{\beta} \mathbb P[0\connect{\mathbb H_n\:}y, \: yz \text{ open},\: E_z],
\end{align}
where $E_z$ be the event that there is {\em exactly one} open edge emanating from $z$ (we used that $\mathbb P[zu \text{ closed}, \forall u\neq y]\geq e^{-\beta c_L(|\Lambda_L|-2)}\ge e^{-\beta}$). As a consequence,
\begin{align}
\sum_{y\in S_L}\mathbb P[0\connect{\mathbb H_n\:}y]&= \sum_{y\in S_L}\mathbb P[0\connect{\mathbb H_n\:}y,|A_y|< \tfrac18|\Lambda_L|]+\mathbb P[0\connect{\mathbb H_n\:}y,|A_y|\ge \tfrac18|\Lambda_L|]\\
&\le \sum_{y\in S_L}\sum_{z\in S_{L/2}\cap\Lambda_L(y)}\frac{8}{|\Lambda_L|}\mathbb P[0\connect{\mathbb H_n\:}y,z]+ \frac{8e^\beta}{p|\Lambda_L|}\mathbb P[0\connect{\mathbb H_n\:}y,yz\text{ open},E_z],\label{eq:proof comp psi phi3}
\end{align}
where on the second line we used \eqref{eq:proof comp psi phi1}--\eqref{eq:proof comp psi phi2}, and the fact that $|\mathcal S_{L/2}\cap \Lambda_L(y)|\geq \tfrac{1}{4}|\Lambda_L|$ (recall that $y\in \mathcal S_L$) to obtain
\begin{equation}
\mathds{1}\{|A_y|<\tfrac{1}{8}|\Lambda_L|\}\leq\mathds{1}\{|S_{L/2}\cap \Lambda_L(y)\setminus A_y|\geq\tfrac{1}{8}|\Lambda_L|\}\leq \frac{8}{|\Lambda_L|}\sum_{z\in S_{L/2}\cap \Lambda_L(y)}\mathds{1}\{z\connect{\mathbb H_n\:}y\}.
\end{equation}
Thus, \eqref{eq:proof comp psi phi3} yields
\begin{align}\notag
\sum_{y\in S_L}\mathbb P[0\connect{\mathbb H_n\:}y]&\le \sum_{z\in S_{L/2}}\sum_{y\in S_L\cap \Lambda_L(z)}\frac{8}{|\Lambda_L|}\mathbb P[0\connect{\mathbb H_n\:}z]+ \frac{8e^\beta}{p|\Lambda_L|}\mathbb P[0\connect{\mathbb H_n\:}z,yz\text{ open},E_z]\\
&\le 8\Big(1+\frac{e^\beta}{p|\Lambda_L|}\Big)\sum_{z\in S_{L/2}}\mathbb P[0\connect{\mathbb H_n\:}z].\label{eq:hhb}
\end{align}
To conclude, we notice that if $z\in \mathcal S_{L/2}$ then $|\{t\notin \mathbb H_n: |z-t|\leq L\}|\geq \tfrac{1}{4}|\Lambda_L|$. As a result,
\begin{align}\label{eq:hhbb}
p\frac14|\Lambda_L|\sum_{z\in S_{L/2}}\mathbb P[0\connect{\mathbb H_n\:}z]&\le \sum_{\substack{z\in S_{L/2}\\t\notin \mathbb H_n}}p_{zt}\mathbb P[0\connect{\mathbb H_n\:}z]\le\varphi_\beta(\mathbb H_n).
\end{align}
Moreover, recall from Lemma \ref{lem: intro lem about betastar} that
\begin{equation}\label{eq:hhbbb}
1\leq|\Lambda_L|p\leq 4
\end{equation}
when $\beta_0\leq \beta\le \beta^*$. Combining \eqref{eq:hha}, \eqref{eq:hhb}, \eqref{eq:hhbb}, and \eqref{eq:hhbbb} imply that 
\begin{equation}
	\psi(\mathbb H_n)\leq \frac{8(p |\Lambda_L|+e^2)}{2L+1}\sum_{z\in \mathcal S_{L/2}}\mathbb P[0\connect{\mathbb H_n\:}z]\leq \frac{8(4+e^2)}{2L+1}\frac{4}{p |\Lambda_L|}\varphi(\mathbb H_n)\leq \frac{32e^2}{L}\varphi(\mathbb H_n).
\end{equation}
%
\end{proof}
\begin{figure}
	\begin{center}\includegraphics[scale=1.2]{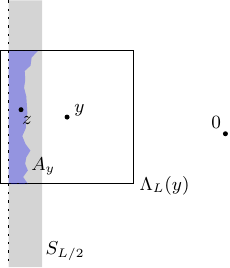}
	    \put(-152,8){$\partial\mathbb H_n$}
		\caption{An illustration of the proof of Lemma \ref{lem: small values 1 perco}. The grey region represents $S_{L/2}$. The blue region is the (random) set $A_y$. It is disconnected from $y$ in $\mathbb H_n$.}
		\label{fig:psiphi}
	\end{center}
\end{figure}

\paragraph{A refined version of Proposition~\ref{prop:bound phi perco}.}
As we conclude this section, we would like to highlight a supplementary result--- building on Proposition~\ref{prop:bound phi perco}--- which will be useful in the next section. We begin with some notations.
Let $\Lambda_n^{+}:=\{x\in \Lambda_n:x_1>0\}$. Also, define the slabs $H=H(L):=\{v \in \mathbb Z^d: \: |v_1|\le L\}$ and $H^-=H^{-}(L):=\{v\in H: \: v_1\leq 0\}$. \begin{Lem}\label{lem:estimate half-space perco} Fix $d>6$ and $\bfC>1$. There exists $L_2=L_2(\bfC,d)>0$ such that for every $L\geq L_2$, every $\beta<\beta^*(\bfC,L)$, every $n\geq 12L$, and every $v\in \Lambda_k^+$ with $6L\le k\le n/2$,
\begin{equation}
\sum_{\substack{y \in \Lambda_n^+\\z\notin \Lambda_n^+\cup H}}\mathbb P_\beta[v\connect{\Lambda_n^+\:}y]p_{yz,\beta} \le 4\Big(\frac {2k}n\Big)^c,
\end{equation}
where $c=c(d):=\frac{|\log(1-\tfrac{1}{4d^2})|}{2\log 2}$.
\end{Lem}

Roughly speaking, Lemma \ref{lem:estimate half-space perco} formalises the fact that when a generalised block $B$ is not centered and 0 is close to $\partial B$, most of the sum in $\varphi_\beta(B)$ comes from edges $yz$ on the side of $B$ that is the closest to $0$. The proof is very similar to that of \cite[Lemma~2.7]{DumPan24WSAW}, but the presence of jumps of range $L$ requires some extra care.
\begin{proof} Fix $d>6$ and $\bfC>1$, and let $K$ be given by Proposition \ref{prop:bound phi perco}. Let $L_2$ be such that $\tfrac{K}{L_2^d}\leq \tfrac{1}{2d}$. Fix $L\geq L_2$.
Define $(n_\ell)_{\ell\geq 0}$ by $n_0=n$ and then $n_{\ell+1}=\lfloor (n_\ell-L)/2\rfloor\vee 0$. We proceed by induction by proving that for every $\ell\ge0$ and $v\in \Lambda_{n_\ell}^+$, 
\begin{equation}\label{eq: induction intermediary step}
\sum_{\substack{y \in \Lambda_n^+\\z\notin \Lambda_n^+\cup H}}\mathbb P_\beta[v\connect{\Lambda_n^+\:}y]p_{yz,\beta} \le \Big(1-\frac{1}{4d^2}\Big)^{\ell}\Big(1+\frac{1}{2d}\Big).
\end{equation}
The case $\ell=0$ follows from Proposition \ref{prop:bound phi perco} and the assumption on $L$. We now transfer the estimate from $\ell$ to $\ell+1$. Fix $v\in \Lambda_{n_{\ell+1}}^+$. Let $B:=\Lambda_{v_1-1}(v)$, which is included in $\Lambda_{n_\ell}^+$ and has one of its faces included in $H$, see Figure \ref{fig:improvePhi}. By symmetry and Proposition \ref{prop:bound phi perco},
\begin{equation}\label{eq: proof closer side contributes the most}
\sum_{\substack{r \in B\\s\notin B\cup H^-}}\mathbb P_\beta[v\connect{B\:}r]p_{rs,\beta} \le \frac{2d-1}{2d}\varphi_\beta(B)\le \Big(1-\frac{1}{2d}\Big)\Big(1+\frac{K}{L^d}\Big)\leq 1-\frac{1}{4d^2}.
\end{equation}
We deduce from Lemma~\ref{Lem: SL upper bound} and the induction hypothesis that 
\begin{align*}
\sum_{\substack{y \in \Lambda_n^+\\z\notin \Lambda_n^+\cup H}}\mathbb P_\beta[v\connect{\Lambda_n^+\:}y]p_{yz,\beta}
&\stackrel{\phantom{\eqref{eq: proof closer side contributes the most}}}\le 
\sum_{\substack{y \in \Lambda_n^+\\z\notin \Lambda_n^+\cup H}}\Big(\sum_{\substack{r \in B\\s\notin B\cup H^-}}\mathbb P_\beta[v\connect{B\:}r]p_{rs,\beta}\mathbb P_\beta[s\connect{\Lambda_n^+\:}y]\Big)p_{yz,\beta} 
\\
&\stackrel{\phantom{\eqref{eq: proof closer side contributes the most}}}=\sum_{\substack{r \in B\\s\notin B\cup H^-}}\mathbb P_\beta[v\connect{B\:}r]\Big(\sum_{\substack{y \in \Lambda_n^+\\z\notin \Lambda_n^+\cup H}}\mathbb P_\beta[s\connect{\Lambda_n^+\:}y]p_{yz,\beta}\Big)p_{rs,\beta}
\\
&\stackrel{\phantom{\eqref{eq: proof closer side contributes the most}}}\le \Big(1-\frac{1}{4d^2}\Big)^{\ell}\Big(1+\frac{1}{2d}\Big)\sum_{\substack{r \in B\\s\notin B\cup H^-}}\mathbb P_\beta[v\connect{B\:}r]p_{rs,\beta}
\\&\stackrel{\eqref{eq: proof closer side contributes the most}}\le \Big(1-\frac{1}{4d^2}\Big)^{\ell+1}\Big(1+\frac{1}{2d}\Big).
\end{align*}
This concludes the induction.
Now, let $k\leq n/2$. Then\footnote{We notice that for $\ell\geq 0$, one has $n_\ell\geq \tfrac{n}{2^\ell}-\sum_{k=0}^{\ell-1}\tfrac{L+2}{2^k}\geq \tfrac{n}{2^\ell}-6L$. Hence, for $\ell=\lfloor m/2\rfloor$ with $m=\log_2(\tfrac{n}{2k})$, we get
\begin{equation*}
	n_\ell\geq \frac{n}{2^{m/2}}-6L=\sqrt{2k}\sqrt{n}-6L\geq 2k-6L\geq k,
\end{equation*}
where we used that $n\geq 2k$ and $k\geq 6L$.}, $k\leq n_{\lfloor m/2\rfloor}$ where $m:=\log_2(\tfrac{n}{2k})$. As a consequence, by \eqref{eq: induction intermediary step}, if $v\in \Lambda_k^+$,
\begin{equation}
	\sum_{\substack{y \in \Lambda_n^+\\z\notin \Lambda_n^+\cup H}}\mathbb P_\beta[v\connect{\Lambda_n^+\:}y]p_{yz,\beta}\leq \Big(1-\frac{1}{4d^2}\Big)^{(m/2)-1}\Big(1+\frac{1}{2d}\Big)\leq 4\Big(\frac{2k}{n}\Big)^c,
\end{equation}
where $c=\frac{|\log(1-\tfrac{1}{4d^2})|}{2\log 2}$, and where we used that for $d\geq 1$, $(1-\tfrac{1}{4d^2})^{-1}\leq 1+\tfrac{1}{2d}\leq 2$.
\end{proof}

\begin{figure}
	\begin{center}
		\includegraphics[scale=1.2]{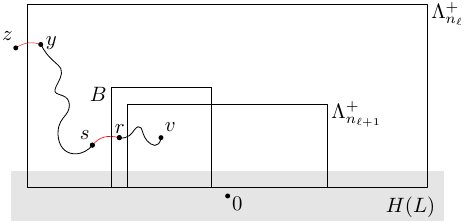}
		\caption{An illustration of the proof of Lemma \ref{lem:estimate half-space perco}. For convenience, we have rotated the picture by $\pi/2$. The light grey region corresponds to $H$. The bold black line corresponds to an open self-avoiding path connecting $v$ to $y$. The red lines correspond to long open edges. The sequence $(n_{\ell})_{\ell\geq 0}$ has been constructed in such a way that $s$ remains in $\Lambda_{n_\ell}$.}
		\label{fig:improvePhi}
	\end{center}
\end{figure}

\subsection{Improving the $\ell^\infty$ bound}
Proposition \ref{prop: improved l^1 perco} implies an improved $\ell^1$-type bound on the half-space two-point function (when $\bfC,L$ are large). We will use this improvement to deduce an improved $\ell^\infty$ bound on the half-space two-point function. 
The following estimate is a key step of the proof.
\begin{Prop}[Control of the gradient]\label{prop:regularity perco}

Fix $d>6$ and $\eta>0$. There exist $\delta=\delta(\eta,d)\in(0,1/2)$ and $\bfC_3=\bfC_3(\eta,d)>0$ such that the following holds true: for every ${\bf C}>\bfC_3$, there exists $L_3=L_3(\eta,\bfC,d)$ such that 
\begin{align}\label{eq:gradient estimate}
\Big|\mathbb P_\beta[u\connect{\mathbb H_n\:}x]-\mathbb P_\beta[v\connect{\mathbb H_n\:}x]\Big|&\le \eta\frac{\bfC}{L^d}\left(\frac{L}{L\vee n}\right)^{d-1}
\end{align}
for every $L\geq L_3$, $\beta<\beta^*(\bfC,L)$, $n\geq 0$,  $x\in \partial\mathbb H_n$, and  $u,v\in \Lambda_{\lfloor \delta n\rfloor\vee L}\setminus\{x\}$.
\end{Prop}
With this regularity estimate, one may improve \eqref{eq:H_beta-' perco}. \begin{Cor}[Improvement of the $\ell^\infty$ bound]\label{prop: improve l inf perco} Fix $d>6$. There exists $\bfC_4=\bfC_4(d)>0$ such that the following holds. For every $\bfC>\bfC_4$, there exists $L_4=L_4(\bfC,d)$ such that for every $L\geq L_4$ and every $\beta<\beta^*(\bfC,L)$,
\begin{align}
\mathbb P_\beta[0\connect{\mathbb H_n\:}x]&\le \frac{\bfC}{2L^d}\left(\frac{L}{L\vee n}\right)^{d-1}&\forall n\geq 0,\:\forall x\in (\partial\mathbb H_n)^*.
\end{align}
\end{Cor}

\begin{proof}
Fix $d>6$. Set $\eta=\tfrac14$. Let $\bfC_1$ be given by Proposition \ref{prop: improved l^1 perco} and $\delta=\delta(\eta,d),\bfC_3=\bfC_3(\eta,d)$ be given by Proposition~\ref{prop:regularity perco}. Let $\bfC >\bfC_1\vee \bfC_3$. It will be taken larger below. Let also $L\geq L_3$ where $L_3=L_3(\eta,\bfC,d)$ is given by Proposition \ref{prop:regularity perco}. Let $n\geq 0$ and $x\in (\partial\mathbb H_n)^*$. Set $
V_n:=\{y\in\Lambda_{ \lfloor \delta n\rfloor\vee L }\setminus\{x\}:y_1=0\}.$
Proposition~\ref{prop:regularity perco} gives that for every $\beta<\beta^*(\bfC,L)$ and every $y\in V_n$,
\begin{align}\notag
\mathbb P_\beta[0\connect{\mathbb H_n\:}x-y]&=\mathbb P_\beta[y\connect{\mathbb H_n\:}x]
\ge \mathbb P_\beta[0\connect{\mathbb H_n\:}x]-\eta\frac{\bfC}{L^d}\left(\frac{L}{L\vee n}\right)^{d-1}\label{eq:h1}.
\end{align}
Averaging over $y\in V_n$ gives
\begin{equation}
	\frac{\psi(\mathbb H_n)}{|V_n|}\geq \frac{1}{|V_n|}\sum_{y\in V_n}\mathbb P_\beta[y\connect{\mathbb H_n\:}x]\geq \mathbb P_\beta[0\connect{\mathbb H_n\:}x]-\eta\frac{\bfC}{L^d}\left(\frac{L}{L\vee n}\right)^{d-1}.
\end{equation}
Proposition~\ref{prop: improved l^1 perco} implies that for $L\ge L_1\vee L_3$ (where $L_1=L_1(\bfC,d)$ is provided by Proposition~\ref{prop: improved l^1 perco}),
\begin{equation}\label{eq:final improvement 2pt}
	\mathbb P_\beta[0\connect{\mathbb H_n\:}x]\leq \frac{1}{|V_n|}\frac{64e^2}{L}+\frac{1}{4}\frac{\bfC}{L^d}\left(\frac{L}{L\vee n}\right)^{d-1},
\end{equation}
where we also used that $\eta=\tfrac{1}{4}$.

Observe that $|V_n|\geq 2^{d-1}(\lfloor\delta n\rfloor\vee L)^{d-1}$. If $n\delta \leq L$, then $|V_n|\geq L^{d-1}$, and \eqref{eq:final improvement 2pt} gives
\begin{equation}
	\mathbb P_\beta[0\connect{\mathbb H_n\:}x]\leq \frac{64e^2}{L^d}+\frac{1}{4}\frac{\bfC}{L^d}\left(\frac{L}{L\vee n}\right)^{d-1}\leq \left(\frac{64e^2}{\delta^{d-1}}+\frac{\bfC}{4}\right)\frac{1}{L^d}\left(\frac{L}{L\vee n}\right)^{d-1}.
\end{equation}
If now $n\delta>L$, then $\lfloor \delta n\rfloor\geq L$ and $|V_n|\geq (2\lfloor \delta n\rfloor)^{d-1}\geq (\delta n)^{d-1}$. Again, \eqref{eq:final improvement 2pt} gives
\begin{equation}
	\mathbb P_\beta[0\connect{\mathbb H_n\:}x]\leq \left(\frac{64e^2}{\delta^{d-1}}+\frac{\bfC}{4}\right)\frac{1}{L^d}\left(\frac{L}{L\vee n}\right)^{d-1}.
\end{equation}
Using the two previously displayed equations, and choosing $\bfC>\bfC_4:=(4\frac{64e^2}{\delta^{d-1}})\vee \bfC_1\vee  \bfC_3$ and $L_4:=L_1\vee L_3$ gives
\begin{equation}
	\mathbb P_\beta[0\connect{\mathbb H_n\:}x]\leq\frac{\bfC}{2L^d}\left(\frac{L}{L\vee n}\right)^{d-1},
\end{equation} 
and concludes the proof.
\end{proof}

The remainder of this section is devoted to the proof of Proposition~\ref{prop:regularity perco}. The following lemma is an iterated version of Lemmata \ref{Lem: SL upper bound} and \ref{Lem: SL lower bound}.

\begin{Lem}\label{lem: iteration of SL} Fix $d>6$, $\bfC>1$, and $L\geq 1$. Then, for every $\beta<\beta^*(\bfC,L)$, every $n\geq 0$, every $x\in \partial \mathbb H_n$, every $u \in \mathbb H_n$, and every $T\geq 1$,
\begin{multline}
	0\leq \Big(\sum_{t=0}^{T}\varphi^t\mathbb P^{\rm RW}_{u}[\sigma^u_x=t, \tau_n^u>t]+\varphi^T\mathbb E^{\rm RW}_{u}\Big[\mathbb P_\beta[X_T^u\connect{\mathbb H_n\:}x]\mathds{1}_{\tau_n^u>T, \: \sigma^u_x>T}\Big]\Big)-\mathbb P_\beta[u\connect{\mathbb H_n\:}x]\\
	\leq \Big(\sum_{t=0}^{T-1}\varphi^t\Big)\max_{w\in \Lambda_{(T-1)L}(u)\cap \mathbb H_n}E_\beta(\{w\},\mathbb H_n,w,x),
\end{multline}
where $\tau_n^u:=\inf\{k\geq 1 : X_k^u \notin \mathbb H_n\}$, $\sigma_x^u:=\inf\{k\geq 0: X_k^u=x\}$, and $\varphi:=\varphi_\beta(\{0\})$.
\end{Lem}
\begin{proof} One application of Lemmata \ref{Lem: SL upper bound} and \ref{Lem: SL lower bound} with $S=\{u\}$ and $\Lambda=\mathbb H_n$ gives,
\begin{align}\label{eq:proofiterationofSL1}
	0\leq \Big(\mathds{1}_{u=x}+\varphi \mathbb E_u^{\rm RW}\Big[\mathbb P_\beta[X_1^u\connect{\mathbb H_n\:}x]\mathds{1}_{\tau_n^u>1}\Big]\Big)&-\mathbb P_\beta[u\connect{\mathbb H_n\:}x]\leq E_\beta(\{u\},\mathbb H_n,u,x),
\end{align}
where we used that the uniform step distribution in $\Lambda_L(u)^*$ is given by $(J_{uz})_{z\in \mathbb Z^d}$ which implies that for $z\in \mathbb Z^d$, $ \varphi^{-1} p_{uz,\beta}=J_{uz}$ and
\begin{equation}
	\sum_{\substack{z\notin S\\z\in \mathbb H_n}}p_{uz,\beta}\mathbb P_\beta[z\connect{\mathbb H_n\:}x]=\varphi \sum_{z\in (\Lambda_L(u)\setminus \{0\})} J_{uz}\mathds{1}_{z\in \mathbb H_n}\mathbb P_\beta[z\connect{\mathbb H_n\:}x]= \varphi \cdot \mathbb E^{\rm RW}_u\Big[\mathbb P_\beta[X_1^u\connect{\mathbb H_n\:}x]\mathds{1}_{\tau^u_n>1}\Big].
\end{equation}

The term in the middle of \eqref{eq:proofiterationofSL1} can be rewritten
\begin{equation}
\Big(\sum_{t=0}^1\varphi^t\mathbb P^{\rm RW}_{u}[\sigma^u_x=t, \: \tau_n^u>t]+\varphi \mathbb E_u^{\rm RW}\Big[\mathbb P_\beta[X_1^u\connect{\mathbb H_n\:}x]\mathds{1}_{\tau_n^u>1,\:\sigma^u_x>1}\Big]\Big)-\mathbb P_\beta[u\connect{\mathbb H_n\:}x].
\end{equation}
The proof follows by iterating the application of the two lemmata (successively with $S=\{X_1^u\},\{X_2^u\},\ldots,\{X_{T-1}^u\}$ and $\Lambda=\mathbb H_n$) and by noticing that,  for $k\leq T-1$,
\begin{align}\notag
	\mathbb E_u^{\rm RW}[E_\beta(\{X_k^u\},\mathbb H_n,X_k^u,x)]&\leq \max_{w\in \Lambda_{kL}(u)\cap \mathbb H_n}E_\beta(\{w\},\mathbb H_n,w,x)\\&\leq \max_{w\in \Lambda_{(T-1)L}(u)\cap \mathbb H_n}E_\beta(\{w\},\mathbb H_n,w,x).
\end{align}
\end{proof}
We are now in a position to prove Proposition \ref{prop:regularity perco}. The proof is divided into two cases: either $n\geq CL$ or $n<CL$, for some $C\gg 1$ to be fixed.

For the former case, the idea will be to reduce the problem to the case $|u-v|\leq 2L \asymp \big(\mathbb E_0^{\rm RW}[|X_1^0|^2]\big)^{1/2}$. This can be performed using a ``reflection argument'' illustrated in Figure \ref{fig: reguBIS}. The strategy of proof for $|u-v|\leq 2L$ is inspired by classical random walk arguments. First, Lemma \ref{lem: iteration of SL} allows us to express the difference between $\mathbb P_\beta[u\connect{\mathbb H_n\:}x]$ and $\mathbb P_\beta[v\connect{\mathbb H_n\:}x]$ in terms of 
\begin{equation} \varphi_\beta(\{0\})^T\Big(\mathbb E^{\rm RW}_{u}\Big[\mathbb P_\beta[X_T^u\connect{\mathbb H_n\:}x]\Big]-\mathbb E^{\rm RW}_{v}\Big[\mathbb P_\beta[X_T^v\connect{\mathbb H_n\:}x]\Big]\Big)
\end{equation}
 for $T\gg 1$ to be fixed. Then, in the spirit of \cite[Lemma~2.4.3]{LawlerLimicRandomWalks2010}, one may choose $T$ large enough so that $X_T^u$ and $X_T^v$ can be coupled to coincide with high probability. An additional complexity arises in our setting, as we require estimates (on $T$) that are uniform in $L$. This difficulty is solved in Proposition \ref{prop:coupling appendix}. Finally, the potentially large quantity $\varphi_\beta(\{0\})^T$ can be made smaller than $2$ using Proposition \ref{prop:bound phi perco} and choosing $L$ large enough. 
 
 The case $n<CL$ is easier and involves bounding each term on the left-hand side of \eqref{eq:gradient estimate} individually. The idea is to iterate Lemma \ref{Lem: SL upper bound} to compare $\mathbb P_\beta[u\connect{\mathbb H_n\:}x]$ (or $\mathbb P_\beta[v\connect{\mathbb H_n\:}x]$) to 
 \begin{equation} \sum_{k\geq 0}\mathbb P_u^{\rm RW}[X_k^u=x, \: \tau^u_n>k].\end{equation} This latter quantity is bounded by $C_{\rm RW}/L^d$ thanks to Proposition \ref{prop: rw estimates 2}, and the proof follows (essentially) from choosing $\bfC\gg C_{\rm RW}$.
\begin{proof}[Proof of Proposition~\textup{\ref{prop:regularity perco}}] Fix $\eta>0$. We let $\delta=\delta(\eta,d)>0$ to be chosen later (think of it as small). Also, let $\bfC>1$ and $L\geq 1$ to be taken large enough. Let $\beta<\beta^*(\bfC,L)$, $n\geq 0$, and $x\in \partial \mathbb H_n$. We distinguish two cases: $\lfloor \delta n\rfloor \vee L= \lfloor \delta n \rfloor$, or $\lfloor \delta n\rfloor \vee L=L$. The first case fixes how small $\delta$ has to be, while the second case fixes how large $\bfC$ and $L$ have to be. For simplicity, below, we omit integer rounding.
\paragraph{Case $n\geq \delta^{-1}L$.} We split the proof in two depending on whether the distance between $u$ and $v$ is smaller than $2L$ or not.

Assume first that $u,v \in \Lambda_{n/3+L}$ are such that $|u-v|\leq 2L$. Let $T\geq 1$ to be fixed and assume that $\delta\leq \tfrac{1}{6T}$, so that $n\geq 6TL$. We set $\varphi:= \varphi_\beta(\{0\})$. By Lemma \ref{lem: iteration of SL},
\begin{align}
	\mathbb P_\beta[u\connect{\mathbb H_n\:}x]&\leq \varphi^T \mathbb E^{\rm RW}_{u}\Big[\mathbb P_\beta[X_T^u\connect{\mathbb H_n\:}x]\Big], \label{eq:r1}
	\\\mathbb P_\beta[v\connect{\mathbb H_n\:}x]&\geq \varphi^T\mathbb E^{\rm RW}_{v}\Big[\mathbb P_\beta[X_T^v\connect{\mathbb H_n\:}x]\Big]-\Big(\sum_{t=0}^{T-1}\varphi^t\Big)\max_{w\in \Lambda_{n/3+TL}}E_\beta(\{w\},\mathbb H_n,w,x),\label{eq:r2}
\end{align}
where we used that since $n\geq 6TL$, one has (if $w\in \{u,v\}$) $\mathds{1}_{\tau_n^w>T}=1$ and also $\mathds{1}_{\sigma_x^w>k}=1$ for all $k\leq T$.
A random-walk estimate derived in Proposition \ref{prop:coupling appendix}
implies the existence of $T=T(\eta,d)$ large enough such that $X_T^u$ and $X_T^v$ can be coupled to coincide with probability larger than $1-\eta/2^{d+2}$. This implies
\begin{align}
\mathbb E^{\rm RW}_{u}\Big[\mathbb P_\beta[X_T^u\connect{\mathbb H_n\:}x]\Big]-\mathbb E^{\rm RW}_{v}\Big[\mathbb P_\beta[X_T^v\connect{\mathbb H_n\:}x]\Big]&\stackrel{\phantom{\eqref{eq:H_beta-' perco}}}\le \frac{\eta}{2^{d+2}}\max\Big\{\mathbb P_\beta[w\connect{\mathbb H_n\:}x]:w\in \Lambda_{n/3+TL}\Big\}\notag
\\&\stackrel{\phantom{\eqref{eq:H_beta-' perco}}}\leq \frac{\eta}{2^{d+2}}\max\Big\{\mathbb P_\beta[w\connect{\mathbb H_n\:}x]:w\in \Lambda_{n/2}\Big\}\notag
\\&\stackrel{\eqref{eq:H_beta-' perco}}\leq \frac{\eta}{8}\frac{\bfC}{Ln^{d-1}}\label{eq: coupling micro scales} 
.
\end{align}
Lemma~\ref{lem: estimate E(x)} gives
\begin{equation}\label{eq:r3}
	\max_{w\in \Lambda_{n/3+TL}}E_\beta(\{w\},\mathbb H_n,w,x)\leq \frac{D}{L^{d}}\frac{\bfC}{Ln^{d-1}}.
\end{equation}
Furthermore, by Proposition \ref{prop:bound phi perco}, $\varphi\le 1+\tfrac{K}{L^d}$. Hence, we may choose $L_3=L_3(\eta,T,\bfC,d)$ large enough such that for $L\geq L_3$, $\frac{D}{L^{d}}\sum_{t=0}^{T-1}\varphi^t\le \eta/4$ and $\varphi^T\leq 2$. Plugging \eqref{eq: coupling micro scales} and \eqref{eq:r3} in the difference of \eqref{eq:r1} and \eqref{eq:r2} gives
\begin{equation}\label{eq:proofregclose}
	\mathbb P_\beta[u\connect{\mathbb H_n\:}x]-\mathbb P_\beta[v\connect{\mathbb H_n\:}x]\leq \frac{\eta}{2}\frac{\bfC}{Ln^{d-1}}.
\end{equation}

We turn to the general case $u,v\in \Lambda_{\delta n}$. We prove the result when $u$ and $v$ differ by one coordinate only and assume that the difference of the coordinate is even as the latter assumption can be relaxed by using the previous estimate when $u$ and $v$ are at distance 1 (which is less than $2L$), and the former by summing increments over coordinates. By rotating and translating\footnote{Indeed, one can check that the proof below still works if we translate the boxes $B^+$ and $B^-$ by $w$ satisfying $|w|\leq \delta n$.}, we may therefore consider $u=k\mathbf{e_1}$ and $v=-k\mathbf{e}_1$ with $k\leq \delta n$. 

Recall that for $\ell\geq 1$, $\Lambda_\ell^+=\{x\in \Lambda_\ell : x_1>0\}$. Consider the sets $B^+:=\Lambda_{n/3}^+$ and $B^-:=-\Lambda_{n/3}^+$. By Lemmata \ref{Lem: SL upper bound} and \ref{Lem: SL lower bound},
\begin{align}
\mathbb P_\beta[u\connect{\mathbb H_n\:}x] &\le \sum_{\substack{y \in B^+\\ z\notin B^+}}\mathbb P_\beta[u\connect{B^+\:}y]p_{yz,\beta}  \mathbb P_\beta[z\connect{\mathbb H_n\:}x]\label{eq:u1},\\
\mathbb P_\beta[v\connect{\mathbb H_n\:}x]&\ge \sum_{\substack{y \in B^-\\ z\notin B^-}}\mathbb P_\beta[v\connect{B^-\:}y]p_{yz,\beta} \mathbb P_\beta[z\connect{\mathbb H_n\:}x] -E_\beta(B^-,\mathbb H_n,v,x).\label{eq:u2}
\end{align}
Now, when $z\in H(L)=\{v\in \mathbb Z^d: |v_1|\leq L\}$, we may associate every pair $(y,z)$ in the sum in~\eqref{eq:u1} with the pair $(y',z')$ symmetric with respect to the hyperplane $\{u\in \mathbb Z^d:u_1=0\}$ in the sum in~\eqref{eq:u2}, see Figure \ref{fig: reguBIS}. By doing so, we notice that $z$ and $z'$ are within a distance $2L$ of each other, and both lie in $\Lambda_{n/3+L}$. Hence, if $L\geq L_3$, for such pairs $(y,z)$ and $(y',z')$, by \eqref{eq:proofregclose},
\begin{equation}
\Big|\mathbb P_\beta[z\connect{\mathbb H_n\:}x] -\mathbb P_\beta[z'\connect{\mathbb H_n\:}x] \Big|
\le \frac{\eta}{2}\frac{\bfC}{Ln^{d-1}}
\end{equation}
Plugging this estimate in the difference of \eqref{eq:u1} and \eqref{eq:u2}, and then invoking Lemma~\ref{lem:estimate half-space perco} with $k=6\delta n$ and $n/3$ (to the cost of potentially increasing $L$ again and decreasing $\delta$ so that $6\delta n\leq n/6$), gives 
\begin{align}\notag
\mathbb P_\beta[&u\connect{\mathbb H_n\:}x]-\mathbb P_\beta[v\connect{\mathbb H_n\:}x]\\\notag
&\le \frac{\eta}{2}\frac{\bfC}{Ln^{d-1}}\varphi_\beta(B^+)+\frac{\bfC}{L(n/2)^{d-1}}\sum_{\substack{y \in B^+\\  z\notin B^+\cup{ H(L)}}}\mathbb P_\beta[u\connect{B^+\:}y]p_{yz,\beta}+E_\beta(B^-,\mathbb H_n,v,x) \\
&\le\frac{\bfC}{Ln^{d-1}}\Big[\frac{\eta}{2}\Big(1+\frac{K}{L^d}\Big)+2^{d-1}4(36\delta)^c+\frac{D}{L^{d}}\Big],
\end{align}
where we used \eqref{eq:H_beta-' perco} in the second line, Proposition \ref{prop:bound phi perco} and Lemma \ref{lem: estimate E(x)} in the third line,  and that if ${\rm d}(z,B^-\cup B^+)\leq L(\leq \tfrac{n}{6T}\leq \tfrac{n}{6})$, then
\begin{align}\notag
	\mathbb P_\beta[z\connect{\mathbb H_n\:}x]&\stackrel{\phantom{\eqref{eq:H_beta-' perco}}}\leq \max\{\mathbb P_\beta[w\connect{\mathbb H_n\:}x]: w\in \Lambda_{n/3+L}\}\leq \max\{\mathbb P_\beta[w\connect{\mathbb H_n\:}x]: w\in \Lambda_{n/2}\}\\&\stackrel{\eqref{eq:H_beta-' perco}}\leq \frac{\bfC}{L(n/2)^{d-1}}.
\end{align}
 
The proof follows by choosing $\delta$ small enough and $L_3$ large enough so that
\begin{equation}
	\frac{\eta}{2}\Big(1+\frac{K}{L_3^d}\Big)+2^{d-1}4(36\delta)^c+\frac{D}{L_3^{d}}\leq \eta.
\end{equation}

\begin{figure}[!htb]
	\begin{center}
		\includegraphics{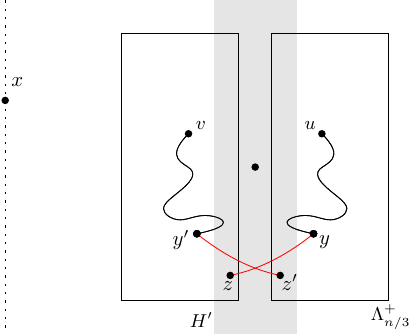}
		\put(-219,10){$\partial\mathbb H_n$}
		\caption{An illustration of the pairing used in the proof of Proposition \ref{prop:regularity perco}. The red path corresponds to a long open edge ``jumping'' outside $\Lambda_{n/3}^+$ (resp. $-\Lambda_{n/3}^+$). The grey region corresponds to $H=H(L)$. Since $u$ (resp.~$v$) is close to $\{u\in \mathbb Z^d : u_1=0\}$, a connection from $u$ (resp.~$v$) to $x$ will likely enter $H'$ if it exits $\Lambda_{n/3}^+$ (resp.~$-\Lambda_{n/3}^+$), as seen in Lemma~\ref{lem:estimate half-space perco}.}
		\label{fig: reguBIS}
	\end{center}
\end{figure}

\paragraph{Case $n\leq\delta^{-1}L$.} The value of $\delta$ is now fixed by the previous case. Here, we will prove that both $\mathbb P_\beta[u\connect{\mathbb H_n\:}x]$ and $\mathbb P_\beta[v\connect{\mathbb H_n\:}x]$ can be bounded efficiently by a random walk estimate (they will end up being both smaller than $\tfrac12 \eta\frac{\bfC}{L^d}\left(L/(L\vee n)\right)^{d-1}$). 

Consider $u\in \Lambda_L$ not equal to $x$. We will choose $\bfC$ large enough, which will force to take $L_3$ even larger. Let $S\geq 1$ to be fixed. Using Lemma \ref{lem: iteration of SL},
\begin{equation}\label{eq:proof improvement l inf perco 1}
	\mathbb P_\beta[u\connect{\mathbb H_n\:}x]\leq (\varphi\vee 1)^S\mathbb E^{\rm RW}_{u}[\mathcal M]+(\varphi\vee 1)^S\mathbb E^{\rm RW}_{u}\Big[\mathbb P_\beta[X_S\connect{\mathbb H_n\:}x]\mathds{1}_{\tau_n^u>S, \: \sigma^u_x>S}\Big],
\end{equation}
where $\mathcal M:=|\{1\leq t\leq S\wedge (\tau_n^u-1) : X_t^u=x\}|$. Proposition \ref{prop: rw estimates 2} applied to $C=\delta^{-1}+1$ gives the existence of $C_{\rm RW}'=C_{\rm RW}'(\delta,d)>0$ such that
\begin{align}
	\mathbb E^{\rm RW}_{u}[\mathcal M]&\leq \frac{C_{\rm RW}'}{L^d}\left(\frac{L}{L\vee n}\right)^{d-1},\label{eq:proof improvement l inf perco 2}\\
	\mathbb P^{\rm RW}_{u}[\tau^u_n>S]&\le \frac{C_{\rm RW}'}{\sqrt{S}}.\label{eq:proof improvement l inf perco 3}
\end{align}
Pick $S$ and then $L$ large enough so that $2C_{\rm RW}'S^{-1/2}\leq \delta^{d-1}\eta/4$ and $(1+K/L^d)^S\leq 2$. In particular, by Proposition \ref{prop:bound phi perco}, one has $(\varphi\vee 1)^S\leq 2$. Plugging \eqref{eq:proof improvement l inf perco 2} and \eqref{eq:proof improvement l inf perco 3} in \eqref{eq:proof improvement l inf perco 1} for this choice of $L$ gives
\begin{align}\notag
	\mathbb P_\beta[u\connect{\mathbb H_n\:}x]&\leq \frac{2C_{\rm RW}'}{L^d}\left(\frac{L}{L\vee n}\right)^{d-1}+\frac{\eta}{4}\frac{\bfC}{L(L/\delta)^{d-1}}
	\\&\leq \frac{2C_{\rm RW}'}{L^d}\left(\frac{L}{L\vee n}\right)^{d-1}+\frac{\eta}{4}\frac{\bfC}{L^d}\left(\frac{L}{L\vee n}\right)^{d-1} ,
\end{align}
where we used \eqref{eq:H_beta-' perco} to argue that $\mathbb P_\beta[X_S\connect{\mathbb H\:}x]\mathds{1}_{\sigma^u_x>S}\leq \tfrac{\bfC}{L^d}$.
Finally, choosing $\bfC$ large enough so that $2C_{\rm RW}'\leq \bfC \eta/4$ gives
\begin{equation}
	\mathbb P_\beta[u\connect{\mathbb H_n\:}x]\leq \frac12\eta\frac{\bfC}{L^d}\left(\frac{L}{L\vee n}\right)^{d-1}
\end{equation}
and concludes the proof.
\end{proof}
%

\subsection{Conclusion}\label{sec:2.3}

We are now in a position to prove the following proposition.

\begin{Prop}\label{prop: final prop proof mainperco thm}
Fix $d>6$. There exist $C,L_0\geq 1$ such that for every $L\geq L_0$,
\begin{align}
\beta_c&\leq 1+\frac{C}{L^d}, \label{eq:bound final beta_c}
\\
\varphi_{\beta_c}(B)&\le 1+\frac{C}{L^d} &\forall B\in\mathcal B,\label{eq:bound final phi block}
\\
E_{\beta_c}(B)&\leq \frac{C}{L^d} &\forall B\in \mathcal B,\label{eq:error amplitude general blocks}
\\
\psi_{\beta_c}(\mathbb H_n)&\le \frac{C}{L} &\forall n\ge 0, \label{eq:bound final psi}
\\
\mathbb P_{\beta_c}[0\connect{}x] &\le \frac{C}{L^d}\left(\frac{L}{L\vee |x|}\right)^{d-2}&\forall x\in (\mathbb Z^d)^*,\label{eq:upper bound below L perco}
\\
\mathbb P_{\beta_c}[0\connect{\mathbb H\:}x]&\le\frac{C}{L^d}\left(\frac{L}{L\vee |x_1|}\right)^{d-1}&\forall x\in \mathbb H^*\label{eq:upper half-spacebound below L perco}.
\end{align} 
\end{Prop}

\begin{proof} Let $\bfC>1$ and $L\geq 1$ to be fixed.
By Proposition~\ref{prop: improved l^1 perco} and Corollary~\ref{prop: improve l inf perco}, we find that if $\bfC>\bfC_1\vee \bfC_4$ and $L\geq L_1\vee L_4$, for every $\beta<\beta^*(\bfC,L)$,
\begin{align}
\psi_\beta(\mathbb H_n)&\le\frac{\bfC}{2L}&\forall n\ge0,\label{eq:finalpropperco1}\\
\mathbb P_\beta[0\connect{\mathbb H_n\:}x]&\le \frac{\bfC}{2L^d}\left(\frac{L}{L\vee n}\right)^{d-1}&\forall n\geq 0, \: \forall x\in (\partial\mathbb H_n)^*.\label{eq:finalpropperco2}
\end{align}
These bounds are still true at $\beta=\beta^*$ thanks to the monotone convergence theorem. We will use \eqref{eq:finalpropperco1}--\eqref{eq:finalpropperco2} to show by contradiction that, for this choice of $\bfC$ and $L$, one has $\beta^*(\bfC,L)=2\wedge \beta_c$.

Let us now assume by contradiction that $\beta^*<2\wedge\beta_c$. Consider $\beta^{**}\in (\beta^*,2\wedge \beta_c)$. Exponential decay of correlations when $\beta<\beta_c$ \cite{Mensikov1986coincidence,AizenmanBarsky1987sharpnessPerco} (see also \cite{DuminilTassionNewProofSharpness2016}) implies the existence of $c,C>0$ (which depend on $\beta^{**}$) such that, for every $n\geq 0$ and every  $x\in \mathbb \partial \mathbb H_n$,
\begin{equation}\label{eq:finalpropperco3}
	\mathbb P_{\beta^{**}}[0\connect{\mathbb H_n\:}x]\leq \mathbb P_{\beta^{**}}[0\connect{}x]\leq Ce^{-c|x|}.
\end{equation}
By monotonicity of the quantities of interest in $\beta$ and \eqref{eq:finalpropperco3}, there exists $N=N(\beta^{**},\bfC,L)$ such that, for all $\beta\leq \beta^{**}$, 
\begin{align}
\psi_{\beta}(\mathbb H_n)=\sum_{x\in (\partial\mathbb H_n)^*}\mathbb P_{\beta}[0\connect{\mathbb H_n\:}x]&<\frac{\bfC}{L}&\forall n\ge N,\label{eq:finalpropperco4}\\
\mathbb P_{\beta}[0\connect{\mathbb H_n\:}x]&< \frac{\bfC}{L^d}\left(\frac{L}{L\vee n}\right)^{d-1}&\forall n\geq 0,\forall x\in \partial\mathbb H_n\setminus \Lambda_N.\label{eq:finalpropperco5}
\end{align}
Now, the continuity (below $\beta_c$) of the maps $\beta\mapsto \max\{\varphi_\beta(\mathbb H_n): 0\leq n \leq N-1\}$ and $\beta\mapsto\max\{\mathbb P_\beta[0\connect{\mathbb H_n\:}x]: x\in \Lambda_N\cap \partial\mathbb H_n, \: n\geq 0\}$ together with \eqref{eq:finalpropperco1}--\eqref{eq:finalpropperco2} gives the existence of $\beta\in(\beta^*,\beta^{**})$ such that
\begin{align}
	\psi_{\beta}(\mathbb H_n)=\sum_{x\in (\partial\mathbb H_n)^*}\mathbb P_{\beta}[0\connect{\mathbb H_n\:}x]&<\frac{\bfC}{L}&\forall n\le N-1,\label{eq:finalpropperco6}\\
\mathbb P_{\beta}[0\connect{\mathbb H_n\:}x]&< \frac{\bfC}{L^d}\left(\frac{L}{L\vee n}\right)^{d-1}&\forall n\geq 0,\forall x\in \partial\mathbb H_n\cap \Lambda_N.\label{eq:finalpropperco7}
\end{align}
Combining \eqref{eq:finalpropperco4}--\eqref{eq:finalpropperco7}, we obtained the existence of $\beta\in (\beta^*,2\wedge \beta_c)$ such that \eqref{eq:H_beta perco} and \eqref{eq:H_beta-' perco} hold. This contradicts the definition of $\beta^*$. 

 From all of this, we obtain that $\beta^*=2\wedge\beta_c$ and that in addition the properties hold until $2 \wedge\beta_c$.  Proposition~\ref{prop:bound phi perco} implies the right bound on $\varphi_\beta(B)$ and Lemma \ref{lem: error amplitude perco} implies the right bound on $E_\beta(B)$, for every $B\in \mathcal B$, and every $\beta<\beta^*$. A new application of the monotone convergence theorem implies the bounds at $\beta^*$. 

It remains to show that $\beta_c<2$ or simply that $\beta^*<2$ for $L$ large enough. To do so, notice that 
\begin{equation}\varphi_{\beta^*}(\{0\})=|\Lambda_L^*|(1-e^{-\beta^*c_L})=c_L^{-1}(1-e^{-\beta^*c_L})\le 1+\frac{K}{L^d}
\end{equation} 
implies that,

\begin{equation}
	\beta^*(1-\frac{\beta^*c_L}{2}+o(c_L))\leq 1+\frac{K}{L^d},
\end{equation}
where $o(c_L)/c_L$ tends to $0$ as $L$ tends to infinity.  In particular, since $\beta^*\leq 2$, if $L$ is large enough, one has
\begin{equation}
	(1-2c_L)\beta^*\leq 1+\frac{K}{L^d}.
\end{equation}
Since $K\geq 1$, we find that for $L$ large enough, $\beta^*\leq 1+\tfrac{4K}{L^d}$. The proof follows by choosing $L$ so large that $\tfrac{4K}{L^d}<1$.
\end{proof}

\begin{Rem} Proposition \ref{prop: final prop proof mainperco thm} also gives the following result which is of independent interest: if $d>6$, there exist $c,C>0$ and $L_0\geq 1$ such that, for every $\beta\leq \beta_c$, and every $n\geq 1$,
\begin{equation}
	\varphi_\beta(\mathbb H_n)\leq Ce^{-cn/L_\beta}.
\end{equation}
We give a proof of this fact. Let $C,L_0$ be given by Proposition \ref{prop: final prop proof mainperco thm} and $\beta\leq \beta_c$. In view of \eqref{eq:bound final phi block}, it is sufficient to treat the case $n\geq 10L_\beta$. Let $k:=\lfloor n/L_\beta\rfloor$. Iterating $k$ times Lemma \eqref{Lem: SL upper bound} with $S$ being translates of $\Lambda_{L_\beta}$, $\Lambda=\mathbb H_n$, $x\in \mathbb H_n$ and $y\notin \mathbb H_n$ gives
\begin{multline}
	\mathbb P_\beta[0\connect{\mathbb H_n\:}x]p_{xy,\beta}\leq \sum_{\substack{u_1\in \Lambda_{L_\beta}\\v_1\notin \Lambda_{L_\beta}}}\mathbb P_\beta[0\connect{\Lambda_{L_\beta}\:}u_1]p_{u_1v_1,\beta}\\\ldots\sum_{\substack{u_k\in \Lambda_{L_\beta}(v_{k-1})\\v_k\notin \Lambda_{L_\beta}(v_{k-1})}}\mathbb P_\beta[v_{k-1}\connect{\Lambda_{L_\beta(v_{k-1})}\:}u_k]p_{u_kv_k,\beta} \mathbb P_\beta[v_k\connect{\mathbb H_n\:}x]p_{xy,\beta}.
	\end{multline}
	Summing the above display over $x\in \mathbb H_n$ and $y\notin \mathbb H_n$ yields
\begin{equation}
	\varphi_\beta(\mathbb H_n)\leq \varphi_\beta(\Lambda_{L_\beta})^k \sup_{\ell\geq 0}\varphi_\beta(\mathbb H_\ell),
\end{equation}
which yields the result using $\varphi_\beta(L_\beta)\leq e^{-2}$, the definition of $k$, and \eqref{eq:bound final phi block} to argue that $\sup_{\ell\geq 0}\varphi_\beta(\mathbb H_\ell)\leq 1+\frac{C}{L^d}\leq 2$ for $L$ large enough.
\end{Rem}

We conclude this section by proving Theorem~\ref{thm:mainperco}.

\begin{proof}[Proof of Theorem~\textup{\ref{thm:mainperco}}] Fix $d>6$. Let $C,L_0$ be given by Proposition \ref{prop: final prop proof mainperco thm} and let $L\geq L_0$.
Let $A=A(d)\geq 1$ be such that 
\begin{equation}\label{eq:def A perco}
t^{d-1}\leq Ae^{t}\qquad   \forall t\geq 1.
\end{equation}
\paragraph{Proof of \eqref{eq:bound full plane perco}.}
First, if $\beta\leq \beta_c$,
Proposition \ref{prop: final prop proof mainperco thm} implies that for $x\in \mathbb Z^d\setminus \{0\}$ with $|x|\le L_\beta$,
\begin{equation}\label{eq: proof nearcritical full space 0 perco}
	\mathbb P_\beta[0\connect{}x]\leq \frac{C}{L^d}\left(\frac{L}{L\vee |x|}\right)^{d-2}\leq \frac{e^2C}{L^d}\left(\frac{L}{L\vee |x|}\right)^{d-2}e^{-2|x|/L_\beta}.
\end{equation}
We now turn to the case of $\beta<\beta_c$ with $L_\beta\geq L$ and $|x|>L_\beta$ (recall that $L_{\beta_c}=\infty$). Iterating \eqref{eq:SLperco} $k:=\lfloor \tfrac{|x|}{L_\beta}\rfloor-2$ times with $S$ being translates of $\Lambda_{L_\beta}$ and $\Lambda=\mathbb Z^d$, we get (recall that by definition of $L_\beta$, one has $\varphi_\beta(\Lambda_{L_\beta})\leq e^{-2}$) 
\begin{align}\label{eq:iterateperco}
\mathbb P_\beta[0\connect{}x]&\stackrel{\phantom{\eqref{eq:upper bound below L perco}}}\le \varphi_\beta(\Lambda_{L_\beta})^k \max\big\{\mathbb P_\beta[y\connect{}x]:y\notin\Lambda_{L_\beta}(x)\big\}\\&\stackrel{\eqref{eq:upper bound below L perco}}\leq e^{-2k}\frac{C}{L^2L_\beta^{d-2}}\stackrel{\phantom{\eqref{eq:upper bound below L perco}}}\leq \frac{e^6C}{L^2L_\beta^{d-2}}e^{-2|x|/L_\beta}
\stackrel{\eqref{eq:def A perco}}\leq \frac{e^6AC}{L^2|x|^{d-2}}e^{-|x|/L_\beta}.\label{eq: proof nearcritical full space perco}
\end{align}
\paragraph{Proof of \eqref{eq:bound half plane perco}.} Again, if $\beta\leq \beta_c$, Proposition \ref{prop: final prop proof mainperco thm} implies that for $x\in \mathbb H^*$ with $x_1\leq L_\beta$,
\begin{equation}\label{eq: proof nearcritical half space 0 perco}
	\mathbb P_\beta[0\connect{\mathbb H\:}x]\leq \frac{C}{L^d}\left(\frac{L}{L\vee |x_1|}\right)^{d-1}\leq \frac{e^2C}{L^d}\left(\frac{L}{L\vee |x_1|}\right)^{d-1}e^{-2|x_1|/L_\beta}.
\end{equation}
Turning to the case of $\beta<\beta_c$ with $L_\beta\geq L$ and $|x_1|>L_\beta$, iterating \eqref{eq:SLperco} $\ell:=\lfloor |x_1|/L_\beta\rfloor-2$ times with $S$ being translates of $\Lambda_{L_\beta}$ (starting with $\Lambda_{L_\beta}(x)$) and $\Lambda=\mathbb H$, we obtain
\begin{align}
	\mathbb P_\beta[0\connect{\mathbb H\:}x]=\mathbb P_\beta[x\connect{\mathbb H\:}0]\leq \varphi_\beta(\Lambda_{L_\beta})^{\ell}\max\{\mathbb P_\beta[0\connect{\mathbb H\:}y: y_1> L_\beta\}&\stackrel{\eqref{eq:upper half-spacebound below L perco}}\leq \frac{e^6C}{LL_\beta^{d-1}}e^{-2|x_1|/L_\beta}
	\\&\stackrel{\eqref{eq:def A perco}}\leq \frac{e^6AC}{L |x_1|^{d-1}}e^{-|x_1|/L_\beta} \label{eq: proof nearcritical half space perco}
\end{align}
Gathering \eqref{eq: proof nearcritical full space 0 perco}--\eqref{eq: proof nearcritical half space 0 perco} and \eqref{eq: proof nearcritical half space perco}, we obtained: for all $\beta\leq \beta_c$ such that $L_\beta\geq L$, for all $x\in \mathbb Z^d\setminus \{0\}$,
\begin{align}
\mathbb P_\beta[0\connect{}x]&\le \frac{e^6AC}{L^d}\left(\frac{L}{L\vee |x|}\right)^{d-2}\exp(-|x|/L_\beta),\\
\mathbb P_\beta[0\connect{\mathbb H\:}x]&\le \frac{e^6AC}{L^d}\left(\frac{L}{L\vee |x_1|}\right)^{d-1}\exp(-|x_1|/L_\beta).
\end{align}
This concludes the proof.
\end{proof}

\section{Lower bounds}\label{sec:lowerboundsperco}
This section is devoted to the proof of Theorem \ref{thm:main2perco}. We follow the strategy of \cite{DumPan24WSAW}. By definition, when $n\leq L_\beta-1$, one has $\varphi_\beta(\Lambda_n)\geq e^{-2}$. The first step will be to turn this inequality into a bound of the form $\psi_\beta(\mathbb H_n)\geq c/L$, see Lemma \ref{lem: petit lemme}. This bound can be viewed as an averaged lower bound on the half-space two-point function. In order to turn it into a point-wise lower bound, we will use an improved regularity estimate (compared to Proposition \ref{prop:regularity perco})--- see Proposition \ref{prop:regularity perco 2}--- which can be viewed as a Harnack-type inequality. The proof of this result follows the exact same lines as in \cite{DumPan24WSAW}. The lower bound on the full-space two-point function will follow from the half-space lower bound together with an application of Lemma \ref{Lem: SL lower bound}.

 In this section, we fix $d>6$, and let $C$ and $L_0$ be given by Proposition \ref{prop: final prop proof mainperco thm}. 
 
\subsection{Preliminaries}
We begin the section by stating a minor generalisation of Proposition \ref{prop:regularity perco}, which will be useful in the proof of the Harnack-type inequality.
\begin{Prop}[Regularity estimate at mesoscopic scales]\label{prop:regularitygeneralperco} For every $\eta>0$, there exist $\delta=\delta(\eta,d)\in(0,1/2)$, and $L_5=L_5(\eta,d)\geq L_0$ such that for every $L\geq L_5$, every $\beta\leq \beta_c$, every $n\geq \delta^{-1}L$, every $\Lambda\supset \Lambda_{3n}$, every $X\subset \Lambda\setminus\Lambda_{3n}$, and every $u,v\in \Lambda_{\lfloor \delta n\rfloor}$,
\begin{multline}\label{eq:reg mesoscopic boosted}
\left|\sum_{x\in X}\Big(\mathbb P_\beta[u\connect{\Lambda\:}x]-\mathbb P_\beta[v\connect{\Lambda\:}x]\Big)\right|\le \eta\max_{w\in \Lambda_{3n}}
\sum_{x\in X}\mathbb P_\beta[w\connect{\Lambda\:}x]\\+\delta^{-1}\max_{w\in \Lambda_{3n}}\max_{\substack{S\in \mathcal B\\S+w\subset \Lambda_{3n}}}\sum_{x\in X}E_\beta(S+w,\Lambda,w,x).
\end{multline}
Moreover, if $u,v \in \Lambda_n$ are such that $|u-v|\leq 2L$, one has
\begin{equation}\label{eq:uv close in the statement}
	\left|\sum_{x\in X}\Big(\mathbb P_\beta[u\connect{\Lambda\:}x]-\mathbb P_\beta[v\connect{\Lambda\:}x]\Big)\right|\le \eta\max_{w\in \Lambda_{3n}}
\sum_{x\in X}\mathbb P_\beta[w\connect{\Lambda\:}x]\\+\delta^{-1}\max_{w\in \Lambda_{3n}}\sum_{x\in X}E_\beta(\{w\},\Lambda,w,x).
\end{equation}
\end{Prop}
The proof follows the same strategy as Proposition \ref{prop:regularity perco} except that we keep track of the form of the term on the right-hand side of \eqref{eq:reg mesoscopic boosted}. 

\begin{proof} Fix $\eta>0$. We let $\delta=\delta(\eta,d)>0$ to be chosen small enough later. Also, let $L\geq L_0$  and $n\geq 1$ to be taken large enough, and fix $\beta\leq\beta_c$. For simplicity, below, we omit integer rounding.

We prove the result for $X=\{x\}$, but the general case follows similarly. Set $\varphi:=\varphi_\beta(\{0\})$. Let $T\geq 1$ to be fixed, $n\geq 6TL$, and first assume that $u,v\in \Lambda_{n}$ with $|u-v|\leq 2L$. Using a minor generalisation of Lemma \ref{lem: iteration of SL} (with $\mathbb H_n$ replaced by $\Lambda$) gives
\begin{align}\label{eq:r1perco}
	\mathbb P_\beta[u\connect{\Lambda\:}x]&\leq \varphi^T\mathbb E_{u}^{\rm RW}\Big[\mathbb P_\beta[X_T\connect{\Lambda\:}x]\Big],\\	\mathbb P_\beta[v\connect{\Lambda\:}x]&\geq \varphi^T\mathbb E_{v}^{\rm RW}\Big[\mathbb P_\beta[X_T\connect{\Lambda\:}x]\Big]-\Big(\sum_{t=0}^{T-1}\varphi^t\Big)\max_{w\in \Lambda_{n+TL}}E_\beta(\{w\},\Lambda,w,x).\label{eq:r2perco}
\end{align}

As in the proof of Proposition \ref{prop:regularity perco}, Proposition \ref{prop:coupling appendix} implies that there exists $T=T(\eta/4,d)$ large enough such that the random walks $X_T^u$ and $X_T^v$ can be coupled to coincide with probability larger than $1-\eta/4$. This observation yields
\begin{equation}\label{eq: coupling micro scales perco}
\mathbb E_{u}^{\rm RW}\Big[\mathbb P_\beta[X_T\connect{\Lambda\:}x]\Big]-\mathbb E_{v}^{\rm RW}\Big[\mathbb P_\beta[X_T\connect{\Lambda\:}x]\Big]\le \frac{\eta}{4}\max\Big\{\mathbb P_\beta[w\connect{\Lambda\:}x]:w\in \Lambda_{n+TL}\Big\}.
\end{equation}
 Furthermore, since $\varphi\le 1+\tfrac{C}{L^d}$, we may choose $L_5=L_5(\eta,T,C,d)\geq L_0$ large enough such that for $L\geq L_5$, one has $\varphi^T\leq 2$. As a result, we obtained, for every $n\geq 6TL$, and every $u,v\in \Lambda_n$ with $|u-v|\leq 2L$,
 \begin{multline}\label{eq: uv close lower bounds}
 	\Big|\mathbb P_\beta[u\connect{\Lambda\:}x]-\mathbb P_\beta[v\connect{\Lambda\:}x]\Big|\leq \frac{\eta}{4}\max\Big\{\mathbb P_\beta[w\connect{\Lambda\:}x]:w\in \Lambda_{n+TL}\Big\}\\+2T\max_{w\in \Lambda_{n+TL}}E_\beta(\{w\},\Lambda,w,x).
 \end{multline}
 This gives \eqref{eq:uv close in the statement} as soon as $\delta^{-1}\geq 6T$.
 
 We now assume that $n\geq \delta^{-1}L$, and let $u,v\in \Lambda_{\delta n}$. We prove the result when $u$ and $v$ differ by one coordinate only and assume that the difference of the coordinate is even as the latter assumption can be relaxed by using the previous estimate when $u$ and $v$ are at distance 1 (which is less than $2L$), and the former by summing increments over coordinates. Again, by rotating and translating, we may therefore consider $u=k\mathbf{e_1}$ and $v=-k\mathbf{e}_1$ with $k\leq \delta n$. Recall that for $\ell\geq 1$, $\Lambda_\ell^+=\{x\in \Lambda_\ell : x_1>0\}$, and that $B^+:=\Lambda_{n/3}^+$ and $B^-:=-\Lambda_{n/3}^+$. By Lemmata \ref{Lem: SL upper bound} and \ref{Lem: SL lower bound},
\begin{align}
\mathbb P_\beta[u\connect{\Lambda\:}x] &\le \sum_{\substack{y \in B^+\\ z\notin B^+}}\mathbb P_\beta[u\connect{B^+\:}y]p_{yz,\beta}  \mathbb P_\beta[z\connect{\Lambda\:}x]\label{eq:u1 bis},\\
\mathbb P_\beta[v\connect{\Lambda\:}x]&\ge \sum_{\substack{y \in B^-\\ z\notin B^-}}\mathbb P_\beta[v\connect{B^-\:}y]p_{yz,\beta} \mathbb P_\beta[z\connect{\Lambda\:}x] -E_\beta(B^-,\Lambda,v,x).\label{eq:u2 bis}
\end{align}
Now, when $z\in H'(L)=\{v\in \mathbb Z^d: |v_2|\leq L\}$, we may associate every pair $(y,z)$ in the sum in~\eqref{eq:u1 bis} with the pair $(y',z')$ symmetric with respect to the hyperplane $\{u\in \mathbb Z^d:u_2=0\}$ in the sum in~\eqref{eq:u2 bis}. Then, $z$ and $z'$ are within a distance $2L$ of each other. Hence, if $L\geq L_5$, for such pairs $(y,z)$ and $(y',z')$, the difference $\mathbb P_\beta[z\connect{\Lambda\:}x]-\mathbb P_\beta[z'\connect{\Lambda\:}x]$ can be bounded thanks to \eqref{eq: uv close lower bounds}.

Plugging this estimate in the difference of \eqref{eq:u1 bis} and \eqref{eq:u2 bis}, and then invoking Lemma~\ref{lem:estimate half-space perco} with $k=6\delta n\:(\geq L)$ and $n/3$ (to the cost of potentially increasing $L$ again and decreasing $\delta$ so that $6\delta n\leq n/6$), gives 
\begin{multline}\label{eq: last equation uv lower bound}
\mathbb P_\beta[u\connect{\Lambda\:}x]-\mathbb P_\beta[v\connect{\Lambda\:}x]\le \left(\frac{\eta}{2}+4(36\delta)^c\right)\max\Big\{\mathbb P_\beta[w\connect{\Lambda\:}x]:w\in \Lambda_{n+TL}\Big\}\\+4T\max_{w\in \Lambda_{n+TL}}E_\beta(\{w\},\Lambda,w,x)+E_\beta(B^-,\Lambda,v,x) ,
\end{multline}
where we used Proposition \ref{prop: final prop proof mainperco thm} to argue that $\varphi(B^+)\leq 1+\frac{C}{L^d}\leq 2$. We now choose $\delta$ small enough so that $4(36\delta)^c\leq \eta/2$ and $\delta^{-1}\geq 6T$. Plugging this choice in \eqref{eq: last equation uv lower bound} gives, for every $u,v\in \Lambda_{\delta n}$,
\begin{multline}
	\mathbb P_\beta[u\connect{\Lambda\:}x]-\mathbb P_\beta[v\connect{\Lambda\:}x]\le \eta \max\Big\{\mathbb P_\beta[w\connect{\Lambda\:}x]:w\in \Lambda_{3n}\Big\}\\+ \delta^{-1}\max_{w\in \Lambda_{3n}}\max_{\substack{S\in \mathcal B\\S+w\subset \Lambda_{3n}}}E_\beta(S+w,\Lambda,w,x),
\end{multline}
and concludes the proof.
\end{proof}

We conclude these preliminaries with an easy estimate, whose proof is postponed to Appendix \ref{appendix:boundEaverage}. The first part of the statement can be viewed as an $\ell^1$ version of Lemma \ref{lem: estimate E(x)}.

\begin{Lem}\label{lem: bound non local error singleton perco} There exists $D_1>0$ such that, for every $L\geq L_0$, and every $\beta\leq \beta_c$, 
\begin{equation}\label{eq:non local term 1}
	\max_{w\in \mathbb H_n}\sum_{x\in \partial \mathbb H_n}{E}_\beta(\{w\},\mathbb H_n,w,x)\leq \frac{D_1}{L^d}, \qquad \forall n\geq 0,
\end{equation}
and 
\begin{equation}\label{eq:non local term 2}
	\max_{x\in \partial \mathbb H_n}\max_{w\in \mathbb H_n\setminus \{x\}}E_\beta(\{w\},\mathbb H_n,w,x)\leq \frac{D_1}{L^{2d}}.
\end{equation}
\end{Lem}

\subsection{A Harnack-type estimate}\label{sec: harnack perco}

As explained in the beginning of the section, we will need to turn averaged estimates into point-wise ones. We therefore show another regularity estimate which can be viewed as a Harnack-type estimate.

It will be convenient to replace $L_\beta$ by the quantity $L_\beta(\varepsilon)$ defined as follows: for $\varepsilon\in [0,1]$, \begin{equation}
L_\beta(\varepsilon):=\inf\{n\ge1:\varphi_\beta(\Lambda_n)\le 1-\varepsilon\}.
\end{equation}
In particular, one has $L_\beta=L_\beta(1-e^{-2})$. In fact, $L_\beta(\varepsilon)$ only differs from $L_\beta$ by a constant, as stated in the next lemma. This result already appeared in the context of the WSAW model in our companion paper, see \cite[Lemma~3.1]{DumPan24WSAW}.
\begin{Lem}\label{lem: l_beta differs by constant} Let $d>6$ and $\varepsilon\in (0,1-e^{-2})$. Recall that $L_0$ is given by Proposition \textup{\ref{prop: final prop proof mainperco thm}}. There exists $C_0=C_0(\varepsilon,L_0,d)>0$ such that, for every $L\geq L_0$, and every $\beta<\beta_c$,
\begin{equation}\label{eq: comparison diff l beta}
	L_\beta(\varepsilon)\leq L_\beta\leq C_0(L_\beta(\varepsilon)\vee L).
\end{equation}
\end{Lem}
\begin{proof} Let $L\geq L_0$ and $\beta<\beta_c$. Observe that the map $\varepsilon\mapsto L_\beta(\varepsilon)$ is increasing. Combined with the fact that $L_\beta=L_\beta(1-e^{-2})$, this gives the first inequality in \eqref{eq: comparison diff l beta}.

We turn to the second one. Let $k\geq 1$ to be chosen large enough. Assume by contradiction that $L_\beta>k(L_\beta(\varepsilon)+L)$. We iterate Lemma \ref{Lem: SL upper bound} $k$ times with $S$ being translates of $\Lambda_{L_\beta(\varepsilon)}$ and $\Lambda=\Lambda_{L_\beta-1}$ to obtain that: for every $x\in \Lambda$ and $y\notin \Lambda$,
\begin{multline}\label{eq:iteration SL}
	\mathbb P_\beta[0\connect{\Lambda\:}x]p_{xy,\beta}\leq \sum_{\substack{u_1\in \Lambda_{L_\beta(\varepsilon)}\\v_1\notin \Lambda_{L_\beta(\varepsilon)}}}\mathbb P_\beta[0\connect{\Lambda_{L_\beta(\varepsilon)}\:}u_1]p_{u_1v_1,\beta}\\\ldots\sum_{\substack{u_k\in \Lambda_{L_\beta(\varepsilon)}(v_{k-1})\\v_k\notin \Lambda_{L_\beta(\varepsilon)}(v_{k-1})}}\mathbb P_\beta[v_{k-1}\connect{\Lambda_{L_\beta(\varepsilon)(v_{k-1})}\:}u_k]p_{u_kv_k,\beta} \mathbb P_\beta[v_k\connect{\Lambda\:}x]p_{xy,\beta}.
	\end{multline}
Summing the above displayed equation over $x\in \Lambda$ and $y\notin \Lambda$ gives
\begin{equation}
\varphi_\beta(\Lambda_{L_\beta-1})\le \varphi_\beta(\Lambda_{L_\beta(\varepsilon)})^k\max\{\varphi_\beta(\Lambda_{L_\beta-1}(x)):x\in \Lambda_{L_\beta-1}\}\le (1-\varepsilon)^k\Big(1+\frac{C}{L_0^d}\Big),
\end{equation}
where $C$ is given by Proposition \ref{prop: final prop proof mainperco thm}. Now, choose $k=k(\varepsilon,L_0,d)$ large enough so that $(1-\varepsilon)^k(1+C/L_0^d)<e^{-2}$. Since $\varphi_\beta(\Lambda_{L_\beta-1})\geq e^{-2}$, this choice contradicts the above displayed equation. As a consequence,
\begin{equation}
	L_\beta\leq k(L_\beta(\varepsilon)+L)\leq 2k(L_\beta(\varepsilon)\vee L).
\end{equation}
Setting $C_0:=2k$ concludes the proof.
\end{proof}

We now state the main technical estimate of this section.

\begin{Prop}[Regularity estimate at macroscopic scales]\label{prop:regularity perco 2} Let $\alpha>0$. There exists $C_{\rm RW}=C_{\rm RW}(\alpha)>0$, and for every $\eta>0$, there exist $A=A(\alpha,\eta)>0$, $L_6\geq L_0$, and $\varepsilon_0=\varepsilon_0(\eta,\alpha)>0$ such that the following holds. For every $L\geq L_6$, every $\varepsilon<\varepsilon_0$, every $n\le 6L_\beta(\varepsilon)$  with $n\geq AL$, every $\beta\le \beta_c$, every $\Lambda\supset\Lambda_{(1+\alpha)n}$, and every $x\in \Lambda\setminus \Lambda_{(1+\alpha)n}$,
\begin{align}
\max_{w\in \Lambda_n}\mathbb P_\beta[w\connect{\Lambda\:}x]&\le C_{\rm RW}\min_{w\in \Lambda_n}\mathbb P_\beta[w\connect{\Lambda\:}x]+\eta \max_{w\in \Lambda_{(1+\alpha)n}}\mathbb P_\beta[w\connect{\Lambda\:}x]\notag
\\&\quad+A\max_{w\in \Lambda_{(1+\alpha)n}}\max_{\substack{S\in \mathcal B\\S+w\subset \Lambda_{(1+\alpha)n}}}{E}_\beta(S+w, \Lambda,w,x).
\end{align}
\end{Prop}
The idea of the proof is to introduce a well-chosen rescaled random walk which we now define. Let $\beta>0$ and $L\geq 1$. If $m\geq 1$, consider the random walk distribution $(Y_k)_{k\geq 0}$ (started at $u\in \mathbb Z^d$) defined by the following step distribution $\mu=\mu_{m,L,\beta}$: if $v\in \mathbb Z^d$,
\begin{equation}\label{eq: rescaled rw def}
	\mathbb P^{{\rm RW},m}_{u}[Y_1=v]=\mu(v-u):=\frac{\mathds{1}_{v\notin \Lambda_{m-1}(u)}}{\varphi_\beta(\Lambda_{m-1})}\sum_{w\in \Lambda_{m-1}(u)}\mathbb P_\beta[u\connect{\Lambda_{m-1}(u)\:}w]p_{wv,\beta}.
\end{equation}
In particular, when $m=1$, $\mathbb P^{{\rm RW},m}_u=\mathbb P^{\rm RW}_u$. Below, we will restrict to the case $m\geq L$. This case corresponds to the situation where $\mu\in \mathcal P_m$, as described in Appendix \ref{appendix:uniform rw estimates}. It will be important for the argument that follows to use estimates involving the above random walk which are uniform in $m$ and $L$. These estimates follow from classical random walk analysis, and can be found in the Appendix \ref{appendix:uniform rw estimates}.

\begin{proof}[Proof of Proposition \textup{\ref{prop:regularity perco 2}}] Let $\alpha,\eta>0$. Set $M:=\lfloor \tfrac{\alpha}{6} n\rfloor$, and let $C_{\rm RW}$ and $N_1$ be given by Proposition \ref{prop:uniform Harnack} applied to $\tfrac{\alpha}{6}$ and $\tfrac{\eta}{4}$.
Let $\delta$ and $L_5$ be given by Proposition \ref{prop:regularitygeneralperco} with $\frac{\eta}{4C_{\rm RW}}$ and let $L\geq L_5$. To the cost of diminishing $\delta$, we additionally assume that
\begin{equation}\label{eq:take delta smaller}
	\delta\leq (N_1\alpha)^{-1}.
\end{equation}
Recall the random walk distribution $\mathbb P^{{\rm RW},m}_u$ introduced above. Let $\tau$ be the hitting time of $\mathbb Z^d\setminus\Lambda_{n+M}$. If $m=\lfloor (\delta\alpha/36)^{10/9}n\rfloor$, then $n/m\geq N_1$ by \eqref{eq:take delta smaller}. We further assume that $n\geq (\delta\alpha/36)^{-10/9}(L+1)$ so that $m\geq L$ and Proposition \ref{prop:uniform Harnack} applies. Using this latter result, we get, for every $u,v\in \Lambda_n$,
\begin{multline}\label{eq:proof propreg2 1}
	\mathbb E^{{\rm RW},m}_u\Big[\mathbb P_\beta[X_\tau \connect{\Lambda\:}x]\Big]\leq C_{\rm RW}\mathbb E^{{\rm RW},m}_v\Big[\mathbb P_\beta[X_\tau \connect{\Lambda\:}x]\Big]+\frac{\eta}{4}\max\{\mathbb P_\beta[w \connect{\Lambda\:}x]:w\in \Lambda_{(1+\alpha)n}\}\\+2C_{\rm RW}\max \Big\{|\mathbb P_\beta[w \connect{\Lambda\:}x]-\mathbb P_\beta[w'\connect{\Lambda\:}x]|: w,w'\in \Lambda_{\lfloor \delta M\rfloor}(z), \: z\in \partial \Lambda_{n+M}\Big\} 
\end{multline}
where we used that, for this choice of $M,m$ and for $n$ large enough, one has: 
\begin{align}
3m(n/m)^{1/10}&\leq \frac{\delta \alpha}{12}n\leq \frac{\delta\alpha}{6}n-\delta-1\leq \lfloor\delta M\rfloor,\\
n+M+2m&\leq (1+\alpha)n.\end{align} 
Combining \eqref{eq:proof propreg2 1} with Corollary \ref{cor: estimate srw} (applied to $\frac{\eta}{8C_{\rm RW}}$ and $A\approx (\alpha\delta)^{-10/9}$) provides $T=T(\eta/(8C_{\rm RW}),\alpha,d)>0$ and $\varepsilon_0=\varepsilon_0(T)$ such that $(1+\varepsilon_0)^T\leq 2$, and for every $\varphi\in [1-\varepsilon_0,1+\varepsilon_0]$,
\begin{multline}\label{eq:proof propreg2 1.5}
	\mathbb E_u^{{\rm RW},m}\Big[\varphi^{\tau\wedge T}\mathbb P_\beta[X_{\tau\wedge T}\connect{\Lambda\:}x]\Big]\leq C_{\rm RW}\mathbb E_v^{{\rm RW},m}\Big[\varphi^{\tau\wedge T}\mathbb P_\beta[X_{\tau\wedge T}\connect{\Lambda\:}x]\Big]\\+\frac{\eta}{2}\max\{\mathbb P_\beta[w\connect{\Lambda\:}x]:w\in \Lambda_{(1+\alpha)n}\}\\+2C_{\rm RW}\max \Big\{|\mathbb P_\beta[w\connect{\Lambda\:}x]-\mathbb P_\beta[w'\connect{\Lambda\:}x]|: w,w'\in \Lambda_{\lfloor \delta M\rfloor}(z), \: z\in \partial \Lambda_{n+M}\Big\}.
\end{multline}
Thanks to Proposition \ref{prop: exponential moment exit time}, for every $u\in \Lambda_n$,
\begin{equation}\label{eq:proof propreg2 2}
	\mathbb E_u^{{\rm RW},m}[\tau]\leq C_1(1+\alpha)^2\delta^{-20/9},
\end{equation}
for some $C_1=C_1(d)>0$.
%
We are now in a position to prove the desired result.

Let $L\geq L_5$ and $\varepsilon<\varepsilon_0$. By choosing $L$ large enough we may additionally assume that $C/L^d\leq \varepsilon_0$ where we recall that $C$ is given by Proposition \ref{prop: final prop proof mainperco thm}. Fix $u,v\in \Lambda_n$. By definition of $m$ and \eqref{eq:take delta smaller}, we observe that $6m<n$. Since we also have $n\le 6L_\beta(\varepsilon)$, Proposition~\ref{prop: final prop proof mainperco thm} gives
\begin{equation}
1-\varepsilon\le \varphi_\beta(\Lambda_{m-1})\le 1+\frac{C}{L^d}\le 1+\varepsilon_0.
\end{equation}
We now set $\varphi:=\varphi_\beta(\Lambda_{m-1})$. Iterating the two bounds of Lemmata \ref{Lem: SL upper bound} and \ref{Lem: SL lower bound} with $S$ being translates of $\Lambda_{m-1}$, and until time $\tau\wedge T$ gives \begin{align}
\mathbb P_\beta[u\connect{\Lambda\:}x]&\le  \mathbb E_{u}^{{\rm RW},m}\Big[\varphi^{\tau\wedge T} \mathbb P_\beta[Y_{\tau\wedge T}\connect{\Lambda\:}x]\Big],\label{eq:hea perco}\\
\mathbb P_\beta[v\connect{\Lambda\:}x]&\ge \mathbb E_{v}^{{\rm RW},m}\Big[\varphi^{\tau\wedge T} \mathbb P_\beta[Y_{\tau\wedge T}\connect{\Lambda\:}x]\Big]-\mathbb E_{v}^{{\rm RW},m}\left[\sum_{t=0}^{\tau\wedge T-1}\varphi^t{E}_\beta(\Lambda_{m-1}(Y_t),\Lambda,Y_t,x)\right].\label{eq:hda perco}
\end{align}
Using that $(1+\varepsilon_0)^T\leq 2$ and \eqref{eq:proof propreg2 2}, we get
\begin{multline}\label{eq:proof propref2 2.5}
	\mathbb E_{v}^{{\rm RW},m}\left[\sum_{t=0}^{\tau\wedge T-1}\varphi^t{E}_\beta(\Lambda_{m-1}(Y_t),\Lambda,Y_t,x)\right]\\\leq 2C_1(1+\alpha)^2\delta^{-20/9}\max_{w\in \Lambda_{(1+\alpha)n}}\max_{\substack{S\in \mathcal B\\S+w\subset \Lambda_{(1+\alpha)n}}}{E}_\beta(S+w, \Lambda,w,x).
\end{multline}
Combining \eqref{eq:hea perco} and \eqref{eq:proof propreg2 1.5} gives
\begin{multline}\label{eq:proof propreg2 3}
	\mathbb P_\beta[u\connect{\Lambda\:}x]\leq C_{\rm RW}\mathbb E_v^{{\rm RW},m}\Big[\varphi^{\tau\wedge T}\mathbb P_\beta[Y_{\tau \wedge T}\connect{\Lambda\:}x]\Big]+\frac{\eta}{2}\max\{\mathbb P_\beta[w\connect{\Lambda\:}x]:w\in \Lambda_{(1+\alpha)n}\}\\+2C_{\rm RW}\max \Big\{|\mathbb P_\beta[w\connect{\Lambda\:}x]-\mathbb P_\beta[w'\connect{\Lambda\:}x]|: w,w'\in \Lambda_{\lfloor \delta M\rfloor}(z), \: z\in \partial \Lambda_{n+M}\Big\}. 
\end{multline}
Assume that $M\geq \delta^{-1}L$, which occurs as soon as $n\geq 6\alpha^{-1}\delta^{-1}(L+1)$. By Proposition \ref{prop:regularitygeneralperco} (recall that it is applied to $\tfrac{\eta}{4C_{\rm RW}}$),
\begin{multline}\label{eq:proof propreg2 4}
	2C_{\rm RW}\max \Big\{|\mathbb P_\beta[w\connect{\Lambda\:}x]-\mathbb P_\beta[w'\connect{\Lambda\:}x]|: w,w'\in \Lambda_{\lfloor \delta M\rfloor}(z), \: z\in \partial \Lambda_{n+M}\Big\}\\\leq \frac{\eta}{2}\max\{\mathbb P_\beta[w\connect{\Lambda\:}x]: w \in \Lambda_{n+3M}\}+2C_{\rm RW}\delta^{-1}\max_{w\in \Lambda_{(1+\alpha)n}}\max_{\substack{S\in \mathcal B\\S+w\subset \Lambda_{(1+\alpha)n}}}{E}_\beta(S+w, \Lambda,w,x).
\end{multline}
Set $B:=2C_{\rm RW}(C_1(1+\alpha)^2\delta^{-20/9}+\delta^{-1})$. 
Plugging \eqref{eq:hda perco}, \eqref{eq:proof propref2 2.5}, and \eqref{eq:proof propreg2 4} in \eqref{eq:proof propreg2 3} gives that, for every $n\geq \Big((\delta\alpha/36)^{-10/9}\vee 6\alpha^{-1}\delta^{-1}\Big)(L+1)$,
\begin{multline}
	\mathbb P_\beta[u\connect{\Lambda\:}x]\leq C_{\rm RW}\mathbb P_\beta[v\connect{\Lambda\:}x]+\eta\max\{\mathbb P_\beta[w\connect{\Lambda\:}x]:w\in \Lambda_{(1+\alpha)n}\}
	\\+B\max_{w\in \Lambda_{(1+\alpha)n}}\max_{\substack{S\in \mathcal B\\S+w\subset \Lambda_{(1+\alpha)n}}}{E}_\beta(S+w, \Lambda,w,x),
\end{multline}
where we used that $n+3M\leq (1+\alpha)n$. Since $u$ and $v$ are arbitrary in $\Lambda_n$, the proof follows readily from setting $A:=B\vee (\delta\alpha/36)^{-10/9}\vee 6\alpha^{-1}\delta^{-1}$.
 \end{proof}

\subsection{Lower bound on $\psi_\beta(\mathbb H_n)$}
We now turn to our basic estimate for this section. It is a lower bound corresponding to the upper bound on $\psi_\beta(\mathbb H_n)$ obtained in the previous section. Recall that $\beta_0$ is such that $\varphi_{\beta_0}(\{0\})=1$. Introduce for $n,k\geq 1$,
\begin{equation}
	\psi^{[k]}_\beta(\mathbb H_n):=\sum_{\substack{x\in \partial \mathbb H_n\\|x|\leq k}}\mathbb P_\beta[0\connect{\mathbb H_n\:}x].
\end{equation}
\begin{Lem}\label{lem: petit lemme} There exist $c,A'>0$ and $L_7\geq L_0$ such that for every $L\geq L_7$, every $\beta\le\beta_c$, and every $1\le n\le L_\beta-1$ with $n\geq A'L$,
\begin{equation}
\psi_\beta(\mathbb H_n)\geq\psi_\beta^{[n]}(\mathbb H_n)\ge \frac{c}{L}.
\end{equation}
\end{Lem}

\begin{proof} The first inequality is clear, we therefore focus on the second one. Let $\eta>0$ to be fixed. Let $\delta=\delta(\eta)$ and $L_5=L_5(\eta)$ be given by Proposition \ref{prop:regularitygeneralperco}. Assume that $L\geq L_5$ and that $n> 3\delta^{-1}L$. 
We begin by noticing that, thanks to the BK inequality,
\begin{align}
	\varphi_\beta(\Lambda_n)
	&\leq 2d\sum_{\substack{y\in \Lambda_n\\
	y_1\in \{-n,\ldots,-n+L-1\}\\z\notin \mathbb H_n}}\mathbb P_\beta[0\connect{\Lambda_n\:}y]p_{yz,\beta}\notag
	\\&\leq 2d\sum_{k=0}^{L-1}\sum_{\substack{y\in \partial\mathbb H_{n-k}\\|y|\leq n\\ z\notin \mathbb H_n}}p_{yz,\beta}\sum_{\ell=0}^k\sum_{\substack{u\in \partial\mathbb H_{n-\ell}\\|u|\leq n}}\mathbb P_\beta[0\connect{\mathbb H_{n-\ell}\:}u]\mathbb P_\beta[u\connect{\mathbb H_{n-\ell}\:}y]\label{eq:lemma4.5bound}	\\&\leq 2d\sum_{\ell=0}^{L-1}\sum_{\substack{u\in \partial \mathbb H_{n-\ell}\\|u|\leq n}}\mathbb P_\beta[0\connect{\mathbb H_{n-\ell}\:}u]\sum_{k=\ell}^{L-1}\sum_{\substack{y\in \partial\mathbb H_{n-k}\\z\notin \mathbb H_n}}p_{yz,\beta}\mathbb P_\beta[u\connect{\mathbb H_{n-\ell}\:}y]\notag
	\\&\leq 2d\sum_{\ell=0}^{L-1}\psi_\beta^{[n]}(\mathbb H_{n-\ell})\sum_{k=\ell}^{L-1}\left(\mathds{1}_{k=\ell}+\frac{2dC^2(k+L)}{L^2}\right)\notag
	\\&\leq 2d(1+2dC^2)\sum_{\ell=0}^{L-1}\psi_\beta^{[n]}(\mathbb H_{n-\ell})\label{eq:proof lower bound psi1},
\end{align}
where in \eqref{eq:lemma4.5bound} we decomposed an open path from $0$ to $y$ in $\Lambda_n$ according to the left-most point $u$ it visits (see Figure \ref{fig:perco1}), and where we used \eqref{eq: sum over half plane at distance k from half plane perco} in the fourth inequality. Now, we apply Proposition \ref{prop:regularitygeneralperco} with $\delta^{-1}L$ to obtain that, if $u=0$ and $v=k\mathbf{e}_1$ with $0\leq k \leq L-1$, $\Lambda=\mathbb H_n$, and $X=\{x\in \partial \mathbb H_n: |x|\leq n\}\subset \mathbb H_n\setminus \Lambda_{3\delta^{-1}L}$,
\begin{multline}
	\Big|\sum_{x\in X}\mathbb P_\beta[u\connect{\mathbb H_n\:}x]-\sum_{x\in X}\mathbb P_\beta[v\connect{\mathbb H_n\:}x]\Big|=|\psi^{[n]}_\beta(\mathbb H_n)-\psi^{[n]}_\beta(\mathbb H_{n-k})|\\\leq \eta \max_{1\leq j \leq 3\delta^{-1}L}\psi_\beta(\mathbb H_{n-j})+\delta^{-1}\max_{w\in \Lambda_{3\delta^{-1}L}}\sum_{x\in \partial \mathbb H_n}E_\beta(\{w\},\mathbb H_n,w,x)
\end{multline}
We can apply Lemma \ref{lem: bound non local error singleton perco} to get
\begin{equation}
	\max_{w\in \Lambda_{3\delta^{-1}L}}\sum_{x\in \partial \mathbb H_n}E_\beta(\{w\},\mathbb H_n,w,x)\leq \frac{D_1}{L^d}.
\end{equation}
Combining the two previously displayed equations and using \eqref{eq:bound final psi}, we obtain
\begin{equation}\label{eq:proof lower bound psi2}
\sum_{\ell=0}^{L-1}\psi_\beta^{[n]}(\mathbb H_{n-\ell})\le L\psi_\beta^{[n]}(\mathbb H_n)+\eta C+\delta^{-1}\frac{D_1}{L^{d-1}}.
\end{equation}
Observe that since $n\le L_\beta-1$, one has $\varphi_\beta(\Lambda_n)\ge 1/e^2$. Plugging this and \eqref{eq:proof lower bound psi2} into \eqref{eq:proof lower bound psi1} yields,
\begin{equation}
	\frac{1}{e^2 2d(1+2dC^2)}\leq L\psi_\beta^{[n]}(\mathbb H_n)+\eta C+\delta^{-1}\frac{D_1}{L^{d-1}}.
\end{equation}
We now choose $\eta$ small enough so that $\eta C\leq \tfrac{1}{4}[e^2 2d(1+2dC^2)]^{-1}$ (this only influences how small $\delta$ is, and how large $L_5$ is), and $L_7\geq L_5(\eta)$ large enough so that $\delta^{-1}\frac{D_1}{L_7^{d-1}}\leq \tfrac{1}{4}[e^2 2d(1+2dC^2)]^{-1}$. This concludes the proof by setting $A':=6\delta^{-1}$ and $c:=\tfrac{1}{2}[e^2 2d(1+2dC^2)]^{-1}$.
\end{proof}

 \subsection{Proof of Theorem \ref{thm:main2perco}}

We start by proving a lower bound on the half-space two-point function at scale below $6L_\beta(\varepsilon)$ for a sufficiently small $\varepsilon$. Let 
\begin{equation}
A_n:=\{x\in \mathbb Z^d:x_1=|x|=n\}.
\end{equation}

\begin{Lem}\label{lem:lower below half-space perco} There exist $c,\varepsilon_0>0$ and $L_8\geq L_0$ such that for every $L\geq L_8$, every $\varepsilon<\varepsilon_0$, every $\beta_0\le \beta\le \beta_c$, and every $x\in \mathbb H$ with $x_1=|x|\le 6L_\beta(\varepsilon)$,
\begin{equation}\label{eq:opperco}
\mathbb P_\beta[0\connect{\mathbb H\:}x]\ge \frac{c}{L^d}\left(\frac{L}{L\vee |x|}\right)^{d-1}.
\end{equation}
\end{Lem}

\begin{proof} Assume that $\varepsilon \in [0,1-e^{-2}]$, and $L\geq L_7$ where $L_7$ is given by Lemma \ref{lem: petit lemme}. Let also $A'$ be given by Lemma \ref{lem: petit lemme}.  First, by translation invariance, for every $A'L\leq n\leq L_\beta(\varepsilon)-1$,
 \begin{align}\label{eq:lower average perco}
\frac1{|A_n|}\sum_{y\in A_n}\mathbb P_\beta[0\connect{\mathbb H\:}y]\geq\frac{1}{|A_n|}\psi_\beta^{[n]}(\mathbb H_n)&\ge \frac{c}{Ln^{d-1}},
\end{align}
where $c>0$ is given by Lemma~\ref{lem: petit lemme}, and where we used that $L_\beta\geq L_\beta(\varepsilon)$. We want to turn this averaged estimate into a point-wise one. 

Fix $x\in A_N$ with $N\leq 6L_\beta(\varepsilon)$ and set $n:=\lfloor N/6\rfloor-1\leq L_\beta(\varepsilon)-1$. Let $\eta>0$ to be fixed. Let $C_{\rm RW}, A, L_6, \varepsilon_0>0$ be given by Proposition \ref{prop:regularity perco 2} with $\alpha=\tfrac{1}{12}$, $\eta$, $N$, and $\Lambda=\mathbb H$. We additionally assume $L\geq L_6$ and $\varepsilon<\varepsilon_0$. We consider two cases according to how large $n$ is.
\paragraph{Case $n\geq (A\vee A')L$.} In this case, we can apply Proposition \ref{prop:regularity perco 2} to get, for every $y\in A_n$,
\begin{align}\notag
	\mathbb P_\beta[0\connect{\mathbb H\:}y]&\leq C_{\rm RW} \mathbb P_\beta[0\connect{\mathbb H\:}x]+\eta \max\Big\{\mathbb P_\beta[0\connect{\mathbb H\:}w]: w_1\geq n/2\Big\}
	\\&+A\max\Big\{{E}_\beta(S, \mathbb H,w,0): w_1\geq n/2, \: S\in \mathcal B,\: (S+w)\subset \Lambda_{13N/12}(\tfrac{7N}{6}\mathbf{e}_1)\Big\},\label{eq:how reg 2 is applied}
\end{align}
see Figure \ref{fig:only use alpha} for an illustration.
\begin{figure}
	\begin{center}
		\includegraphics{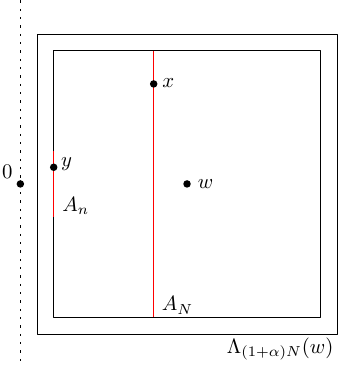}
		\caption{An illustration of how Proposition \ref{prop:regularity perco 2} was applied to obtain \eqref{eq:how reg 2 is applied}. In this picture $w=\frac{7N}{6}\mathbf{e}_1$. The inner box is $\Lambda_N(w)$. The parameter $\alpha$ is chosen in such a way that $\Lambda_{(1+\alpha)N}(w)\subset \{z\in \mathbb H: z_1\geq n/2\}$.}
		\label{fig:only use alpha}
	\end{center}
\end{figure}
Combining the above display with \eqref{eq:lower average perco}, the upper bound from \eqref{eq:upper half-spacebound below L perco}, and Lemma \ref{lem: estimate E(x)} yields
\begin{equation}
	C_{\rm RW}\mathbb P_\beta[0\connect{\mathbb H\:}x]\geq \frac{c}{Ln^{d-1}}-\eta\frac{2^{d-1}C}{Ln^{d-1}}-\frac{AD}{L^{d}}\frac{C}{Ln^{d-1}}.
\end{equation}
Choosing $\eta$ small enough, and then $L$ large enough concludes the proof in that case.
\paragraph{Case $n<(A\vee A')L$.} We now handle the small values of $n$. Fix $x\in \partial\mathbb H_n$ with $|x|\leq n\leq (A\vee A')L$. Set $\varphi:=\varphi_\beta(\{0\})\geq 1$. Let $T\geq 1$ to be fixed. In the spirit of Lemma \ref{lem: iteration of SL}, we iterate $\sigma_x \wedge T$ times \eqref{eq:reversed SL perco} with $S$ being a singleton and $\Lambda=\mathbb H_n$ to obtain,
\begin{equation}\label{eq:lb half 1 small}
	\mathbb P_\beta[0\connect{\mathbb H_n\:}x]\geq \mathbb E^{\rm RW}_{0}[\varphi^{\sigma_x}\mathds{1}_{\sigma_x\leq T,\sigma_x<\tau}]-\mathbb E^{\rm RW}_0\left[\sum_{t=0}^{\sigma_x\wedge T-1}\varphi^t{E}_\beta(\{X_t\},\mathbb H_n,X_t,x)\right],
\end{equation}
where  $\sigma_x$ and $\tau$ are respectively the hitting times of $x$ and the complement of $\mathbb H_n$. Observe that
\begin{align}\notag
	\mathbb E^{\rm RW}_0\left[\sum_{t=0}^{\sigma_x\wedge T-1}\varphi^t{E}_\beta(\{X_t\},\mathbb H_n,X_t,x)\right]&\leq \Big(\sum_{t=0}^{T-1}\varphi^t\Big)\max_{w\in \mathbb H_n\setminus \{x\}}E_\beta(\{w\},\mathbb H_n,w,x)\\&\leq \frac{D_1}{L^{2d}}\Big(\sum_{t=0}^{T-1}\varphi^t\Big),
\end{align}
where we used \eqref{eq:non local term 2} in the second inequality. Since $\beta\geq \beta_0$, we have $\varphi\geq 1$. Now, observe that, since the step distribution of $X$ is uniform over $\Lambda_L^*$, there exists $c_1=c_1(d)>0$ such that
\begin{equation}
	\mathbb P_0^{\rm RW}[\sigma_x\leq T,\sigma_x<\tau \: | \: \{ \exists \ell\leq T: X_{\ell-1}\in \Lambda_{L}(x)\setminus\{x\}, \text{ and } X_k\in \mathbb H_n, \: \forall k\leq \ell-1\}]\geq \frac{c_1}{L^d}.
\end{equation}
As a result,
\begin{equation}
	\mathbb P_0^{\rm RW}[\sigma_x\leq T,\sigma_x<\tau]\geq \frac{c_1}{L^d}\mathbb P_0^{\rm RW}[\exists \ell\leq T: X_{\ell-1}\in \Lambda_{L}(x)\setminus\{x\}, \text{ and } X_k\in \mathbb H_n, \: \forall k\leq \ell-1].
\end{equation}
It remains to show that the probability in the right-hand side of the above display is uniformly (in $L$) bounded from below by a constant $c_2=c_2(A,A',d)>0$ for $T$ large enough. We claim that this follows by a standard random walk computation: since we work with the uniform distribution on $\Lambda^*_L$, it is easy to make the random walk ``drift'' towards $\Lambda_L(x)$ in such a way that each step gets the random walker closer (in infinite norm) to $x$ by a (random) quantity between $L/4$ and $L/2$. Since we assumed that $n\leq (A\vee A')L$ and that $|x|\leq n$, we may find $T=T(A,A',d)>0$ large enough, and $c_2=c_2(A,A',d)>0$ such that 
\begin{equation}
	\mathbb P_0^{\rm RW}[\exists \ell\leq T: X_{\ell-1}\in \Lambda_{L}(x)\setminus\{x\}, \text{ and } X_k\in \mathbb H_n, \: \forall k\leq \ell-1]\geq c_2.
\end{equation}
Plugging the above displayed equations in \eqref{eq:lb half 1 small} yields
\begin{equation}
	\mathbb P_\beta[0\connect{\mathbb H_n\:}x]\geq \frac{c_1c_2}{L^d}-\frac{D_1 T\varphi^T}{L^d}\frac{1}{L^d}.
\end{equation}
Thanks to Proposition \ref{prop: final prop proof mainperco thm}, one has $\varphi\leq 1+\frac{C}{L^d}$. We now choose $L$ large enough so that $\frac{D_1 T \varphi^T}{L^d}\leq \frac{c_1c_2}{2}$ to conclude the proof.
\end{proof}

We now conclude this section with the (quite technical) proof of the lower bound on the full-space two-point function.

\begin{Lem}\label{lem:lower below full-space perco}  There exist $c,\varepsilon_0>0$ and $L_9\geq L_0$ such that for every $L\geq L_9$, every $\varepsilon<\varepsilon_0$, every $\beta_0\le \beta\le \beta_c$, and every $x\in \mathbb Z^d$ with $|x|\leq 5L_\beta(\varepsilon)$,
\begin{equation}\label{eq:lower bound small values full space perco}
\mathbb P_\beta[0\connect{}x]\ge \frac{c}{L^d}\left(\frac{L}{L\vee |x|}\right)^{d-2}.
\end{equation}
\end{Lem}
\begin{proof} We fix $\beta\in [\beta_0,\beta_c]$ and drop it from the notation. The constants $C_i>0$ only depend on $d$. Let $L_8, \varepsilon_0$ be given by Lemma \ref{lem:lower below half-space perco}. We will choose $L_9\geq L_8$ even larger below. Let $L\geq L_8$ and $\varepsilon<\varepsilon_0$. Let $|x|\leq 5L_\beta(\varepsilon)$. By symmetry, we may consider $x\in A_n$, where $n=|x|\geq 1$. 

We begin by assuming that $n\geq 4L$. By \eqref{eq:reversed SL perco} applied to $S=\mathbb H_k$ and $\Lambda=\mathbb H_{k+1}$,
\begin{multline}\label{eq: lower bound typical full space perco}
\mathbb P[0\connect{}x]=\mathbb P[0\connect{\mathbb H\:}x]+\sum_{k\geq 0}\Big(\mathbb P[0\connect{\mathbb H_{k+1}\:}x]-\mathbb P[0\connect{\mathbb H_k\:}x]\Big)
\\\ge \sum_{k\ge 0}\sum_{\substack{y\in \mathbb H_k,\\ z\in \mathbb H_{k+1}\setminus\mathbb H_k}}\mathbb P[0\connect{\mathbb H_k\:}y]p_{yz} \mathbb P[z\connect{\mathbb H_{k+1}\:}x]-\sum_{k\geq 0}E(\mathbb H_k,\mathbb H_{k+1},0,x).
\end{multline}
Looking at the first sum, we see that
\begin{align}
 	\sum_{k\ge 0}\sum_{\substack{y\in \mathbb H_k,\\ z\in \mathbb H_{k+1}\setminus\mathbb H_k}}\mathbb P&[0\connect{\mathbb H_k\:}y]p_{yz} \mathbb P[z\connect{\mathbb H_{k+1}\:}x]\notag
 	\\&\stackrel{\phantom{\eqref{eq:opperco}}}\geq \sum_{k\ge 0}\sum_{\substack{y\in \mathbb H_k,\\ z\in \mathbb H_{k+1}\setminus\mathbb H_k}}\mathbb P[0\connect{\mathbb H_k\:}y]p_{yz}\min\{\mathbb P[z\connect{\mathbb H_{k+1}\:}x]:-z_1=|z|=k+1\}\notag
 	\\&\stackrel{\eqref{eq:opperco}}\geq \sum_{k=0}^{n\wedge (L_\beta(\varepsilon)-1)}\sum_{\substack{y\in \mathbb H_k,\\ z\in \mathbb H_{k+1}\setminus\mathbb H_k}}\mathbb P[0\connect{\mathbb H_k\:}y]p_{yz}\frac{c}{L}\frac{1}{(n+k+1)^{d-1}}\notag
 	\\&\stackrel{\phantom{\eqref{eq:opperco}}}\geq \sum_{k=0}^{n\wedge (L_\beta(\varepsilon)-1)}\frac1{2dL}\varphi(\Lambda_k)\frac{c}{L(n+k+1)^{d-1}}\notag
 	\\&\stackrel{\phantom{\eqref{eq:opperco}}}\geq \frac{c_1}{L^2n^{d-2}},\label{eq: lower bound full space good perco}
\end{align}
where on the third inequality, we used that by symmetry,\begin{align}\notag
	\varphi_\beta(\Lambda_k)\leq 2d\sum_{\substack{y\in \Lambda_k\\z\notin \mathbb H_k}}\mathbb P[0\connect{\Lambda_k\:}y]p_{yz}&\leq 2dL p|[-L,L]^{d-1}|\sum_{\substack{y\in \Lambda_k\\\mathrm{d}(y,\mathbb H_k^c)\leq L}}\mathbb P[0\connect{\Lambda_k\:}y]\\&=2dL \sum_{\substack{y \in \mathbb H_k\\z\in \mathbb H_{k+1}\setminus \mathbb H_k}}\mathbb P[0\connect{\mathbb H_k\:}y]p_{yz},
\end{align}
and where $c_1=c_1(d)>0$ and where we used the bound $\varphi(\Lambda_k)\geq 1-\varepsilon$ provided by the condition $k\leq n\wedge(L_\beta(\varepsilon)-1)$.
It remains to control the ``error'' term in \eqref{eq: lower bound typical full space perco}. We proceed as in \cite[Lemma~3.6]{DumPan24WSAW}. Notice that
\begin{equation}
	\sum_{k\geq 0}E(\mathbb H_k,\mathbb H_{k+1},0,x)=(I)+(II),
\end{equation}
where 
\begin{equation}
	(I):= \sum_{k\geq 0}\sum_{u,v\in \mathbb H_k}\sum_{\substack{y\in \mathbb H_k\\z\in \mathbb H_{k+1}\setminus \mathbb H_k}}\mathbb P[0\connect{\mathbb H_k\:}u]\mathbb P[u\connect{\mathbb H_k\:}y]p_{yz}\mathbb P[z\connect{\mathbb H_{k+1}\:}v]\mathbb P[u\connect{\mathbb H_{k+1}\:}v]\mathbb P[v\connect{\mathbb H_{k+1}\:}x],
\end{equation}
\begin{multline}
	(II):=\sum_{k\geq 0}\sum_{\substack{u\in \mathbb H_k\\v\in \mathbb H_{k+1}}}\sum_{\substack{y,s\in \mathbb H_k\\z,t\in \mathbb H_{k+1}\setminus \mathbb H_k}}\mathbb P[0\connect{\mathbb H_k\:}u]\mathbb P[u\connect{\mathbb H_k\:}y]\mathbb P[u\connect{\mathbb H_k\:}s]p_{yz}p_{st}\\\times\mathbb P[z\connect{\mathbb H_{k+1}\:}v]\mathbb P[t\connect{\mathbb H_{k+1}\:}v]\mathbb P[v\connect{\mathbb H_{k+1}\:}x].
\end{multline}
See Figure \ref{fig:percolbfullspace} for a diagrammatic representation of $(I)$ and $(II)$.
\begin{figure}
	\begin{center}
		\includegraphics{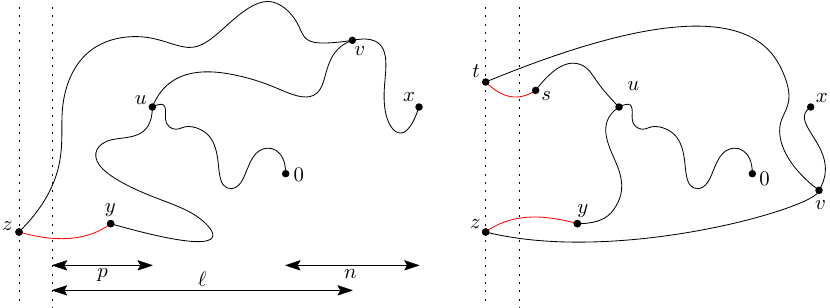}
				\put(-145,10){$\partial\mathbb H_k$}
		\caption{On the left, a diagrammatic representation of $(I)$. On the right, a diagrammatic representation of $(II)$. The red lines corresponds to long edges jumping out of $\mathbb H_k$.}
		\label{fig:percolbfullspace}
	\end{center}
\end{figure}
\paragraph{Bound on $(I)$.} We begin by controlling the contribution in $(I)$ coming from $v\in \partial\mathbb H_{k-\ell}$ with $\ell\geq (n+k)/2$. We write
\begin{align}
	(I)_1&:=\sum_{\substack{k\ge0\\ \ell\geq (n+k)/2}}\sum_{\substack{u\in \mathbb H_{k}\\v\in \partial\mathbb H_{k-\ell}\\y\in \mathbb H_k\\z\in \partial\mathbb H_{k+1}}}\mathbb P[0\connect{\mathbb H_k\:}u]\mathbb P[u\connect{\mathbb H_k\:}y]p_{yz}\mathbb P[z\connect{\mathbb H_{k+1}\:}v]\mathbb P[u\connect{\mathbb H_{k+1}\:}v]\mathbb P[v\connect{\mathbb H_{k+1}\:}x]\notag
	\\&\leq \frac{C_1}{L}\sum_{\substack{k\ge0\\ \ell\geq (n+k)/2}}\frac{C}{L^d}\Big(\frac{L}{L\vee (\ell+1)}\Big)^{d-1}\sum_{\substack{u\in \mathbb H_k\\v\in \partial\mathbb H_{k-\ell}}}\mathbb P[0\connect{\mathbb H_k\:}u]\mathbb P[u\connect{\mathbb H_{k+1}\:}v]\mathbb P[v\connect{\mathbb H_{k+1}\:}x],\label{eq:bound error lower bound full plane perco 1}
\end{align}
where we used Proposition \ref{prop: final prop proof mainperco thm} to bound $\mathbb P[z\connect{\mathbb H_{k+1}\:}v]$, together with \eqref{eq: bound volume times p} and \eqref{eq: sum over half plane at distance k from half plane perco} to argue that
\begin{equation}\label{eq:bound error lower bound full plane perco 1.5}
	\sum_{\substack{y\in \mathbb H_k\\z\in \partial\mathbb H_{k+1}}}\mathbb P [u\connect{\mathbb H_k\:}y]p_{yz}\leq \frac{|\Lambda_L|p_\beta}{2L+1}\sum_{m=0}^L\sum_{y\in \partial\mathbb H_{k-m}}\mathbb P[u\connect{\mathbb H_k\:}y]\leq \frac{C_1}{L}.
\end{equation}
Now, observe that
\begin{align}\notag
	\sum_{\ell\geq (n+k)/2}\frac{C}{L^d}\Big(\frac{L}{L\vee (\ell+1)}\Big)^{d-1}&\sum_{\substack{u\in \mathbb H_k\\v\in \partial\mathbb H_{k-\ell}}}\mathbb P[0\connect{\mathbb H_k\:}u]\mathbb P[u\connect{\mathbb H_{k+1}\:}v]\mathbb P[v\connect{\mathbb H_{k+1}\:}x]
	\\&\leq \frac{C_2}{L(n+k)^{d-1}}\sum_{u,v\in \mathbb Z^d}\mathbb P[0\connect{}u]\mathbb P[u\connect{}v]\mathbb P[v\connect{}x]\notag
	\\&\leq \frac{C_3}{L(n+k)^{d-1}}\frac{1}{L^d},\label{eq:bound error lower bound full plane perco 2}
\end{align}
where we used Proposition \ref{prop: final prop proof mainperco thm} to argue that, since $x\neq 0$, one has
\begin{equation}
	C_2\sum_{u,v\in \mathbb Z^d}\mathbb P[0\connect{}u]\mathbb P[u\connect{}v]\mathbb P[v\connect{}x]\leq  C_3 L^{-d}.
\end{equation}
Combining \eqref{eq:bound error lower bound full plane perco 1} and \eqref{eq:bound error lower bound full plane perco 2}, we obtain that 
\begin{equation}
	(I)_1\leq \frac{C_1C_3}{L^2}\frac{1}{L^d}\sum_{k\geq 0} \frac{1}{(n+k)^{d-1}}\leq \frac{C_4}{L^d}\frac{1}{L^2n^{d-2}}.\label{eq: I 1}
\end{equation}
We now bound
\begin{equation}
	(I)_2:=\sum_{\substack{k\ge0\\ \ell<(n+k)/2}}\sum_{\substack{u\in \partial\mathbb H_{k}\\v\in \partial \mathbb H_{k-\ell}\\y\in \mathbb H_k\\z\in \partial\mathbb H_{k+1}}}\mathbb P[0\connect{\mathbb H_k\:}u]\mathbb P[u\connect{\mathbb H_k\:}y]p_{yz}\mathbb P[z\connect{\mathbb H_{k+1}\:}v]\mathbb P[u\connect{\mathbb H_{k+1}\:}v]\mathbb P[v\connect{\mathbb H_{k+1}\:}x].
\end{equation}
For this, we also decompose the sum according to the position of $u$ and introduce $p\geq 0$ such that $u\in \partial \mathbb H_{k-p}$.
If $\ell<(n+k)/2$, Lemma \ref{lem: distance k from half-space} implies that
\begin{equation}\label{eq:bound error lower bound full plane perco 3}
	\max_{v\in \partial \mathbb H_{k-\ell}}\mathbb P[v\connect{\mathbb H_{k+1}\:}x]\leq \frac{C_5(\ell+L)}{L^2(n+k)^{d-1}}.
\end{equation}
 Hence,
 \begin{multline}\label{eq:bound error lower bound full plane perco 4}
 	(I)_2\leq \frac{C_5}{L^2}\sum_{k\geq 0}\frac{1}{(n+k)^{d-1}}\sum_{p \geq 0}\sum_{u\in \partial \mathbb H_{k-p}}\mathbb P[0\connect{\mathbb H_k\:}u]\sum_{\substack{y\in \mathbb H_k\\z \in \partial \mathbb H_{k+1}}}\mathbb P[u\connect{\mathbb H_k\:}y]p_{yz}\\\times\sum_{\ell\geq 0}(\ell+L)\sum_{v\in \partial \mathbb H_{k-\ell}}\mathbb P[z\connect{\mathbb H_{k+1}\:}v]\mathbb P[u\connect{\mathbb H_{k+1}\:}v].
 \end{multline}
 We now bound the sum over $\ell$ in this right-hand side of \eqref{eq:bound error lower bound full plane perco 4}. We split the sum into two contributions: either $\ell \geq p/2$ or $\ell <p/2$. 
 
On the one hand, if $\ell \geq p/2$, we use again Proposition \ref{prop: final prop proof mainperco thm} to bound $\mathbb P[z\connect{\mathbb H_{k+1}\:}v]$, and obtain
  \begin{equation}
 	\sum_{v\in \partial\mathbb H_{k-\ell}}\mathbb P[z\connect{\mathbb H_{k+1}\:}v]\mathbb P[u\connect{\mathbb H_{k+1}\:}v]\leq \frac{C}{L^d}\left(\frac{L}{L\vee (\ell+1)}\right)^{d-1}\Big(\mathds{1}_{p=\ell}+\frac{2dC^2}{L^2}(p+1+L)\Big),
 \end{equation}
 where we additionally used \eqref{eq: sum over half plane at distance k from half plane perco} and translation invariance to get
 \begin{align}\notag
 	\sum_{v\in \partial\mathbb H_{k-\ell}}\mathbb P[u\connect{\mathbb H_{k+1}\:}v]&=\sum_{u'\in \partial\mathbb H_{k-p}}\mathbb P[u'\connect{\mathbb H_{k+1}\:}(-k+\ell)\mathbf{e}_1]\\&=\sum_{u''\in \partial \mathbb H_{\ell+1-(p+1)}}\mathbb P[u''\connect{\mathbb H_{\ell+1}\:}0]\notag
 	\\&\leq \mathds{1}_{p=\ell}+\frac{2dC^2}{L^2}(p+1+L).
 \end{align}
 Then, we write
 \begin{equation}
 	\sum_{\ell\geq p/2}(\ell +L)\frac{C}{L^d}\left(\frac{L}{L\vee (\ell+1)}\right)^{d-1}\left(\mathds{1}_{p=\ell}+\frac{2dC^2}{L^2}(p+1+L)\right)\leq \frac{C_7}{L^3}\left(\frac{1}{L\vee (p+1)}\right)^{d-4},\label{eq:bound error lower bound full plane perco 5}
 \end{equation}
 where we used that for all $m\geq 0$, one has $(m+L)\leq 2(L\vee (m+1))$, to get that the left-hand side in the above display is bounded by
 \begin{equation}
 	\frac{2C}{L}\left(\frac{1}{L\vee (p+1)}\right)^{d-2}+\frac{8dC^3}{L^3}(L\vee (p+1))\sum_{\ell \geq p/2}\left(\frac{1}{L\vee (\ell+1)}\right)^{d-2}.
 \end{equation}
On the other hand, using Proposition \ref{prop: final prop proof mainperco thm} twice, one time to get $\sum_{v\in \partial \mathbb H_{k-\ell}}\mathbb P[z\connect{\mathbb H_k\:}v]\leq \frac{C}{L}$, and a second time to get $\mathbb P[u\connect{\mathbb H_{k+1}\:}v]\leq \mathbb P[u\connect{}v]\leq \frac{C}{L^d}\left(\frac{L}{L\vee (p/2)}\right)^{d-2}$, we obtain
\begin{align}
\sum_{0\leq \ell<p/2}(\ell+L)\sum_{v\in \partial \mathbb H_{k-\ell}}\mathbb P[z\connect{\mathbb H_{k+1}\:}v]\mathbb P[u\connect{\mathbb H_{k+1}\:}v]&\leq \frac{C}{L}(p/2)(p/2+L)\frac{C}{L^d}\left(\frac{L}{L\vee (p/2)}\right)^{d-2}\notag
\\&\leq \frac{C_8}{L^3}\left(\frac{1}{L\vee (p+1)}\right)^{d-4}.\label{eq:bound error lower bound full plane perco 6}
\end{align}
Combining \eqref{eq:bound error lower bound full plane perco 5} and \eqref{eq:bound error lower bound full plane perco 6}, we get
\begin{equation}
	\sum_{\ell\geq 0}(\ell+L)\sum_{v\in \partial \mathbb H_{k-\ell}}\mathbb P[z\connect{\mathbb H_{k+1}\:}v]\mathbb P[u\connect{\mathbb H_{k+1}\:}v]\leq \frac{C_9}{L^3}\left(\frac{1}{L\vee (p+1)}\right)^{d-4}.
\end{equation}
Plugging the above displayed equation in \eqref{eq:bound error lower bound full plane perco 4}, and using \eqref{eq:bound error lower bound full plane perco 1.5} yields
\begin{align}
	(I)_2&\leq \frac{C_{5}}{L^2}\frac{C_9}{L^3}\frac{C_1}{L}\sum_{k\geq 0}\frac{1}{(n+k)^{d-1}}\sum_{p\geq 0}\left(\frac{1}{L\vee (p+1)}\right)^{d-4}\left(\mathds{1}_{p=k}+\frac{2dC^2}{L^2}(p+L)\right)\notag
	\\&\leq \frac{C_{10}}{L^d}\frac{1}{L^2n^{d-2}},\label{eq: I 2}
\end{align}
where we used that $d>6$.
\paragraph{Bound on $(II)$.} Again, we decompose on the values of $\ell,p$ such that $v\in \mathbb H_{k+1-\ell}$ and $u\in \mathbb H_{k-p}$. 

As above, we begin by considering the contribution $(II)_1$ coming from $\ell \geq (n+k)/2$ and $p\leq \ell$. One has,
\begin{multline}\label{eq:bound error lower bound full plane perco 6}
	(II)_1\leq \frac{C_1^2}{L^2}\sum_{k\geq 0}\frac{C2^{d-1}}{L(n+k)^{d-1}}\sum_{\ell\geq (n+k)/2}\frac{C}{L(\ell+1)^{d-1}}\\\times\sum_{p\leq \ell}\sum_{u\in \partial \mathbb H_{k-p}}\mathbb P[0\connect{\mathbb H_k\:}u]\sum_{v\in \partial\mathbb H_{k+1-\ell}}\mathbb P[v\connect{\mathbb H_{k+1}\:}x],
\end{multline}
where we used Proposition \ref{prop: final prop proof mainperco thm} to argue that (recall that $n\geq 4L$)
\begin{align}
	\mathbb P[z\connect{\mathbb H_{k+1}\:}v]&\leq \frac{C2^{d-1}}{L(n+k)^{d-1}}, \\
	\mathbb P[t\connect{\mathbb H_{k+1}\:}v]&\leq \frac{C}{L^d}\left(\frac{L}{L\vee (\ell +1)}\right)^{d-1}=\frac{C}{L(\ell+1)^{d-1}},
\end{align}
 together with \eqref{eq:bound error lower bound full plane perco 1.5} twice to argue that 
 \begin{equation}\label{eq:bound error lower bound full bonus 1}
 	\sum_{\substack{y\in \mathbb H_k\\z\in \partial\mathbb H_{k+1}}}\mathbb P[u\connect{\mathbb H_k\:}y]p_{yz}\leq \frac{C_1}{L}, \quad \sum_{\substack{s\in \mathbb H_k\\t\in \partial\mathbb H_{k+1}}}\mathbb P[u\connect{\mathbb H_k\:}s]p_{st}\leq \frac{C_1}{L}.
 \end{equation}
Now, observe that by \eqref{eq: sum over half plane at distance k from half plane perco}
\begin{equation}\label{eq:bound error lower bound full plane perco 7}
	\sum_{v\in \partial \mathbb H_{k+1-\ell}}\mathbb P[v\connect{\mathbb H_{k+1}\:}x]\leq \mathds{1}_{\ell=n+k+1}+\frac{2dC^2}{L^2}(\ell+L),
\end{equation}
and 
\begin{equation}\label{eq:bound error lower bound full plane perco 8}
	\sum_{u\in \partial \mathbb H_{k-p}}\mathbb P[0\connect{\mathbb H_k\:}u]\leq \mathds{1}_{p=k}+\frac{2dC^2}{L^2}(p+L).
\end{equation}
Then, observe that $\sum_{p\leq \ell}\left(\mathds{1}_{p=k}+\frac{2dC^2}{L^2}(p+L)\right)\leq 1+\frac{4dC^2}{L^2}(\ell+1)^2$ (where we used that $\ell\geq (n+k)/2\geq L$). Moreover,
\begin{equation}
	\sum_{\ell\geq (n+k)/2}\frac{C}{L(\ell+1)^{d-1}}\left(\mathds{1}_{\ell=n+k}+\frac{2dC^2}{L^2}(\ell+L)\right)\left(1+\frac{4dC^2}{L^2}(\ell+1)^2\right)\leq \frac{C_{11}}{L^d}.
\end{equation}
Plugging the previously displayed equation in \eqref{eq:bound error lower bound full plane perco 6} gives
\begin{equation}\label{eq: II 1}
	(II)_1\leq \frac{C_{12}}{L^d}\frac{1}{L^2n^{d-2}}.
\end{equation}
We now bound the contribution $(II)_2$ coming from $\ell\geq (n+k)/2$ and $p>\ell$. This time, observe that
\begin{multline}
	(II)_2\leq 4\frac{C_1}{L}\frac{C}{L}\sum_{k\geq 0} \frac{C2^{d-1}}{L(n+k)^{d-1}}\\\times\sum_{\ell \geq (n+k)/2}\left(\mathds{1}_{\ell=n+k}+\frac{2dC^2}{L^2}(\ell+L)\right)\sum_{p>\ell}\frac{C2^{d-1}}{L p^{d-1}}\left(\mathds{1}_{p=k}+\frac{2dC^2}{L^2}(p+L)\right),
\end{multline}
where this time we used \eqref{eq: half plane at distance $k$ from half space perco} to get that (since $p>\ell\geq 2L$)
\begin{equation}
	\mathbb P[u\connect{\mathbb H_k\:}y]\leq\frac{C^2}{L(p-L)^{d-1}}\leq \frac{C^22^{d-1}}{Lp^{d-1}},
\end{equation}
 then we used \eqref{eq: bound volume times p} to get $\sum_{y\in \mathbb Z^d}p_{yz}\leq 4$, \eqref{eq:bound error lower bound full plane perco 1.5} again to get $\sum_{\substack{s\in \mathbb H_k\\t\in \partial\mathbb H_{k+1}}}\mathbb P[u\connect{\mathbb H_k\:}s]p_{st}\leq \frac{C_1}{L}$, Proposition \ref{prop: final prop proof mainperco thm} to get 
 \begin{equation}
 	\sum_{z\in \partial\mathbb H_{k+1}}\mathbb P[z\connect{\mathbb H_{k+1}\:}v]\leq \frac{C}{L},\quad \mathbb P[t\connect{\mathbb H_{k+1}\:}v]\leq \frac{C2^{d-1}}{L(n+k)^{d-1}},
 \end{equation}
and we where we used \eqref{eq:bound error lower bound full plane perco 7} to bound $\sum_{v\in \partial \mathbb H_{k+1-\ell}}\mathbb P[v\connect{\mathbb H_{k+1}\:}x]$ and \eqref{eq:bound error lower bound full plane perco 8} to bound $\sum_{u\in \partial\mathbb H_{k-p}}\mathbb P[0\connect{\mathbb H_k\:}u]$.
 Again, one may check that
 \begin{equation}\label{eq: II 2}
 	(II)_2\leq \frac{C_{13}}{L^d}\frac{1}{L^2n^{d-2}}.
 \end{equation}

We now turn to the contribution $(II)_3$ in $(II)$ coming from $\ell <(n+k)/2$ and $p\leq \ell$. Observe that,
\begin{multline}\label{eq:bound error lower bound full plane perco 9}
	(II)_3\leq \frac{C_1^2}{L^2}\left(1+\frac{C}{L}\right)\sum_{k\geq 0}\frac{C^2 2^{d-1}}{L^2(n+k)^{d-1}}\\\times\sum_{\ell <(n+k)/2}(\ell+L)\frac{C}{L^d}\left(\frac{L}{L\vee (\ell+1)}\right)^{d-1}\left(1+\frac{4dC^2}{L^2}(\ell+1)^2\right),
\end{multline}
where we used \eqref{eq: half plane at distance $k$ from half space perco} to argue that
\begin{equation}\label{eq:bound error lower bound full bonus 2}
	\mathbb P[v\connect{\mathbb H_{k+1}\:}x]\leq \frac{C^2(\ell+L)}{L^{d+1}}\left(\frac{L}{L\vee (n+k-\ell)}\right)^{d-1}\leq \frac{C^2(\ell+L)2^{d-1}}{L^2(n+k)^{d-1}},
\end{equation}
then we used Proposition \ref{prop: final prop proof mainperco thm} to get that 
\begin{equation}
	\sum_{v\in \partial \mathbb H_{k+1-\ell}}\mathbb P[z\connect{\mathbb H_{k+1}\:}v]\leq \mathds{1}_{\ell=0}+\frac{C}{L}, \quad \mathbb P[t\connect{\mathbb H_{k+1}\:}v]\leq \frac{C}{L^d}\left(\frac{L}{L\vee (\ell+1)}\right)^{d-1},
\end{equation}
then \eqref{eq:bound error lower bound full bonus 1}, and the inequality below \eqref{eq:bound error lower bound full plane perco 8} to get $\sum_{p\leq \ell}\sum_{u\in \partial\mathbb H_{k-p}}\mathbb P[0\connect{\mathbb H_k\:}u]\leq 1+\frac{4dC^2}{L^2}(\ell+1)^2$. From \eqref{eq:bound error lower bound full plane perco 9} one finds that
\begin{equation}\label{eq: II 3}
	(II)_3\leq \frac{C_{14}}{L^d}\frac{1}{L^2n^{d-2}}.
\end{equation}

We are almost done. It remains to control the contribution $(II)_4$ coming from $\ell <(n+k)/2$ and $p>\ell$. We write,
\begin{multline}\label{eq:bound error lower bound full plane perco 10}
	(II)_4\leq 4\frac{C_1}{L}\sum_{k\geq 0}\frac{C^22^{d-1}}{L^2(n+k)^{d-1}}\sum_{\ell <(n+k)/2}(\ell+L)\frac{C}{L^d}\left(\frac{L}{L\vee (\ell+1)}\right)^{d-1}\left(\mathds{1}_{\ell=0}+\frac{C}{L}\right)\\\times\sum_{p>\ell}\frac{C^2}{L^d}\left(\frac{L}{L\vee (p-L)}\right)^{d-1}\left(\mathds{1}_{k=p}+\frac{2dC^2}{L^2}(p+L)\right),
\end{multline}
where we used \eqref{eq:bound error lower bound full bonus 2}, \eqref{eq: half plane at distance $k$ from half space perco} to argue that $\mathbb P[u\connect{\mathbb H_k\:}y]\leq\frac{C^2}{L^d}\left(\frac{L}{L\vee (p-L)}\right)^{d-1}$, \eqref{eq: bound volume times p} to get $\sum_{y}p_{yz}\leq 4$, Proposition \ref{prop: final prop proof mainperco thm} to get 
\begin{equation}
	\sum_{z\in \partial\mathbb H_{k+1}}\mathbb P[z\connect{\mathbb H_{k+1}\:}v]\leq \mathds{1}_{\ell=0}+\frac{C}{L},\quad \mathbb P[t\connect{\mathbb H_{k+1}\:}v]\leq \frac{C}{L^d}\left(\frac{L}{L\vee (\ell+1)}\right)^{d-1},
\end{equation}
then \eqref{eq:bound error lower bound full plane perco 1.5} to get $\sum_{\substack{s\in \mathbb H_k\\t\in \partial\mathbb H_{k+1}}}\mathbb P[u\connect{\mathbb H_k\:}s]p_{st}\leq \frac{C_1}{L}$, and finally \eqref{eq:bound error lower bound full plane perco 8} to get
\begin{equation}
\sum_{u\in \partial\mathbb H_{k-p}}\mathbb P[0\connect{\mathbb H_k\:}u]\leq \mathds{1}_{k=p}+\frac{2dC^2}{L^2}(p+L).\end{equation} Using \eqref{eq:bound error lower bound full plane perco 10} we may show that
\begin{equation}\label{eq: II 4}
 (II)_4\leq \frac{C_{15}}{L^d}\frac{1}{L^2n^{d-2}}.
\end{equation}
Plugging \eqref{eq: lower bound full space good perco}, \eqref{eq: I 1}, \eqref{eq: I 2}, \eqref{eq: II 1}, \eqref{eq: II 2}, \eqref{eq: II 3}, and \eqref{eq: II 4} in \eqref{eq: lower bound typical full space perco} gives
\begin{equation}
	\mathbb P[0\connect{}x]\geq \frac{c_1}{L^2n^{d-2}}-\frac{C_4+C_{10}+C_{12}+C_{13}+C_{14}+C_{15}}{L^d}\frac{1}{L^2n^{d-2}}.
\end{equation}
The proof for $n\geq 4L$ follows from choosing $L_9$ large enough so that 
\begin{equation}
\frac{C_4+C_{10}+C_{12}+C_{13}+C_{14}+C_{15}}{L_9^d}\leq \frac{c_1}{2},
\end{equation}
and setting $c=c_1/2$.

If $n<4L$, we observe that $\mathbb P[0\connect{}x]\geq \mathbb P[0\connect{\mathbb H\:}x]$ and use Lemma \ref{lem:lower below half-space perco}. This concludes the proof (to the cost of potentially decreasing $c$)
\end{proof}

The two preceding lemmata are enough to conclude the proof of Theorem \ref{thm:main2perco}.

\begin{proof}[Proof of Theorem~\textup{\ref{thm:main2perco}}] Theorem \ref{thm:main2perco} immediately follows from the two preceding lemmata and \eqref{eq: comparison diff l beta}.	
\end{proof}

\appendix
\section{Appendix: random walk estimates}\label{appendix:rw}
\subsection{Estimates on $\mathbb P^{\rm RW}$}
Let $u \in \mathbb Z^d$. Recall from \eqref{eq:def random walk J} the definition of $\mathbb P^{\rm RW}_u$, and let $\tau^u:=\inf \{\ell\geq 1: X_\ell^u\notin \mathbb H\}$.
\begin{Prop}\label{prop: rw estimates} Let $d>2$ and $L\geq 1$. There exists $C_{\rm RW}=C_{\rm RW}(d)>0$ such that, for every $x\in \mathbb H^*$,
\begin{equation}\label{eq:first rw estimates}
	\mathbb E_0^{\rm RW}\Big[\sum_{\ell<\tau^0}\mathds{1}_{X_\ell^0=x}\Big]\leq \frac{C_{\rm RW}}{L^d}\left(\frac{L}{L\vee |x_1|}\right)^{d-1}
\end{equation}
\end{Prop}
\begin{proof} Using \cite{UchiyamaGreenFunctionEstimates1998}, we obtain the existence of $C_1=C_1(d)>0$ such that for every $L\geq 1$ and every $x\in (\mathbb Z^d)^*$,
\begin{equation}\label{eq:interm proof halfspace green function 1}
	\sum_{\ell\geq 0}\mathbb P_0^{\rm RW}[X_\ell^0=x]\leq \frac{C_1}{L^d}\left(\frac{L}{L\vee |x|}\right)^{d-2}.
\end{equation}
As a consequence, if $|x_1|\leq 10L$,
\begin{equation}\label{eq:interm proof halfspace green function 1.5}
	\sum_{\ell\geq 0}\mathbb P_0^{\rm RW}[X_\ell^0=x]\leq \frac{C_1}{L^d}\leq \frac{C_1 10^{d-1}}{L^d}\left(\frac{L}{L\vee |x_1|}\right)^{d-1}.
\end{equation}
This gives \eqref{eq:first rw estimates} when $x\in \mathbb H^*$ satisfies $|x_1|\leq 10L$.

We now assume that $|x_1|>10L$. A classical Gambler's ruin estimate (see for instance\cite[Chapter~5.1.1]{LawlerLimicRandomWalks2010}) yields the existence of $C_2=C_2(d)>0$ such that for every $L\geq 1$ and every $k>0$,
\begin{equation}\label{eq:eq:interm proof halfspace green function 2}
	\mathbb P_0^{\rm RW}[\tilde{\tau}_k^0<\tau^0]\leq \frac{C_2 L}{k},
\end{equation}
where $\tilde{\tau}_k^0:=\inf\{\ell\geq 1: (X_\ell^0)_1\geq k\}$. 
Finally, we observe that for $k=\lfloor |x_1|/2\rfloor$,
\begin{align}\notag
	\mathbb E_0^{\rm RW}\Big[\sum_{\ell<\tau}&\mathds{1}_{X_\ell^0=x}\Big]\\&\leq\sum_{\ell \geq 0}\sum_{m=0}^\ell\sum_{y\in \mathbb H}\mathbb P_0^{\rm RW}[X_\ell^0=x, \: X_m^0=y, \: \tilde{\tau}^0_k=m, \: \tau^0>m]\notag
	\\&=\sum_{m\geq 0}\sum_{y\in \mathbb H: \:k\leq |y_1|<k+L}\mathbb P_0^{\rm RW}[X_m^0=y, \: \tilde{\tau}_k^0=m, \: \tau^0>m]\sum_{\ell\geq m}\mathbb P_0^{\rm RW}[X_{\ell-m}^0=x-y]\notag
	\\&\leq \frac{C_1}{L^d}\left(\frac{L}{L\vee (|x_1|/2-L)}\right)^{d-2}\mathbb P_0^{\rm RW}[\tilde{\tau}_k^0<\tau]\notag
	\\&\leq \frac{C_1}{L^2}\left(\frac{4}{|x_1|}\right)^{d-2}\frac{C_2L}{|x_1|/3}\leq \frac{C_4}{L}\left(\frac{1}{|x_1|}\right)^{d-1},
\end{align}
where on the second line we used Markov's property, on the third line we used \eqref{eq:interm proof halfspace green function 1}, in the last line we used \eqref{eq:eq:interm proof halfspace green function 2} together with the fact that $|x_1|-2-L\geq |x_1|/4$ and $\lfloor |x_1|/2\rfloor\geq |x_1|/3$, and where $C_4=C_4(d)>0$. This concludes the proof.
\end{proof}
\begin{Prop}\label{prop: rw estimates 2} Let $d>2$, $L\geq 1$, and $C>0$. There exists $C_{\rm RW}'=C_{\rm RW}'(C,d)>0$ such that for every $u\in \mathbb H$ with $|u_1|\leq CL$ the following properties hold. For every $x\in \mathbb H\setminus \{u\}$ with $|x_1|\leq CL$, 
\begin{equation}\label{eq:rw estimate uniform in u 1}
	\mathbb E_u^{\rm RW}\Big[\sum_{\ell<\tau^u}\mathds{1}_{X_\ell^u=x}\Big]\leq \frac{C_{\rm RW}'}{L^d}\left(\frac{L}{L\vee |x_1|}\right)^{d-1},
\end{equation}
and for every $T\geq 1$,
\begin{equation}\label{eq:rw estimate uniform in u 2}
	\mathbb P_u^{\rm RW}[\tau^u>T]\leq \frac{C_{\rm RW}'}{\sqrt{T}}.
\end{equation}
\end{Prop}
\begin{proof} Fix $u\in \mathbb H$ with $|u_1|\leq CL$. We begin by showing \eqref{eq:rw estimate uniform in u 1}. Fix $x\in \mathbb H\setminus \{u\}$ with $|x_1|\leq CL$. Using \eqref{eq:interm proof halfspace green function 1}, we obtain
\begin{equation}
	\mathbb E^{\rm RW}_u\Big[\sum_{\ell <\tau^u}\mathds{1}_{X_\ell^u=x}\Big]\leq \frac{C_1}{L^d}\leq \frac{C_1C^{d-1}}{L^d}\left(\frac{L}{L\vee |x_1|}\right)^{d-1},
\end{equation}
from which \eqref{eq:rw estimate uniform in u 1} follows.

We now turn to \eqref{eq:rw estimate uniform in u 2}. Let $t>0$ to be chosen small enough. Observe that the process $Y$ of the first coordinate of $X^u/(tL)$ is a one-dimensional random walk started at $u_1/(tL)$. We denote its law by $\mathbb P_{u_1}$ and let $\tau':=\inf\{k\geq 1: Y_k<0\}$. Observe that
\begin{equation}
	\mathbb P^{\rm RW}_u[\tau^u>T]\leq \mathbb P_{u_1}[\tau'>T].
\end{equation}
Now, for $t$ sufficiently small (independently of $L$), the process $Y$ satisfies the assumptions of \cite[Theorem~5.1.7]{LawlerLimicRandomWalks2010}. In particular, there exists $C_2=C_2(d,t)>0$ such that,
\begin{equation}
	\mathbb P_{u_1}[\tau'>T]\leq \frac{C_2 \cdot(u_1/tL)}{\sqrt{T}}\leq \frac{C_2 C t^{-1}}{\sqrt{T}}.
\end{equation}
The proof follows from setting $C_{\rm RW}':=C_1C^{d-1}\vee C_2 C t^{-1}$.
\end{proof}

\begin{Prop}[A coupling estimate]\label{prop:coupling appendix} Let $d\geq 1$ and $\varepsilon>0$. There exists $T=T(\varepsilon,d)>0$ such that the following holds. For every $L\geq 1$ and every $u,v\in \Lambda_{2L}$, there exists a pair of random variables $(Y^u,Y^v)$ of law $\mathbf P$ such that $Y^u$ (resp. $Y^v$) has the same law as $X^u_T$ (resp. $X^v_T$) and
\begin{equation}
	\mathbf P[Y^u\neq Y^v]\leq \varepsilon.
\end{equation}
\end{Prop}

\begin{proof} Let $\kappa>0$ to be chosen small enough in terms of $\varepsilon$. 

We first treat the case $|u-v|\leq 1\vee \kappa L$. Recall that $J_{st}=c_L\mathds{1}_{1\leq |s-t|\leq L}$ with $c_L=|\Lambda_L^*|^{-1}$. The total variation distance $\mathrm{d}_{\rm TV}$ between $X_1^u$ and $X_1^v$ satisfies
\begin{equation}
	\mathrm{d}_{\rm TV}(X_1^u,X_1^v)=\frac{1}{2}\sum_{y\in \mathbb Z^d}|J_{uy}-J_{vy}|\leq \frac{c_L}{2}|\Lambda_L(u)\Delta\Lambda_L(v)|\leq C_1 \kappa^d,
\end{equation} 
where $\Lambda_L(u)\Delta\Lambda_L(v)$ denotes the symmetric difference of sets, and where $C_1=C_1(d)>0$. As a consequence (see for instance \cite[Appendix A.4.2]{LawlerLimicRandomWalks2010}), there exists a coupling $(\xi^u,\xi^v)\sim\mathbf P_{u,v}^1$ of $(X_1^u,X_1^v)$ such that
\begin{equation}\label{eq:coupling1}
	\mathbf P_{u,v}^1[\xi^u=\xi^v]\geq 1-C_1\kappa^d.
\end{equation}
We now fix $\kappa$ small enough so that $C_1\kappa^d\leq \varepsilon/2$. This settles the first case.

We now treat the case $\kappa L \leq |u-v|\leq 2L$. We claim that it is sufficient to show that there exists $T=T(\varepsilon,d)>0$ and $(Z^u,Z^v)\sim\mathbf P^2$ a coupling of $(X^u_{T-1},X_{T-1}^v)$ which satisfies
\begin{equation}\label{eq: intermediate step coupling}
	\mathbf P^2[|Z^u-Z^v|\leq\kappa L]\geq 1-\frac{\varepsilon}{2}.
\end{equation}
Indeed, let $(Z^u,Z^v)\sim \mathbf P^2$. Let also $U\sim X_1^0$ be independent of $(Z^u,Z^v)$. We construct $(Y^u,Y^v)$ as follows: if $\{|Z^u-Z^v|\leq \kappa L\}$ occurs, then sample $(\xi^{Z^u},\xi^{Z^v})\sim \mathbf P_{Z^u,Z^v}^1$ and set $(Y^u,Y^v):=(\xi^{Z^u},\xi^{Z^v})$; otherwise set $(Y^u,Y^v):=(Z^u+U,Z^v+U)$. It is easy to check that $Y^u$ (resp. $Y^v$) has the same law as $X_T^u$ (resp. $X_T^v$). Letting $\mathbf P$ be the law of $(Y^u,Y^v)$, we find that
\begin{equation}
	\mathbf P[Y^u\neq Y^v]\leq \mathbf P^2[|Z^u-Z^v|>\kappa L]+\mathbf E^2[\mathbf P^1_{Z^u,Z^v}[\xi^{Z^u} \neq \xi^{Z^v}]\mathds{1}_{|Z^u-Z^v|\leq \kappa L}]\leq \varepsilon,
\end{equation}
where we used \eqref{eq:coupling1} and \eqref{eq: intermediate step coupling}.
 
It remains to prove \eqref{eq: intermediate step coupling}. Observe that the projection of $X^u$ on its $i$-th coordinate is a one-dimensional random walk $(S_k^{u_i})_{k\geq 0}$ of step distribution 
\begin{equation}\label{eq: law increments one dim walk}
	\mathsf P_{u_i}[S_{1}^{u_i}=x+u_i]=c_L (2L+1)^{d-1}\mathds{1}_{1\leq |x|\leq L}+c_L\Big((2L+1)^{d-1}-1\Big)\mathds{1}_{x=0}, \qquad (x\in \mathbb Z).
\end{equation}
Thus, it suffices to find a coupling $(V^{u_1},V^{v_1})\sim\mathsf{P}$ of $(S^{u_1},S^{v_1})$ for which, for every $T$ large enough $|V^{u}_T-V^{v}_T|\leq \kappa L$ with high probability. The general case then follows by constructing the coupling coordinate by coordinate. Such a coupling can be obtained through the following procedure. 

Below, we assume that $L$ is large enough so that $\kappa L\geq 2$. The case of the small values of $L$ can be treated using \cite[Lemma~2.4.3]{LawlerLimicRandomWalks2010}. Without loss of generality, we also assume that $v_1=0$ and $2L\geq x:=u_1>\kappa L$.

 The idea is to find a coupling $(V^x,V^0)$ of $(S^{x},S^{0})$ such that the difference process $V^{x}-V^{0}$ is a random walk that makes jumps at a distance at most $\kappa L$, and stops evolving as soon as it enters $[-\kappa L, \kappa L]$. This classical step is done using the \emph{Ornstein coupling} (see for instance \cite[(3.3)]{den2012probability}). We describe it fully for sake of completeness. We let $(\omega_i)_{i\geq 0}$ and $(\omega_i')_{i\geq 0}$ be two independent sequences of i.i.d. random variables of law $S_1^0$. We construct the pair $(V^x,V^0)$ inductively as follows:
 \begin{enumerate}
 	\item Set $(V^x_0,V^0_0)=(x,0)$.
 	\item Let $k\geq 0$ and assume that $(V^x_i,V^0_i)_{i\leq k}$ is constructed and that $(V^x_i)_{i\leq k}$ (resp. $(V^0_i)_{i\leq k}$) has the same law as $(S^x_i)_{i\leq k}$ (resp. $(S^0_i)_{i\leq k}$). 
 	\begin{enumerate}
 	\item If $|V^x_k-V^0_k|\leq \kappa L$, set $V^x_{k+1}-V^x_k=V^0_{k+1}-V^0_k=\omega_{k+1}$.
 	\item If $|V^x_k-V^0_k|>\kappa L$, set $V^x_{k+1}-V^x_k=\omega_{k+1}$ and $V^0_{k+1}-V^0_k=\omega''_{k+1}$, where
 	\begin{equation}
 		\omega''_{k+1}= \left\{
    \begin{array}{ll}
        \omega_{k+1}' & \mbox{if } |\omega_{k+1}-\omega_{k+1}'|\leq \kappa L, \\
        \omega_{k+1} & \mbox{otherwise.}
    \end{array}
\right.
 	\end{equation}
 	\end{enumerate}
 \end{enumerate}
  We denote by $\mathsf P$ the law of the pair $(V^x,V^0)$ then obtained. One can check (see \cite{den2012probability}) that $(V^x_i)_{i\geq 0}$ (resp. $(V^0_i)_{i\geq 0}$) has the same law as $(S^x_i)_{i\geq 0}$ (resp. $(S^0_i)_{i\geq 0}$). The difference process $\tilde S:=V^x-V^0$ is a random walk started at $x$ and stopped at $\tilde \tau:=\inf \{k\geq 1: \tilde S_k<\kappa L\}$ with increments $\tilde \xi$ of law $(\omega_1-\omega_1')\mathds{1}_{|\omega-\omega'|\leq \kappa L}$. By definition, one has, for each $T\geq 1$,
  \begin{equation}
  	\mathsf P[|V^x_T-V^0_T|>\kappa L]\leq \mathsf P[\tilde \tau >T].
  \end{equation}
  Using \eqref{eq: law increments one dim walk}, one may check that the renormalised walk $\tilde S/(tL)$ satisfies the assumptions of \cite[Theorem~5.1.7]{LawlerLimicRandomWalks2010} for a well chosen $t>0$. It then follows that there exists $C_2=C_2(d,t)>0$ such that for every $T\geq 1$,
  \begin{equation}
  	\mathsf P[\tilde \tau >T]\leq \frac{C_2\cdot (x/(tL))}{\sqrt{T}}\leq \frac{2C_2t^{-1}}{\sqrt{T}}.
  \end{equation}
The result follows from combining the two previously displayed equations and choosing $T$ large enough.
\end{proof}

\subsection{Uniform estimates on finite-range random walks}\label{appendix:uniform rw estimates}

We now turn to a series of estimates holding uniformly among a class of finite-range random walks on $\mathbb Z^d$. These estimates are meant to be applied to the random walk distribution $\mathbb P^{{\rm RW},m}$ introduced in \eqref{eq: rescaled rw def} but actually hold in a wider generality which we now describe.

Let us introduce some notations. For every $m\geq 1$, let $\mathcal P_m$ denote the set of measures on $\mathbb Z^d$ that satisfy the following assumptions: for every $\mu \in \mathcal P_m$, one has
\begin{enumerate}
	\item[$(i)$] $\mu$ is invariant under reflections and rotations of $\mathbb Z^d$;
	\item[$(ii)$] $\mu$ is supported on $\Lambda_{2m}\setminus \Lambda_{m-1}$.
\end{enumerate}
For $\mu \in \mathcal P_m$ and $u \in \mathbb Z^d$, we let $\mathbb P_{\mu,m}$ be the law of the random walk $(X_k)_{k\geq 0}$ on $\mathbb Z^d$  started at $u$ with step distribution $\mu$. Observe that for $\mu \in \mathcal P_m$,
\begin{equation}\label{eq: estimate sigma}
	m^2\leq \sigma(\mu,m)^2:=\mathbb E_{\mu,0}[|X_1|_2^2]\leq 4dm^2,
\end{equation}
where we recall that $| \cdot |_2$ denotes the standard Euclidean norm on $\mathbb R^d$. When clear from context, we omit $\mu$ from the notation.
\begin{Rem} If $m\geq L$, we observe that the step distribution of \eqref{eq: rescaled rw def} belongs to $\mathcal P_m$. Moreover, for every $m\geq 1$, the step distribution of $\mathbb P'$ considered in the proof of \cite[Proposition~3.2]{DumPan24WSAW} belongs to $\mathcal P_m$.
\end{Rem}
%
Our first estimate concerns the expected exit time of a large box $\Lambda_n$. For $n\geq 1$, set $\tau_n:=\inf\{k\geq 0: X_k\notin \Lambda_n\}$.

\begin{Prop}\label{prop: exponential moment exit time} Let $d\geq 1$. For every $n,m\geq 1$ with $n\geq m$, every $\mu \in \mathcal P_m$, and every $u\in \Lambda_n$,
\begin{equation}
	\mathbb E_{\mu,u}[\tau_n]\leq 9d\left(\frac{n}{m}\right)^2.
\end{equation}

\end{Prop}
\begin{proof} 
Let $n,m\geq 1$ with $n\geq m$. Fix $\mu \in \mathcal P_m$ and drop it from the notation. Also, let $u=(u_1,\ldots,u_d)\in \Lambda_n$. Let $X^{(i)}$ be the projection of $X$ onto the $i$-th coordinate. Note that if $X\sim \mathbb P_{0}$, the process $X^{(i)}$ is a random walk $S$ on $\mathbb Z$, whose law does not depend on $i$ by symmetries of $\mu$. We denote by $\mathbf P_x$ the law of $S$ started at $x\in \mathbb Z$ and observe that, by symmetry of $\mathbb P_0$, one has $\mathbf E_0[|S_1|^2]\geq \tfrac{1}{d}m^2$.  Introduce $\boldsymbol{\tau}_{n}:=\inf\{k\geq 1: |S_k|\geq n+1\}$. Noticing that $\tau_n=\inf_{1\leq i \leq d}\inf\{k\geq 1: |X_k^{(i)}|\geq n+1\}$, we get
\begin{equation}\label{eq:estimate srw 1}
	\mathbb E_u[\tau_n]\leq \max_{1\leq i\leq d}\mathbf E_{u_i}[\boldsymbol{\tau}_{n}].
\end{equation}
Following a classical argument, we observe that $(S_k^2-\mathbf E_0[|S_1|^2]k)_{k\geq 0}$ is a martingale, and use the stopping time theorem to argue that $\mathbf E_{u_i}[S^2_{\boldsymbol{\tau}_n}-\mathbf E_0[|S_1|^2]\boldsymbol{\tau}_n]=u_i^2$, which gives
\begin{equation}\label{eq:estimate srw2}
	\mathbf E_{u_i}[\boldsymbol{\tau}_n]\leq (n+2m)^2\mathbf E_0[|S_1|^2]^{-1}\leq 9d\left(\frac{n}{m}\right)^2.
\end{equation}
Plugging \eqref{eq:estimate srw2} in \eqref{eq:estimate srw 1} yields the result.
%
\end{proof}

We will also need the following result.
\begin{Cor}\label{cor: estimate srw} Let $d\geq 1$ and $A,\eta>0$. Let $n,m\geq 1$ with $(n/m)\leq A$. There exist $T,\varepsilon>0$ which depend on $(A, \eta,d)$ such that the following holds. For every $\mu \in \mathcal P_m$, every $\varphi\in [1-\varepsilon,1+\varepsilon]$, every $f:\mathbb Z^d\rightarrow \mathbb R^+$, and every $u \in \Lambda_n$,
\begin{equation}
	\Big|\mathbb E_{\mu,u}[\varphi^{\tau_n\wedge T}f(X_{\tau_n\wedge T})]-\mathbb E_{\mu,u}[f(X_{\tau_n})]\Big|\leq \eta \max\{f(w):w\in \Lambda_{n+2m}\}.
\end{equation}
\end{Cor}
\begin{proof} 
Let $n,m\geq 1$ with $(n/m)\leq A$. Let $\mu \in \mathcal P_m$ and drop it from the notation.
Let $T\geq 1$ to be fixed and set $\varepsilon=\varepsilon(T)>0$ such that $(1+\varepsilon)^T\leq 1+\eta/4$ and $(1-\varepsilon)^T\geq 1-\eta/4$. Let $\varphi\in [1-\varepsilon,1+\varepsilon]$. Observe that, for $f:\mathbb Z^d\rightarrow \mathbb R^+$, and $u\in \Lambda_n$,
\begin{multline}\label{eq:coro srw2}
	\Big|\mathbb E_u[\varphi^{\tau_n\wedge T}f(X_{\tau_n\wedge T})]-\mathbb E_u[f(X_{\tau_n})]\Big|\leq 2\max\{f(w):w\in \Lambda_{n+2m}\} \cdot\Big(\frac{\eta}{4}+\mathbb P_u[\tau_n>T]\Big).
\end{multline}
Using Markov's inequality and Proposition \ref{prop: exponential moment exit time}, we find
\begin{equation}\label{eq:coro srw3}
	\mathbb P_u[\tau_n>T]\leq \frac{9dA^2}{T}.
\end{equation}
Now, we choose $T=T(A,\eta,d)>0$ large enough so that $9dA^2/T\leq \eta/4$. Combining \eqref{eq:coro srw2} and \eqref{eq:coro srw3} with such $T,\varepsilon$ gives
\begin{equation}\label{eq:coro srw4}
	\Big|\mathbb E_u[\varphi^{\tau_n\wedge T}f(X_{\tau_n\wedge T})]-\mathbb E_u[f(X_{\tau_n})]\Big|\leq \eta\max\{f(w):w\in \Lambda_{n+2m}\}.
\end{equation}	
This concludes the proof.
\end{proof}

We now prove the Harnack-type estimate used in the proof of Proposition \ref{prop:regularity perco 2}. Let us start with some additional notations. For $x\in \mathbb R^d$, we let $\mathsf{P}_x$ denote the law of the standard Brownian motion $B$ on $\mathbb R^d$ started at $x$. We also let $\mathsf P_t^{(1)}$ be the law of the one-dimensional standard Brownian motion $B^{(1)}$ started at $t\in \mathbb R$. For simplicity, we omit integer rounding from our notation.

\begin{Prop}[Uniform Harnack-type estimate]\label{prop:uniform Harnack} Let $d\geq 3$, $\alpha>0$, and $\eta>0$. There exists $C_{\rm RW}=C_{\rm RW}(\alpha,d)>0$ and $N_1=N_1(\eta,\alpha,d)>0$ such that the following holds. For every $n,m\geq 1$ satisfying $\tfrac{n}{m}\geq N_1$, every $\mu\in \mathcal P_m$, every $f:\mathbb Z^d\rightarrow \mathbb R^+$, and every $u,v\in \Lambda_n$,
\begin{multline}
	\mathbb E_{\mu,u}[f(X_\tau)]\leq C_{\rm RW}\mathbb E_{\mu,v}[f(X_\tau)] +\eta\max\{f(w):w\in \Lambda_{(1+\alpha)n+2m}\}\\+2C_{\rm RW}\max\Big\{|f(w)-f(w')|: w,w'\in \Lambda_{3m(n/m)^{1/10}}(z), \: z\in \partial \Lambda_{(1+\alpha)n}\Big\},
\end{multline}
where $\tau:=\inf \{k\geq 0: X_k\notin \Lambda_{(1+\alpha)n}\}$.
\end{Prop}
\begin{proof} We fix $m\geq 1$ and let $n\geq m$ to be chosen large enough. We fix $\mu \in \mathcal P_m$ and drop it from the notation. The idea is to couple the trajectory of the random walk $X$ with a Brownian motion, and transfer the Harnack inequality from the continuum to the discrete. Below, we view $f$ as a function defined on  $[-(1+\alpha)n,(1+\alpha)n]^d$ by setting, for $x\in [-(1+\alpha)n,(1+\alpha)n]^d$, $f(x):=f([x])$ where $[x]:=(\lfloor x_1\rfloor,\ldots,\lfloor x_d\rfloor)$. Fix $\eta>0$ and let $\eta'$ to be fixed in terms of $\eta, \alpha$ and $d$.

We will use a result of Zaitsev \cite{zaitsev1998multidimensional} that can be stated as follows: there exist $c_0,t_0>0$ such that, for every $m\geq 1$ and every $u\in \mathbb Z^d$, there exists a coupling $(X,B)\sim \mathbf P_{u,m}$ where $X\sim \mathbb P_u$ and $B\sim \mathsf{P}_{u/\sigma(m)}$, such that for every $t\geq t_0$, and every $K\geq 1$,
\begin{equation}\label{eq:BM 0}
	\mathbf P_{u,m}\left[\max_{0\leq k \leq K}|X_k-\sigma(m)B_k|\geq t\sigma(m)\log K\right]\leq e^{-c_0 t \log K}.
\end{equation}
Below, we let $K:=\Big(\tfrac{(1+\alpha)n}{m}\Big)^3$ and fix $t=t_0$. For this choice, we write
\begin{equation}
	\mathcal G_1:=\Big\{\max_{0\leq k \leq K}|X_k-\sigma(m)B_k|\leq t\sigma(m)\log K\Big\},
\end{equation}
and observe that, from \eqref{eq:BM 0},
\begin{equation}\label{eq:BM 1}
	\mathbf P_{u,m}[\mathcal G_1^c]\leq \left(\frac{m}{(1+\alpha)n}\right)^{3c_0t}\leq \eta',
\end{equation}
for $n/m$ large enough (in terms of $\eta',\alpha$ and $d$).
We also introduce the following stopping times:
\begin{align}
	\tau_1&:=\inf\{k\geq 0: X_k\notin \Lambda_{(1+\alpha)n}\},
	\\\tau_2&:=\inf\{t\geq 0: \sigma(m)B_t\notin [-(1+\alpha)n,(1+\alpha)n]^d\},
\end{align}
and let
\begin{equation}
	\mathcal G_2:=\{\tau_1\leq K\}\cap \{\tau_2\leq K\}.
\end{equation}
\begin{Claim}\label{claim1} For $n/m$ large enough (in terms of $\eta',\alpha$ and $d$), and for every $u \in \Lambda_n$, \begin{equation}\label{eq:BM 2}
	\mathbf P_{u,m}[\mathcal G_2^c]\leq \eta'.
\end{equation}
\end{Claim}
\begin{proof}[Proof of Claim \textup{\ref{claim1}}] We simply use Markov's inequality and Proposition \ref{prop: exponential moment exit time} to argue that
\begin{equation}
	\mathbf P_{u,m}[\tau_1>K]\leq \frac{\mathbf E_u[\tau_1]}{K}\leq 9d\left(\frac{(1+\alpha)n}{m}\right)^2\left(\frac{(1+\alpha)n}{m}\right)^{-3}\leq \frac{\eta'}{2},
\end{equation}
when $n/m$ is large enough. A similar argument holds to bound\footnote{For full disclosure, the martingale argument used in the proof of Proposition \ref{prop: exponential moment exit time} to bound the expected exit time can be adapted to the Brownian motion by noticing that $(B_t^2-t)_{t\geq 0}$ is a continuous martingale (see for instance \cite[Chapter~3.3]{le2016brownian}).} $\mathbf P_{u,m}[\tau_2>K]$.
\end{proof}
As a consequence of the above, we find that, the trajectories of $X$ and $B$ remain ``close to each other'' until $\tau_1\vee \tau_2$ with high probability (under $\mathbf P_{u,m}$). Our next claim is that $X_{\tau_1}$ and $\sigma(m)B_{\tau_2}$ are close to each other with high probability. We let
\begin{equation}
	\mathcal G_3:=\Big\{|X_{\tau_1}-\sigma(m)B_{\tau_2}|\leq 2m(n/m)^{1/10}\Big\}.
\end{equation}
\begin{Claim}\label{claim2} For $n/m$ large enough (in terms of $\eta',\alpha$ and $d$), and for every $u \in \Lambda_n$,
\begin{equation}\label{eq:BM 2.5}
	\mathbf P_{u,m}\Big[\mathcal G_1 \cap \mathcal G_2 \cap \mathcal G_3^c\Big]\leq\eta'.
\end{equation}
\end{Claim}
\begin{proof}[Proof of Claim \textup{\ref{claim2}}] We fix $u\in \Lambda_n$, and assume that $\mathcal G_1\cap \mathcal G_2\cap \mathcal G_3^c$ occurs. We write $\mathbf P=\mathbf P_{u,m}$ and consider either two cases: $\tau_1<\tau_2$ or $\tau_2\geq \tau_1$. 

\paragraph{Case $\tau_1<\tau_2$.} We begin with the former case. Define the Brownian motion $B':=(B_{t+\tau_1})_{t\geq 0}$. If $\mathcal G_3^c$ occurs, $\sigma(m)B'$ has to travel a long distance--- and thus for a long time--- close to the boundary of $[-(1+\alpha)n,(1+\alpha)n]^d$, without exiting it. This last event is very unlikely thanks to (classical) Gambler's ruin type estimates.

 More precisely, let $1\leq i \leq d$ be such that the $i$-the coordinate of $X$ satisfies $|X_{\tau_1}^{(i)}|\geq (1+\alpha)n$. Let $\tilde B:=\tfrac{(1+\alpha)n}{\sigma(m)}-(B')^{(i)}$. The process $\tilde B$ is has law $\mathsf P^{(1)}_{\tilde B_0}$ with $0<\tilde B_0\leq t\log K$. A classical Brownian motion computation (see for instance \cite[Chapter~2.8]{karatzas1991brownian}) gives $c_1,C_1>0$ (which only depend on $d$) such that, for every $T\geq 1$,
 \begin{align}\notag
 	\mathbf P_{u,m}\Big[\max_{0\leq t \leq T} |B'_{t}-B'_0|\geq (n/m)^{1/10}\Big]&\leq 2d\cdot\mathsf P_{0}^{(1)}\Big[\max_{0\leq t \leq 1} |B_{t}^{(1)}|
 	\geq \frac{(n/m)^{1/10}}{\sqrt{T}}\Big]\\&\leq C_1\exp\left(-c_1\frac{(n/m)^{1/5}}{T}\right).
 \end{align}
 Hence, choosing $T=(\log (n/m))^2$ and $(n/m)$ large enough, gives
 \begin{equation}\label{eq:BM 3}
 	\mathbf P_{u,m}\Big[\max_{0\leq t \leq T} |B'_{t}-B'_0|\geq (n/m)^{1/10}\Big]\leq \frac{\eta'}{4}.
 \end{equation}
 Now, if $\{\tau_1<\tau_2\}\cap\mathcal G_1\cap \mathcal G_2\cap \mathcal G_3^c$ occurs, then (recall that $m\geq \tfrac{1}{2\sqrt{d}}\sigma(m)$) 
 \begin{multline}\label{eq:BM 4}
 |\sigma(m)B_{\tau_1}-\sigma(m)B_{\tau_2}|\geq |X_{\tau_1}-\sigma(m)B_{\tau_2}|-|X_{\tau_1}-\sigma(m)B_{\tau_1}|\\\geq 2m(n/m)^{1/10}- t\sigma(m)\log K\geq \sigma(m)(n/m)^{1/10},
 \end{multline}
 provided $(n/m)$ is large enough.
As a consequence of \eqref{eq:BM 4}, we get 
 \begin{equation}\label{eq:BM 5}
 	\{0\leq \tau_2-\tau_1\leq (\log (n/m))^2\}\cap \mathcal G_1\cap \mathcal G_2\cap \mathcal G_3^c\subset \Big\{\max_{0\leq t \leq (\log(n/m))^2} |B'_{t}-B'_0|\geq (n/m)^{1/10}\Big\}.
 \end{equation}
 Introduce $\tau_{\tilde B}:=\inf\{t\geq 0: \tilde B_t\leq 0\}$. Collecting the above work, we obtain that \begin{align}\notag
 	\mathbf P_{u,m}[&\{\tau_1<\tau_2\}\cap \mathcal G_1\cap \mathcal G_2\cap \mathcal G_3^c]\\&\leq \mathbf P_{u,m}[\{0<\tau_2-\tau_1\leq (\log (n/m))^2\}\cap \mathcal G_1\cap \mathcal G_2\cap \mathcal G_3^c]+\mathbf P_{u,m}[\{\tau_2-\tau_1\geq (\log (n/m))^2\}\cap \mathcal G_1]\notag\\&\leq \frac{\eta'}{4}+\mathbf P_{u,m}[ \{\tau_{\tilde B} \geq (\log (n/m))^2\}\cap \{0<\tilde B_0\leq t\log K\}]\label{eq:BM 6},
 \end{align}
 where we used \eqref{eq:BM 3} and \eqref{eq:BM 5} in the second inequality. Then, thanks to the reflection principle (see for instance \cite[Chapter~2.6]{karatzas1991brownian}), setting $\tau^{(1)}:=\inf \{t\geq 0: B_t^{(1)}\leq 0\}$,
 \begin{align}\notag
 \mathbf P_{u,m}[ \{\tau_{\tilde B} \geq (\log (n/m))^2\}\cap &\{0<\tilde B_0\leq t\log K\}]\\&\leq \sup_{0<v\leq t\log K}\mathsf P_v^{(1)}[\tau^{(1)} \geq (\log(n/m))^2]\notag\\&= \sup_{0<v\leq t\log K}
 \int_{(\log(n/m))^2}^\infty \frac{v}{\sqrt{2\pi s^3}}\exp\left(-\frac{v^2}{2s}\right)\mathrm{d}s\notag
 \\&\leq \frac{\eta'}{4} \label{eq:BM 7},
 \end{align}
 where in the last line, we chose $(n/m)$ large enough. Thus, combining \eqref{eq:BM 6} and \eqref{eq:BM 7}, we have obtained,
 \begin{equation}\label{eq:BM 7.5}
 	\mathbf P_{u,m}[\{\tau_1<\tau_2\}\cap \mathcal G_1\cap \mathcal G_2\cap 
 	\mathcal G_3^c]\leq \frac{\eta'}{2}.
 \end{equation} 
 \paragraph{Case $\tau_1\geq \tau_2$.} We now turn to the case $\tau_1\geq\tau_2$. This case is essentially symmetric to the first one. Indeed, as in \eqref{eq:BM 4}, if $\{\tau_2\geq \tau_1\}\cap \mathcal G_1\cap \mathcal G_2\cap \mathcal G_3^c$ occurs, then
 \begin{equation}
 	|X_{\tau_2}-X_{\tau_1}|\geq \sigma(m)(n/m)^{1/10}.
 \end{equation}
 However, under the occurrence of the same events, this also implies that
 \begin{equation}
 	|\sigma(m)B_{\tau_2}-\sigma(m)B_{\tau_1}|\geq |X_{\tau_1}-X_{\tau_2}|-|X_{\tau_1}-\sigma(m)B_{\tau_1}|-|X_{\tau_2}-\sigma(m)B_{\tau_2}|\geq \frac{1}{2}\sigma(m)(n/m)^{1/10},
 \end{equation}
 for $(n/m)$ large enough. As a result, one may argue as \eqref{eq:BM 5} in to obtain that
 \begin{equation}\label{eq:BM 8}
 	\mathbf P_{u,m}[\{0\leq \tau_1-\tau_2\leq (\log (n/m))^2\}\cap \mathcal G_1\cap \mathcal G_2\cap \mathcal G_3^c]\leq \frac{\eta'}{4}
 \end{equation}
 for $(n/m)$ large enough. Now, let again $1\leq i \leq d$ be such that $|B_{\tau_2}^{(i)}|=(1+\alpha)n$, and set $\tilde X:= ((1+\alpha)n-X^{(i)}_{k+\lfloor \tau_2\rfloor})_{k\geq 0}$. Thanks to a similar computation as in \eqref{eq:BM 3}, one has that $|B_{\tau_2}-B_{\lfloor \tau_2\rfloor}|\leq t\log K$ with probability at least $1-\eta'/8$ for $(n/m)$ large enough. Under this last event and $\{\tau_1\geq \tau_2\}\cap \mathcal G_1\cap \mathcal G_2\cap \mathcal G_3^c$, the process $\tilde X$ is a one-dimensional random walk started at $v\in \mathbb Z$ with $0\leq v \leq 2\sigma(m)t\log K$. We denote by $\tilde{\mathbb P}_v$ its law and let $\tau_{\tilde X}:=\inf\{k\geq 1: \tilde X_k\leq 0\}$. A Gambler's ruin estimate applied to the process $\tilde X/\sigma(m)$ (see for instance \cite[Theorem~5.1.7]{LawlerLimicRandomWalks2010}), gives the existence of $C_2>0$ such that
 \begin{multline}
 	\mathbf P_{u,m}[\{\tau_1-\tau_2\geq (\log (n/m))^2\}\cap \mathcal G_1\cap \mathcal G_2\cap \mathcal G_3^c\cap \{|B_{\tau_2}-B_{\lfloor \tau _2\rfloor}|\leq t\log K\}]\\\leq \sup_{0\leq v\leq 2t\log K}\tilde{\mathbb P}_v[\tau_{\tilde X} \geq (\log (n/m))^2]\leq \frac{C_2(2t\log K)}{(\log(n/m))^2}\leq \frac{\eta'}{8},
 \end{multline}
 if $(n/m)$ is large enough. Combining the previously displayed equation and \eqref{eq:BM 8} gives that
 \begin{equation}\label{eq:BM 9}
 	\mathbf P_{u,m}[\{\tau_1\geq \tau_2\}\cap \mathcal G_1\cap \mathcal G_2\cap\mathcal G_3^c]\leq \frac{\eta'}{2}.
 \end{equation}
 Combining \eqref{eq:BM 7.5} and \eqref{eq:BM 9} yields the result, when $(n/m)$ is chosen large enough (in terms of $\eta',\alpha$ and $d$).
\end{proof}

With Claims \ref{claim1} and \ref{claim2} in hands, we are in a position to conclude. Set $\mathcal G:=\mathcal G_1\cap \mathcal G_2\cap \mathcal G_3$. Then, for every $u \in \Lambda_n$,
\begin{equation}
	\mathbb E_u[f(X_\tau)]=\mathbf E_{u,m}[f(X_{\tau_1})]\leq \mathbf E_{u,m}[f(X_{\tau_1})\mathds{1}_{\mathcal G}]+3\eta'\max\{f(w):w\in \Lambda_{(1+\alpha)n+2m}\},
\end{equation}
where we used \eqref{eq:BM 1}, \eqref{eq:BM 2}, and \eqref{eq:BM 2.5} to argue that $\mathbf P_{u,m}[\mathcal G^c]\leq 3\eta'$, and the fact that the step distribution $\mu$ has range $2m$. Now, observe that
\begin{multline}\label{eq:BM 10}
	\mathbf E_{u,m}[f(X_{\tau_1})\mathds{1}_{\mathcal G}]\leq \mathbf E_{u,m}[f(\sigma(m)B_{\tau_2})]\\+\max\Big\{|f(w)-f(w')|: w,w'\in \Lambda_{3m(n/m)^{1/10}}(z), \: z\in \partial \Lambda_{(1+\alpha)n}\Big\}.
\end{multline}
Recall that under $\mathbf E_{u,m}$, $B$ is a Brownian motion started at $u/\sigma(m)$. Note that by the scaling properties of $B\sim \mathsf P_u$, the process $\overline{B}:=((n/m)^{-1/2}B_{t(n/m)})_{t\geq 0}$ is a Brownian motion started at $u(n/m)^{-1/2}$. Moreover, $(n/m)\tau_2$ has the same distribution as $\overline{\tau}:=\inf\{t\geq 0: |\overline{B}_t|=(1+\alpha)\}$. The function $\varphi:x\in (-(1+\alpha),(1+\alpha))^d\mapsto \mathsf{E}_x[f(\sigma(m)(n/m)^{1/2}B_{\overline{\tau}})]\in\mathbb R^+$ is harmonic (see for instance \cite[Chapter~4.2]{karatzas1991brownian}). Thanks to the Harnack inequality, there exists $C_{\rm RW}=C_{\rm RW}(\alpha,d)\geq 1$ such that, for every $x,y\in [-1,1]^d$,
\begin{equation}
	\varphi(x)\leq C_{\rm RW} \varphi(y).
\end{equation}
Hence, if $v\in \Lambda_n$, one has,
\begin{equation}\label{eq:BM 11}
	\mathbf E_{u,m}[f(\sigma(m)B_{\tau_2})]\leq C_{\rm RW}\mathbf E_{v,m}[f(\sigma(m)B_{\tau_2})].
\end{equation}
Now, observe, that
\begin{align}
	\mathbf E_{v,m}[f(&\sigma(m)B_{\tau_2})]\notag\\&\leq \mathbf E_{v,m}[f(\sigma(m)B_{\tau_2})\mathds{1}_{\mathcal G}]+3\eta'\max\{f(w):w\in \Lambda_{(1+\alpha)n+2m}\}\notag
	\\&\leq \mathbf E_{v,m}[f(X_{\tau_1})]+3\eta'\max\{f(w):w\in \Lambda_{(1+\alpha)n+2m}\}\notag\\&\qquad+\max\Big\{|f(w)-f(w')|: w,w'\in \Lambda_{3m(n/m)^{1/10}}(z), \: z\in \partial \Lambda_{(1+\alpha)n}\Big\}\label{eq:BM 12},
\end{align}
where we used that $\mathbf P_{v,m}[\mathcal G^c]\leq 3\eta'$ in the first inequality.
Plugging \eqref{eq:BM 11} and \eqref{eq:BM 12} in \eqref{eq:BM 10}, and choosing $\eta'=\frac{\eta}{6C_{\rm RW}}$ (which requires to choose $n/m$ large enough in terms of $\eta,\alpha$ and $d$) concludes the proof.
\end{proof}

\section{Appendix: a convolution estimate}\label{appendix:convolution}
\begin{Prop}\label{prop: convolution estimate} Let $d>4$ and $L\geq 1$. Let $f:\mathbb Z^d\rightarrow \mathbb R^+$. Assume that there exists $\bfC>0$ such that for all $x\in \mathbb Z^d$,
\begin{equation}\label{eq:assumption convolution estimate}
	f(x)\leq \mathds{1}_{x=0}+\frac{\bfC}{L^d}\left(\frac{L}{L\vee |x|}\right)^{d-2}.
\end{equation}
Then, there exists $A=A(\bfC,d)>0$ (independent of $L$) such that, for all $x\in \mathbb Z^d$,
\begin{equation}
	(f*f)(x):=\sum_{y\in \mathbb Z^d}f(y)f(x-y)\leq A\Big(\mathds{1}_{x=0}+\mathds{1}_{x\neq 0}\frac{1}{L^4}{\left(\frac{1}{L\vee |x|}\right)^{d-4}}\Big).
\end{equation}
\end{Prop}
\begin{proof} Below, the constants $C_i$ depend on $\bfC$ and $d$ only. The case $x=0$ follows from the hypothesis that $d>4$. We turn to the case $x\neq 0$, which we split into two sub-cases $|x|\leq 2L$ and $|x|>2L$. Note that for $|x|\leq 2L$, it is sufficient to obtain a bound by $C_0/L^d$ since one has
\begin{equation}
	\frac{1}{L^d}\leq \frac{2^{d-4}}{L^4}\left(\frac{1}{L\vee |x|}\right)^{d-4}.
\end{equation}

\paragraph{Case $|x|\leq 2L$.} We write
\begin{equation}
	\sum_{y\in \Lambda_{4L}\setminus\{0,x\}}f(y)f(x-y)\stackrel{\eqref{eq:assumption convolution estimate}}\leq |\Lambda_{4L}|\frac{\bfC^2}{L^{2d}}\leq \frac{C_1}{L^d}.
\end{equation}
Moreover,
\begin{equation}
	\sum_{y\notin\Lambda_{4L}}f(y)f(x-y)\stackrel{\eqref{eq:assumption convolution estimate}}\leq \frac{C_2}{L^4}\sum_{|u|\geq 2L}\frac{1}{|u|^{2d-4}}\leq \frac{C_3}{L^d}.
\end{equation}
The proof in this case follows from the two last displayed equations and the fact that $f(0)f(x)\leq \tfrac{C_4}{L^d}$.
\paragraph{Case $|x|>2L$.} We start by noticing that
\begin{align}\notag
	\sum_{y\in \Lambda_{|x|/2}\cup \Lambda_{|x|/2}(x)}f(y)f(x-y)&\stackrel{\phantom{\eqref{eq:assumption convolution estimate}}}\leq {2\Big(\max_{u\in \Lambda_{|x|/2}}f(x+u)\Big)\sum_{y\in \Lambda_{|x|/2}}f(y)}
	\\&\stackrel{\eqref{eq:assumption convolution estimate}}\leq \frac{C_5}{L^2|x|^{d-2}}\sum_{y\in \Lambda_{|x|/2}}f(y)\notag
	\\&\stackrel{\eqref{eq:assumption convolution estimate}}\leq \frac{C_6}{L^2|x|^{d-2}}\frac{|x|^2}{L^2}=\frac{C_6}{L^4|x|^{d-4}}.
\end{align}
Then,
\begin{equation}
	\sum_{y\notin \Lambda_{|x|/2}\cup \Lambda_{|x|/2}(x)}f(y)f(x-y)\stackrel{\eqref{eq:assumption convolution estimate}}\leq \frac{C_7}{L^4}\sum_{|u|\geq|x|/2}\frac{1}{|u|^{2d-4}}\leq \frac{C_8}{L^4|x|^{d-4}}.
\end{equation}
The proof follows readily.
\end{proof}
\section{Appendix: proof of Lemma \ref{lem: estimate E(x)}}\label{appendix:boundE}
Fix $d>6$, $\bfC>1$, and $L\geq 1$. Also, fix $\beta<\beta^*(\bfC,L)$ and drop it from the notations. Let $n>24L$, $x\in \partial \mathbb H_n$, $w \in \Lambda_{n/2}$ and $B\in \mathcal B$ such that $B+w\subset\Lambda_{n/2}$. Note that by translation invariance, $E(B+w,\mathbb H_n,w,x)=E(B,\mathbb H_{n-w_1},0,x)$. Since $w\in \Lambda_{n/2}$, we can assume that $w=0$ and $n>12L$. By definition,
\begin{align}
	&E(B,\mathbb H_n,0,x)=\notag
	\\&\sum_{\substack{u\in B\\v\in B}}\sum_{\substack{y\in B\\z\in \mathbb H_n\setminus B}}\mathbb P[0\connect{B\:}u]\mathbb P[u\connect{B\:}y]p_{yz}\mathbb P[z\connect{\mathbb H_n\:}v]\mathbb P[u\connect{\mathbb H_n\:}v]\mathbb P_\beta[v\connect{\mathbb H_n\:}x]\notag
	\\&+ \sum_{\substack{u\in B\\v\in \mathbb H_n}}\sum_{\substack{y,s\in B\\ z,t\in \mathbb H_n \setminus B\\ yz\neq st}}\mathbb P[0\connect{B\:}u]\mathbb P[u\connect{B\:}y]\mathbb P_\beta[u\connect{B\:}s]p_{yz}p_{st}\mathbb P[z\connect{\mathbb H_n\:}v]\mathbb P[t\connect{\mathbb H_n\:}v]\mathbb P[v\connect{\mathbb H_n\:}x]\label{eq:proofE(x)1}.
\end{align}
We split $E(B,\mathbb H_n,0,x)$ into a contribution coming from $v\in \mathbb H_{3n/4}$ that we denote $(I)$, and a contribution coming from $v\in \mathbb H_{n}\setminus\mathbb H_{3n/4}$ that we denote $(II)$.
\paragraph{Bound on $(I)$.}
Notice that
\begin{equation}
	\max_{v\in \mathbb H_{3n/4}}\mathbb P[v\connect{\mathbb H_n\:}x]\stackrel{\eqref{eq:H_beta-' perco}}\leq \frac{\bfC}{L}\frac{1}{(n/4)^{d-1}}. 
\end{equation}
Recall the definition of the error amplitude $E(B)$. By monotonicity and Proposition~\ref{lem: error amplitude perco}, 
\begin{equation}
	(I)\leq E(B)\cdot 4^{d-1}\frac{\bfC}{Ln^{d-1}}\leq \frac{4^{d-1}K}{L^d}\frac{\bfC}{Ln^{d-1}}.
\end{equation}
\paragraph{Bound on $(II)$.} We turn to the contribution for $v\in \mathbb H_n\setminus \mathbb H_{3n/4}$. This contribution is only coming from \eqref{eq:proofE(x)1}. Notice that $z,t$ contribute if they are at a distance at most $L$ from $B$, that is $z,t\in \Lambda_{n/2+L}\subset \Lambda_{n/2+n/12}$. If $p\in \{0,\ldots,n/4-1\}$, $v\in \partial \mathbb H_{n-p}$, and $z,t$ are as above, then $|z-v|,|t-v|\geq n/6$ and
\begin{align}\label{eq:appendixC1}
	\mathbb P[z\connect{\mathbb H_n\:}v]\mathbb P[t\connect{\mathbb H_n\:}v]&\stackrel{\eqref{eq: full plane estimate from half plane perco}}\leq \frac{9\bfC^4}{L^4}\frac{6^{2d-4}}{n^{2d-4}}.
\end{align}
Moreover, using translation invariance,
\begin{equation}\label{eq:appendixC2}
	\sum_{v\in \partial\mathbb H_{n-p}}\mathbb P[v\connect{\mathbb H_n\:}x]= {\mathds{1}_{p=0}}+\psi_\beta(\mathbb H_{p})\stackrel{\eqref{eq:H_beta perco}}\leq {\mathds{1}_{p=0}}+\frac{\bfC}{L}.
\end{equation}
For a fixed $u\in B$, Proposition \ref{prop:bound phi perco} gives
\begin{equation}\label{eq:appendixC3}
	\sum_{\substack{y, s\in B\\ z,t\in \mathbb H_n \setminus B\\ yz\neq st}}\mathbb P[u \connect{B\:}y]\mathbb P[u\connect{B\:}s]p_{yz}p_{st}\leq \varphi(B-u)^2\leq \Big( 1+\frac{K}{L^d}\Big)^2{\leq (1+K)^2}.
\end{equation}
Finally, we use \eqref{eq: full plane estimate from half plane perco} to get $C_1=C_1(\bfC,d)>0$ such that
\begin{equation}\label{eq:appendixC4}
	\sum_{u\in B}\mathbb P[0\connect{B\:}u] 
	\leq
	\sum_{u\in \Lambda_{n/2}}\mathbb P[0\connect{}u]
	\leq 1+C_1\left(\frac{n}{L}\right)^2.
\end{equation}

Putting all the previous displayed equations together we obtain $C_2,C_3,C_4>0$ which only depend on $\bfC$ and $d$ such that
\begin{align*}
	(II)&\stackrel{\eqref{eq:appendixC1}}\leq \frac{C_2}{L^4n^{2d-4}}\sum_{p=0}^{n/4-1}\sum_{\substack{v\in \partial \mathbb H_{n-p}\\u\in B}}\sum_{\substack{y,s\in B\\ z,t\in \mathbb H_n \setminus B\\ yz\neq st}}\mathbb P[0\connect{B\:}u]\mathbb P[u\connect{B\:}y]\mathbb P_\beta[u\connect{B\:}s]p_{yz}p_{st}\mathbb P[v\connect{\mathbb H_n\:}x]
	\\&\stackrel{\eqref{eq:appendixC2}}\leq \frac{C_2}{L^4n^{2d-4}}\sum_{p=0}^{n}\Big(\mathds{1}_{p=0}+\frac{\bfC}{L}\Big)\sum_{u\in B}\mathbb P[0\connect{B\:}u] \sum_{\substack{y,s\in B\\ z,t\in \mathbb H_n \setminus B\\ yz\neq st}}\mathbb P[u\connect{B\:}y]\mathbb P_\beta[u\connect{B\:}s]p_{yz}p_{st}
	\\&\stackrel{\eqref{eq:appendixC3}}\leq \frac{C_3}{L^4n^{2d-4}}\sum_{p=0}^{n-1}\Big(1+\frac{\bfC n}{L}\Big)\sum_{u\in B}\mathbb P[0\connect{B\:}u]
	\\&\stackrel{\eqref{eq:appendixC4}}\leq \frac{C_3}{L^4n^{2d-4}}\Big(1+\frac{ \bfC n}{L}\Big)\Big(1+C_1\left(\frac{n}{L}\right)^2\Big) \leq \frac{C_4}{L^3n^{d-3}}\left(\frac{n}{L}\right)^3\frac{\bfC}{Ln^{d-1}}\stackrel{n>12L}\leq \frac{C_4}{L^d}\frac{\bfC}{Ln^{d-1}},
	\end{align*}
where we used $d>6$ in the last inequality.
The result follows from setting $D:=4^{d-1}K+C_4$.

\section{Appendix: proof of Lemma \ref{lem: bound non local error singleton perco}}\label{appendix:boundEaverage}

Fix $d>6$, and let $C,L_0\geq 1$ be given by Proposition \ref{prop: final prop proof mainperco thm}. Finally, let $L\geq L_0$, and $n\geq 0$. Recall \eqref{eq:error first term}--\eqref{eq:error second term}. We begin by proving \eqref{eq:non local term 1}. If $w\in \mathbb H_n$,
\begin{align*}
	\sum_{x\in \partial \mathbb H_n}{E}_\beta&(\{w\},\mathbb H_n,w,x)\\
	&=\sum_{x\in \partial \mathbb H_n}\sum_{z\in \mathbb H_n\setminus \{ w\}}p_{wz,\beta}\mathbb P_\beta[z\connect{\mathbb H_n\:}w]\mathbb P_\beta[w\connect{\mathbb H_n\:}x]
	\\&+\sum_{x\in \partial \mathbb H_n}\sum_{v\in \mathbb H_n}\sum_{\substack{z,t\in \mathbb H_n\setminus \{w\}\\z\neq t}}p_{wz,\beta}p_{wt,\beta}\mathbb P_\beta[z\connect{\mathbb H_n\:}v]\mathbb P_\beta[t\connect{\mathbb H_n\:}v]\mathbb P_\beta[v\connect{\mathbb H_n\:}x]
	\\&=:(I)+(II).
\end{align*}
\paragraph{Bound on $(I)$.} By \eqref{eq: bound volume times p} and Proposition \ref{prop: final prop proof mainperco thm}, one has
\begin{equation}
	(I)\leq \frac{C}{L}\sum_{z\in \mathbb H_n\setminus \{w\}}p_{wz,\beta}\mathbb P_\beta[z\connect{\mathbb H_n\:}w]\leq \frac{C}{L}\frac{C}{L^d}|\Lambda_L|p_\beta\leq \frac{4C}{L}\frac{C}{L^d}.
\end{equation}
\paragraph{Bound on $(II)$.} On split $(II)$ into two contributions: $v\in \mathbb H_{n-1}$ and $v\in \partial \mathbb H_n$. The former contribution is bounded by 
\begin{multline}
	\sum_{v\in \mathbb H_{n-1}}\sum_{\substack{z,t\in \mathbb H_n\setminus \{w\}\\z\neq t}}p_{wz,\beta}p_{wt,\beta}\mathbb P_\beta[z\connect{\:}v]\mathbb P_\beta[t\connect{\:}v]\Big(\sum_{x\in \partial \mathbb H_n}\mathbb P_\beta[v\connect{\mathbb H_n\:}x]\Big)\\\leq \frac{C}{L} E_\beta(\{w\})\leq \frac{C}{L}\frac{C}{L^d},
\end{multline}
where we used Proposition \ref{prop: final prop proof mainperco thm} twice.
Then, one more use of Proposition \ref{prop: final prop proof mainperco thm} gives
\begin{align*}
	\sum_{v\in \partial \mathbb H_n}\sum_{\substack{z,t\in \mathbb H_n\setminus \{w\}\\z\neq t}}p_{wz,\beta}p_{wt,\beta}&\mathbb P_\beta[z\connect{\mathbb H_n\:}v]\mathbb P_\beta[t\connect{\mathbb H_n\:}v]\Big(\sum_{x\in \partial\mathbb H_n}\mathbb P_\beta[v\connect{\mathbb H_n\:}x]\Big)
	\\ &\leq \Big(1+\frac{C}{L}\Big)
\sum_{v\in \partial \mathbb H_n}\sum_{\substack{z,t\in \mathbb H_n\setminus \{w\}\\z\neq t}}p_{wz,\beta}p_{wt,\beta}\mathbb P_\beta[z\connect{\mathbb H_n\:}v]\mathbb P_\beta[t\connect{\mathbb H_n\:}v]
	\\&\leq 2\Big(1+\frac{C}{L}\Big)\frac{C}{L^d}\sum_{v\in \partial \mathbb H_n}\sum_{z,t\in \mathbb H_n}p_{wz,\beta}p_{wt,\beta}\mathbb P_\beta[t\connect{\mathbb H_n\:}v]
	\\&\leq 2\Big(1+\frac{C}{L}\Big)^2\frac{C}{L^d}\Big(\sum_{z\in \mathbb Z^d}p_{0z,\beta}\Big)^2 
	\\&\leq\frac{32C}{L^{d}}\Big(1+\frac{C}{L}\Big)^2\leq \frac{128C^3}{L^d},
\end{align*}
where in the second inequality we used the fact that $z\neq t$ to ensure that one of them is different from $v$, and in the fourth inequality we used \eqref{eq: bound volume times p}.

We now turn to the proof of \eqref{eq:non local term 2}. Let $x\in \partial \mathbb H_n$ and $w\in \mathbb H_n\setminus \{x\}$. Write
\begin{multline}
	E_\beta(\{w\},\mathbb H_n,w,x)=\sum_{z\in \mathbb H_n\setminus \{ w\}}p_{wz,\beta}\mathbb P_\beta[z\connect{\mathbb H_n\:}w]\mathbb P_\beta[w\connect{\mathbb H_n\:}x]\\+\sum_{v\in \mathbb H_n}\sum_{\substack{z,t\in \mathbb H_n\setminus \{w\}\\z\neq t}}p_{wz,\beta}p_{wt,\beta}\mathbb P_\beta[z\connect{\mathbb H_n\:}v]\mathbb P_\beta[t\connect{\mathbb H_n\:}v]\mathbb P_\beta[v\connect{\mathbb H_n\:}x]=:(III)+(IV).
\end{multline}
\paragraph{Bound on $(III)$.} We use the computation done in the bound of $(I)$, the hypothesis that $w\neq x$, and Proposition \ref{prop: final prop proof mainperco thm} to get that
\begin{equation}
	(III)\leq \left(\frac{C}{L^d}\right)^2|\Lambda_L|p_\beta\leq 4\left(\frac{C}{L^d}\right)^2.
\end{equation}
\paragraph{Bound on $(IV)$.} As in $(II)$, the contribution coming from $v\in \mathbb H_n\setminus\{x\}$ in $(IV)$ can be bounded by
\begin{equation}
	\frac{C}{L^d}E_\beta(\{w\})\leq\left(\frac{C}{L^d}\right)^2.
\end{equation}
It remains to observe that
\begin{align}\notag
	\sum_{\substack{z,t\in \mathbb H_n\setminus\{w\}\\z\neq t}}&p_{wz,\beta}p_{wt,\beta}\mathbb P_\beta[z\connect{\mathbb H_n\:}x]\mathbb P_\beta[t\connect{\mathbb H_n\:}x]\\
	&\leq 16\left(\frac{C}{L^d}\right)^2+\sum_{\substack{z,t\in \mathbb H_n\setminus\{w\}\\z\neq t\\ z=x \text{ or }t=x}}p_{wz,\beta}p_{wt,\beta}\mathbb P_\beta[z\connect{\mathbb H_n\:}x]\mathbb P_\beta[t\connect{\mathbb H_n\:}x]\notag\\
	&\leq 16\left(\frac{C}{L^d}\right)^2+2p_\beta |\Lambda_L|p_\beta  \frac{C}{L^d}\leq \frac{C_1}{L^{2d}},
\end{align}
where we used \eqref{eq: bound volume times p}, and where $C_1=C_1(d)>0$. The proof follows readily from setting $D_1:=128C^3\vee C_1$.

\bibliographystyle{alpha}
\bibliography{biblio.bib}

\end{document}